\documentclass[11pt,a4paper]{article}

\usepackage{url,amsmath,amssymb,latexsym,pstricks,mathrsfs,comment,amsthm,graphicx,tikz,tikz-cd,enumerate,accents,pgffor,cite,wrapfig,multicol,float,cases,calc,bibspacing,geometry}
\usepackage[colorlinks]{hyperref}

\usetikzlibrary{calc}

\let\oldproofname=\proofname
\renewcommand{\proofname}{\rm\bf{\oldproofname}}

\renewcommand{\arraystretch}{1.2}

\allowdisplaybreaks

\newcommand{\PT}{\mathcal P\mathcal T}
\newcommand{\PP}{\mathscr{P}\P}
\newcommand{\M}{\mathcal M}
\newcommand{\C}{\mathscr C}
\renewcommand{\O}{\mathcal O}
\newcommand{\I}{\mathcal I}
\renewcommand{\S}{\mathcal S} 
\newcommand{\T}{\mathcal T} 
\newcommand{\A}{\mathcal A} 
\newcommand{\B}{\mathcal B} 
\renewcommand{\P}{\mathcal P} 
\newcommand{\PB}{\mathcal{PB}} 
\newcommand{\J}{\mathcal J}

\renewcommand{\H}{\mathrel{\mathscr H}}
\renewcommand{\L}{\mathrel{\mathscr L}}
\newcommand{\R}{\mathrel{\mathscr R}}
\newcommand{\D}{\mathrel{\mathscr D}}
\newcommand{\gJ}{\mathrel{\mathscr J}}
\newcommand{\K}{\mathrel{\mathscr K}}

\newcommand{\bq}{{\bf q}}
\newcommand{\bn}{\mathbf{n}}
\newcommand{\bm}{\mathbf{m}}
\newcommand{\bP}{{\bf P}}
\newcommand{\bQ}{{\bf Q}}

\newcommand{\om}{\omega}
\newcommand{\al}{\alpha}
\newcommand{\be}{\beta}
\newcommand{\ga}{\gamma}
\newcommand{\de}{\delta}
\newcommand{\ve}{\varepsilon}
\newcommand{\lam}{\lambda}
\newcommand{\si}{\sigma}

\newcommand{\De}{\Delta}
\newcommand{\Ga}{\Gamma}
\newcommand{\Om}{\Omega}

\newcommand{\alh}{\widehat\al}
\newcommand{\beh}{\widehat\be}
\newcommand{\gah}{\widehat\ga}
\newcommand{\deh}{\widehat\de}
\newcommand{\tauh}{\widehat\tau}

\newcommand{\alt}{\widetilde{\al}}
\newcommand{\bet}{\widetilde{\be}}
\newcommand{\xit}{\widetilde{\xi}}

\newcommand{\xib}{\overline{\xi}}
\newcommand{\pib}{\overline{\pi}}
\newcommand{\sib}{\overline{\si}}
\newcommand{\taub}{\overline{\tau}}
\newcommand{\Deb}{\overline{\De}}
\newcommand{\Rb}{\overline{R}}
\newcommand{\idb}{\overline{\id}}

\newcommand{\Proj}{\operatorname{Proj}}
\newcommand{\Cong}{\operatorname{Cong}}
\newcommand{\coker}{\operatorname{coker}}
\newcommand{\dom}{\operatorname{dom}} 
\newcommand{\codom}{\operatorname{codom}}
\newcommand{\rank}{\operatorname{rank}}
\newcommand{\id}{\operatorname{id}}

\newcommand{\set}[2]{\{ {#1} : {#2} \}} 
\newcommand{\bigset}[2]{\big\{ {#1}: {#2} \big\}} 
\newcommand{\normal}[1]{\la\!\la#1\ra\!\ra}
\newcommand{\cg}[2]{(#1,#2)^\sharp}
\newcommand{\cgI}[2]{(#1,#2)^\sharp_{\I_n}}
\newcommand{\cgset}[4]{\big\{(#1,#2),(#3,#4)\big\}^\sharp}

\newcommand{\Rect}{\mathfrak{R}}
\newcommand{\Rtwo}{\mathbb R^2}
\renewcommand{\emptyset}{\varnothing}
\newcommand{\1}{\id_n}
\newcommand{\emptypartperm}{[\emptyset]}
\newcommand{\lar}[1]{ \xrightarrow {\ #1\ }}
\newcommand{\mt}{\mapsto}
\newcommand{\sm}{\setminus}
\newcommand{\sub}{\subseteq}
\newcommand{\suq}{\subsetneq}
\newcommand{\la}{\langle}
\newcommand{\ra}{\rangle}

\newcommand{\OR}{\qquad\text{or}\qquad}
\newcommand{\COMMA}{,\quad}
\newcommand{\AND}{\qquad\text{and}\qquad}
\newcommand{\ANDSIM}{\qquad\text{and similarly}\qquad}
\newcommand{\ANd}{\quad\text{and}\quad}
\renewcommand{\iff}{\ \Leftrightarrow\ }
\renewcommand{\implies}{\ \Rightarrow\ }

\newcommand{\bit}{\begin{itemize}}
\newcommand{\eit}{\end{itemize}}
\newcommand{\ben}{\begin{enumerate}}
\newcommand{\een}{\end{enumerate}}
\newcommand{\itemit}[1]{\item[\emph{(#1)}]}
\newcommand{\pf}{\begin{proof}}
\newcommand{\epf}{\end{proof}}
\newcommand{\epfres}{\hfill\qed}
\newcommand{\epfreseq}{\tag*{\qed}}

\renewcommand{\c}{@{}c@{}}
\newcommand{\cend}{@{}c@{\hspace{1.5truemm}}}
\newcommand{\cstart}{@{\hspace{1.5truemm}}c@{}}

\newcommand{\partI}[8]{
\Big( 
{ \scriptsize \renewcommand*{\arraystretch}{1}
\begin{array} {\c|\c|\c|\c|\c|\cend}
 #1 \:&\: \cdots \:&\: #2 \:&\: #3 \:&\: \cdots \:&\: #4 \\ \cline{4-6}
 #5 \:&\: \cdots \:&\: #6 \:&\: #7 \:&\: \cdots \:&\: #8 
\rule[0mm]{0mm}{2.7mm}
\end{array} 
}
\hspace{-1.5 truemm} \Big)
}

\newcommand{\partABCD}{\partI{A_1}{A_q}{C_1}{C_r}{B_1}{B_q}{D_1}{D_s}}

\newcommand{\partII}[6]{
\Big(  \hspace{-1.5 truemm}
{\scriptsize \renewcommand*{\arraystretch}{1} \begin{array} {\cstart|\c|\cend}
#1 \:&\: #2 \:&\: #3  \\ \cline{1-3}
#4 \:&\: #5 \:&\: #6 
\rule[0mm]{0mm}{2.7mm}
\end{array}  }
\hspace{-1.5 truemm} \Big) 
}

\newcommand{\partIII}[8]{
\Big(  \hspace{-1.5 truemm}
{\scriptsize \renewcommand*{\arraystretch}{1} \begin{array} {\cstart|\c|\c|\cend}
#1 \:&\: #2 \:&\: #3 \:&\: #4 \\ \cline{1-4}
#5 \:&\: #6 \:&\: #7 \:&\: #8
\rule[0mm]{0mm}{2.7mm}
\end{array}  }
\hspace{-1.5 truemm} \Big) 
}

\newcommand{\partIV}[8]{
\Big(  
{\scriptsize \renewcommand*{\arraystretch}{1} \begin{array} {\c|\c|\c|\cend}
#1 \:&\: #2 \:&\: #3 \:&\: #4 \\ \cline{2-4}
#5 \:&\: #6 \:&\: #7 \:&\: #8
\rule[0mm]{0mm}{2.7mm}
\end{array}  }
\hspace{-1.5 truemm} \Big) 
}

\newcommand{\partV}[8]{
\Big(  
{\scriptsize \renewcommand*{\arraystretch}{1} \begin{array} {\c|\c|\c|\c|\cend}
#1 \:&\: #2 \:&\: #3 \:&\: \cdots \:&\: #4 \\ \cline{3-5}
#5 \:&\: #6 \:&\: #7 \:&\: \cdots \:&\: #8
\rule[0mm]{0mm}{2.7mm}
\end{array}  }
\hspace{-1.5 truemm} \Big) 
}

\newcommand{\partVI}[8]{
\Big(  
{\scriptsize \renewcommand*{\arraystretch}{1} \begin{array} {\c|\c|\c|\c|\cend}
#1 \:&\: #2 \:&\: #3 \:&\: \cdots \:&\: #4 \\ \cline{2-5}
#5 \:&\: #6 \:&\: #7 \:&\: \cdots \:&\: #8
\rule[0mm]{0mm}{2.7mm}
\end{array}  }
\hspace{-1.5 truemm} \Big) 
}

\newcommand{\partVII}[8]{
\Big(  \hspace{-1.5 truemm}
{\scriptsize \renewcommand*{\arraystretch}{1} \begin{array} {\cstart|\c|\c|\c|\cend}
#1 \:&\: #2 \:&\: #3 \:&\: \cdots \:&\: #4 \\ \cline{1-5}
#5 \:&\: #6 \:&\: #7 \:&\: \cdots \:&\: #8
\rule[0mm]{0mm}{2.7mm}
\end{array}  }
\hspace{-1.5 truemm} \Big) 
}

\newcommand{\partXVII}[6]{
\Bigl( 
{ \scriptsize\renewcommand*{\arraystretch}{1} \begin{array} {\c|\c|\cend}
#1\: &\: #2\: &\: #3 \\ \cline{2-3}
#4\: &\: #5\: &\: #6
\rule[0mm]{0mm}{2.7mm}
\end{array}  }
\hspace{-1.5 truemm} \Bigr)
}

\newcommand{\partXIX}[8]{
\Big(  \hspace{-1.5 truemm}
{\scriptsize \renewcommand*{\arraystretch}{1} \begin{array} {\cstart|\c|\c|\c|\cend}
#1 \:&\: #2 \:&\: #3 \:&\: \cdots \:&\: #4 \\ \cline{1-5}
#5 \:&\: #6 \:&\: #7 \:&\: \cdots \:&\: #8
\rule[0mm]{0mm}{2.7mm}
\end{array}  }
\hspace{-1.5 truemm} \Big) 
}

\newcommand{\partpermI}[2]{
\Big[
{ \scriptsize \renewcommand*{\arraystretch}{1}
\begin{array} {\c}
 #1  \\ 
 #2
\rule[0mm]{0mm}{2.7mm}
\end{array} 
}
\Big]
}

\newcommand{\partpermII}[4]{
\Big[
{ \scriptsize \renewcommand*{\arraystretch}{1}
\begin{array} {\c|\c}
 #1 \:&\: #2 \\ 
 #3 \:&\: #4
\rule[0mm]{0mm}{2.7mm}
\end{array} 
}
\Big]
}

\newcommand{\partpermIII}[6]{
\Big[
{ \scriptsize \renewcommand*{\arraystretch}{1}
\begin{array} {\c|\c|\c}
 #1 \:&\: #2 \:&\: #3 \\ 
 #4 \:&\: #5 \:&\: #6 
\rule[0mm]{0mm}{2.7mm}
\end{array} 
}
\Big]
}

\newcommand{\uuv}[1]{\fill (#1,4)circle(.17);}
\newcommand{\uv}[1]{\fill (#1,2)circle(.17);}
\newcommand{\lv}[1]{\fill (#1,0)circle(.17);}
\newcommand{\uvs}[1]{{\foreach \x in {#1} { \uv{\x}}}}
\newcommand{\lvs}[1]{{\foreach \x in {#1} { \lv{\x}}}}
\newcommand{\uvw}[1]{\draw[fill=white] (#1,2)circle(.17);}
\newcommand{\lvw}[1]{\draw[fill=white] (#1,0)circle(.17);}
\newcommand{\uvert}[1]{\fill (#1,2)circle(.2);}
\renewcommand{\lvert}[1]{\fill (#1,0)circle(.2);}

\newcommand{\stline}[2]{\draw(#1,2)--(#2,0);}
\newcommand{\stlines}[1]{{\foreach \x/\y in {#1} { \stline{\x}{\y} }}}

\newcommand{\darcx}[3]{\draw(#1,0)arc(180:90:#3) (#1+#3,#3)--(#2-#3,#3) (#2-#3,#3) arc(90:0:#3);}
\newcommand{\darc}[2]{\darcx{#1}{#2}{.4}}
\newcommand{\uarcx}[3]{\draw(#1,2)arc(180:270:#3) (#1+#3,2-#3)--(#2-#3,2-#3) (#2-#3,2-#3) arc(270:360:#3);}
\newcommand{\uarc}[2]{\uarcx{#1}{#2}{.4}}
\newcommand{\darcxx}[4]{\draw[#4](#1,0)arc(180:90:#3) (#1+#3,#3)--(#2-#3,#3) (#2-#3,#3) arc(90:0:#3);}
\newcommand{\uarcxx}[4]{\draw[#4](#1,2)arc(180:270:#3) (#1+#3,2-#3)--(#2-#3,2-#3) (#2-#3,2-#3) arc(270:360:#3);}
\newcommand{\stlinex}[3]{\draw[#3](#1,2)--(#2,0);}

\newcommand{\custpartn}[3]{{\lower1.4 ex\hbox{
\begin{tikzpicture}[scale=.3]
\foreach \x in {#1}
{ \uvert{\x}  }
\foreach \x in {#2}
{ \lvert{\x}  }
#3 \end{tikzpicture}
}}}

\newcommand{\dottedinterval}[4]{\draw(#1,#2)--(#1,#2+#3) (#1,#2+#3+#4)--(#1,#2+#3+#4+#3);\draw[dotted](#1,#2+#3)--(#1,#2+#3+#4);}

\newcommand{\uu}{0}\newcommand{\vv}{0}\newcommand{\xx}{0}\newcommand{\yy}{0} \newcommand{\dd}{11}

\newcommand{\uuvlab}[2]{\uuv{#1}\draw(#1-.1,4.8)node[below right]{{\tiny $#2$}};}
\newcommand{\uvlab}[2]{\uv{#1}\draw(#1-.1,2.8)node[below right]{{\tiny $#2$}};}
\newcommand{\uvlabw}[2]{\uvw{#1}\draw(#1-.1,2.8)node[below right]{{\tiny $#2$}};}
\newcommand{\lvlab}[2]{\lv{#1}\draw(#1+.2,-1)node[above left]{{\tiny $#2$}};}
\newcommand{\lvlabw}[2]{\lvw{#1}\draw(#1+.2,-1)node[above left]{{\tiny $#2$}};}

\newcommand{\shiftvalue}{18}

\newcommand{\defkkkkba}[4]{\newcommand{\kk}{#1}\newcommand{\kkkk}{#2}\newcommand{\bbb}{#3}\newcommand{\aaa}{#4}}
\defkkkkba{.5}26{12}

\newcommand{\defbzn}[3]{\newcommand{\bbbb}{#1}\newcommand{\zzzz}{#2}\newcommand{\nnnn}{#3}}
\defbzn{2.5}5{12}

\newcommand{\uvxyorder}[4]{\renewcommand{#1}{1}\renewcommand{#2}{3.5}\renewcommand{#3}{6.5}\renewcommand{#4}{9.25}}

\newcommand{\vertices}{
\uvlab{\vv}v
\uvlabw{\uu}u
\uvlab{\xx}x
\uvlabw{\yy}y
\uvlab{\dd}d
\draw(\dd,2)--(\dd-1,4);
\uuvlab{\dd-1}c
\lvlab{\bbbb}b
\lvlabw{\zzzz}w
\lvlab{\nnnn}n
\uvlab{\vv+\shiftvalue}v
\uvlabw{\uu+\shiftvalue}u
\uvlab{\xx+\shiftvalue}x
\uvlabw{\yy+\shiftvalue}y
\uuvlab{\dd-1+\shiftvalue}c
\draw(\dd+\shiftvalue,2)--(\dd-1+\shiftvalue,4);
\uvlab{\dd+\shiftvalue}d
\lvlab{\bbbb+\shiftvalue}b
\lvlabw{\zzzz+\shiftvalue}w
\lvlab{\nnnn+\shiftvalue}n
}

\newcommand{\udarcx}[3]{\draw(#1,2)--(#1,2.5) arc(180:90:#3) (#1+#3,2.5+#3)--(#2-#3,2.5+#3) (#2-#3,2.5+#3) arc(90:0:#3) --(#2,2);}

\newcommand{\linestau}[4]{
\stlines{#1/\bbbb,#2/\zzzz,#3/\nnnn}
\stlines{#1+\shiftvalue/\bbbb+\shiftvalue,#2+\shiftvalue/\zzzz+\shiftvalue, #3+\shiftvalue/\nnnn+\shiftvalue}
}

\newcommand{\shiftup}{6}

\newcommand{\alphabetataudiagram}[9]{
\begin{scope}[shift={(0,#1)}]
\uvxyorder#2
\newcommand{\sillyone}{#3}
\newcommand{\sillytwo}{#4}
\udarcx{\xx}{\yy}{.5}
\linestau{#5}{#6}{#7}{#8}
#9
\vertices
\draw[|-|](-2,0)--(-2,2);
\draw[|-](-2,4)--(-2,2);
\draw(-1.8,1)node[left]{$\scriptstyle{\tau}$};
\draw(-1.8,3)node[left]{$\scriptstyle{\al}$};
\draw[|-|](-2+\shiftvalue,0)--(-2+\shiftvalue,2);
\draw[|-](-2+\shiftvalue,4)--(-2+\shiftvalue,2);
\draw(-1.8+\shiftvalue,1)node[left]{$\scriptstyle{\tau}$};
\draw(-1.8+\shiftvalue,3)node[left]{$\scriptstyle{\be}$};
\end{scope}
}

\newcommand{\caselabel}[1]{\draw(-5,2)node{(#1)};} 

\newcommand{\arcdl}[1]{arc(360:180:#1/2)}
\newcommand{\arcdr}[1]{arc(180:360:#1/2)}
\newcommand{\arcul}[1]{arc(0:180:#1/2)}
\newcommand{\arcur}[1]{arc(180:0:#1/2)}

\newcommand{\mvw}[1]{\draw[fill=white] (#1,10)circle(.2);}
\newcommand{\mv}[1]{\fill (#1,10)circle(.2);}
\newcommand{\hmv}[1]{\fill (#1,20)circle(.2);}
\newcommand{\lmv}[1]{\fill (#1,0)circle(.2);}

\numberwithin{equation}{section}

\newtheorem{thm}[equation]{Theorem}
\newtheorem{lemma}[equation]{Lemma}
\newtheorem{cor}[equation]{Corollary}
\newtheorem{prop}[equation]{Proposition}

\theoremstyle{definition}

\newtheorem{rem}[equation]{Remark}
\newtheorem{defn}[equation]{Definition}

\begin{document}

\title{Congruence lattices of finite diagram monoids\vspace{-5ex}}
\author{}
\date{}

\renewcommand{\thefootnote}{\fnsymbol{footnote}}

\maketitle
\begin{center}
{\large 
James East,%
\hspace{-.3em}\footnote{Centre for Research in Mathematics, School of Computing, Engineering and Mathematics, Western Sydney University, Locked Bag 1797, Penrith NSW 2751, Australia. {\it Email:} {\tt j.east\,@\,westernsydney.edu.au}}
James D.~Mitchell,%
\hspace{-.2em}\footnote{\label{footnote:JDM}Mathematical Institute, School of Mathematics and Statistics, University of St Andrews, St Andrews, Fife KY16 9SS, UK. {\it Emails:} {\tt jdm3\,@\,st-andrews.ac.uk}, \ {\tt nik.ruskuc\,@\,st-andrews.ac.uk}, \ {\tt mct25\,@\,st-andrews.ac.uk}}
Nik Ru\v skuc,\hspace{-.2em}\textsuperscript{\ref{footnote:JDM}}
Michael Torpey\textsuperscript{\ref{footnote:JDM}}
}
\end{center}

\setcounter{footnote}{0}

\renewcommand{\thefootnote}{\arabic{footnote}}

\vspace{-0.5cm}

\begin{abstract}
We give a complete description of the congruence lattices of the following finite diagram monoids: the partition monoid, the planar partition monoid, the Brauer monoid, the Jones monoid (also known as the Temperley-Lieb monoid), the Motzkin monoid, and the partial Brauer monoid.  All the congruences under discussion arise as special instances of a new construction, involving an ideal $I$, a retraction $I\rightarrow M$ onto the minimal ideal, a congruence on $M$, and a normal subgroup of a maximal subgroup outside $I$.


\end{abstract}

%\tableofcontents

\section{Introduction}\label{sect:intro}
A {congruence} on a semigroup $S$ is an equivalence relation that is
compatible with the multiplicative structure of $S$.  The role played by
congruences in semigroup theory (and general algebra) is analogous to that of
normal subgroups in group theory, and ideals in ring theory: they are precisely the kernels of
homomorphisms, and they govern the formation of quotients.  The set $\Cong(S)$
of all congruences of a semigroup $S$ forms an (algebraic) lattice under
inclusion, known as the {congruence lattice} of $S$.  

The study of
congruences has always been one of the corner-stones of semigroup theory, and a
major strand in this direction has been the description of the congruence
lattices of specific semigroups or families of semigroups.
In his influential~1952 article~\cite{Malcev1952}, Mal$'$cev described the
congruences on the full transformation  monoid $\T_n$.  Analogues for other
classical monoids followed: the monoid of $n\times n$ matrices over a
field (Mal$'$cev \cite{Malcev1953}), the symmetric inverse monoid $\I_n$ (Liber
\cite{Liber1953}), the partial transformation semigroup $\PT_n$ (Sutov \cite{Sutov1961}), and many
others subsequently; see for example \cite{Aizenstat1962,Fernandes2001,MSS2000,Sutov1961_2}.
A contemporary account of these results for $\T_n$, $\PT_n$,
and $\I_n$ can be found in \cite[Section 6.3]{GMbook}.  It turns out that
in each of these cases, the congruence lattice is a chain whose length is a
linear function of $n$.  An even more recent work is the paper
\cite{ABG2016} by Ara\'{u}jo, Bentz and Gomes, describing the congruences in
various direct products of transformation monoids.  

In the current article, we undertake a study of congruence lattices of
{diagram monoids}.  These monoids arise naturally in the study of
{diagram algebras}, a class of algebras with origins in theoretical
physics and representation theory. Key examples are the
{Temperley--Lieb algebras} \cite{TL1971,Martin1994,BDP2002}, 
{Brauer algebras}~\cite{Brauer1937,MarMaz2014},
{partition algebras} \cite{Martin1994,Jones1994_2,HR2005}
and {Motzkin algebras} \cite{BH2014}.  
These diagram algebras are defined by means of diagrammatic basis elements, and
are all {twisted semigroup algebras} \cite{Wilcox2007} of a corresponding
diagram monoid, such as the {Jones},
{Brauer}, {partition} or {Motzkin monoid}; see
Section~\ref{sect:prelim} for the definitions of these monoids and others. 

There are many important connections between diagram and transformation
monoids.  For one thing, the partition monoid $\P_n$ contains copies of the
full transformation monoid~$\T_n$ and the symmetric inverse monoid~$\I_n$.  In
addition, many of the semigroup-theoretic properties of transformation monoids
also hold for diagram monoids.  For example, in each of the above-mentioned
diagram monoids, the ideals form a chain with respect to containment.
Furthermore, the idempotent-generated subsemigroup coincides with the singular
part of the Brauer and partition monoids \cite{JEpnsn,MM2007}, and the proper
ideals of the Jones, Brauer and partition monoids are idempotent-generated
\cite{EG2017}.
When it comes to congruences, however, it turns out that the parallels are
simultaneously less tight and more subtle.  

This paper arose as a consequence
of some initial computational experiments with the Semigroups package
for the computer algebra system GAP \cite{GAP}.  Using newly developed
algorithms for working with congruences, we were able to compute the congruence
lattices of several diagram monoids, including the partition, Brauer and Jones
monoids of relatively low degree.  The results of these computations were
surprising at the time.  While the congruence lattices of the transformation
monoids $\T_n$, $\PT_n$ and $\I_n$ are all finite chains, the congruence
lattices of the diagram monoids have a richer structure, and contain a number
of congruences identifying partitions of low {rank} in non-trivial ways.
The computational experiments allowed us to develop a number of conjectures, and we originally proved these with a series of separate arguments for each of the monoids.  Deeper analysis of these arguments led us to develop the theoretical framework presented in Section~\ref{sec-fam}, which underpins all of the congruence lattices studied here, ultimately leading to the present article.

The exposition is organised as follows.  In Section~\ref{sect:prelim}, we give
the definitions and background material we require.  
In Section~\ref{sec-fam}, we present a number of constructions that build congruences on a finite semigroup from ideals, retractions, maximal subgroups and congruences on the minimal ideal.  It turns out that all of the congruences of the diagram and transformation monoids above are special intstances of this construction.  
In
Section~\ref{sect:SnInOn}, we present the known classifications of congruences
on the symmetric inverse monoid $\I_n$ (Liber \cite{Liber1953}) and the monoid
$\O_n$ of order-preserving partial permutations of an $n$-element chain (Fernandes
\cite{Fernandes2001}), but couched in the framework established in
Section~\ref{sec-fam}.  The rest of the paper is devoted to the
presentation of the main results, treating each of the diagram monoids in turn.
Section~\ref{sect:Pn} concerns the partition monoid~$\P_n$, and serves as
a case study in applying our construction.  In Section~\ref{sect:PBn}, we give a brief treatment of the case of the partial
Brauer monoid $\PB_n$, which turns out to be a nearly verbatim re-run of the
argument for $\P_n$.  Section~\ref{sect:PPnMn} covers the planar partition
monoid $\PP_n$ and the Motzkin monoid $\M_n$ in parallel, and then in 
Sections~\ref{sect:Bn} and~\ref{sect:Jn}, we consider the Brauer and Jones
monoids $\B_n$ and~$\J_n$, respectively.  
The main results are stated in Theorems \ref{thm-CongPn}, \ref{thm-CongPBn}, \ref{thm-CongPPn}, \ref{CongBn} and \ref{thm:cong-Jn}.  Depictions of the various congruence lattices may be found in Figures \ref{fig-CongPn}, \ref{fig-CongPPn}, \ref{fig:Hasse_Bn} and \ref{fig:Hasse_Jn}.

\section{Preliminaries}\label{sect:prelim}

In this section we introduce the partition monoids $\P_n$
(Subsection~\ref{sect:prelim_Pn}) and a number of distinguished submonoids
of~$\P_n$ (Subsection~\ref{sect:prelim_other}), and then discuss their Green's
relations and ideals (Subsection~\ref{sect:prelim_Green}).  We conclude the
section with a brief general discussion of congruences and their lattices in
general (Subsection~\ref{sect:prelim_congruences}).

\subsection{Partition monoids}\label{sect:prelim_Pn}

Let $n$ be a positive integer.  Throughout the article, we write
$\bn=\{1,\ldots,n\}$ and $\bn'=\{1',\ldots,n'\}$.  The \emph{partition monoid
of degree $n$}, denoted $\P_n$, is the monoid of all set partitions of
$\bn\cup\bn'$ under a product described below.  That is, an element of $\P_n$
is a set $\al=\{A_1,\ldots,A_k\}$, for some $k$, where the $A_i$ are pairwise
disjoint nonempty subsets of $\bn\cup\bn'$ whose union is all of $\bn\cup\bn'$;
the $A_i$ are called the \emph{blocks} of $\al$.  

A partition $\al\in\P_n$ may be represented as any graph with vertex set
$\bn\cup\bn'$ with edges so that the connected components of the
graph correspond to the blocks of the partition; such a graph is drawn with
vertices $1,\ldots,n$ on an upper row (increasing from left to right), with
vertices $1',\ldots,n'$ directly below, and with all edges within the rectangle
determined by the vertices.   For example, the partitions
\[
\al = \big\{ \{1,4\},\{2,3,4',5'\},\{5,6\},\{1',2',6'\},\{3'\}\big\} \ANd
\be = \big\{ \{1,2\}, \{3,4,1'\}, \{5,4',5',6'\}, \{6\}, \{2'\}, \{3'\} \big\}
\]
from $\P_6$ are pictured in Figure~\ref{fig:P6}.  As usual, we will identify a
partition with any graph representing~it.

The product of two partitions $\al,\be\in\P_n$ is defined as follows.  Write
$\bn''=\{1'',\ldots,n''\}$.  Let $\al^\vee$ be the graph obtained from $\al$ by
changing the label of each lower vertex $i'$ to~$i''$, and let $\be^\wedge$ be
the graph obtained from $\be$ by changing the label of each upper vertex~$i$
to~$i''$.  Consider now the graph $\Pi(\al,\be)$ on the vertex set~$\bn\cup
\bn'\cup \bn''$ obtained by joining $\al^\vee$ and~$\be^\wedge$ together so
that each lower vertex $i''$ of $\al^\vee$ is identified with the corresponding
upper vertex $i''$ of $\be^\wedge$.  We call $\Pi(\al,\be)$ the \emph{product
graph} of $\al$ and $\be$.  We define $\al\be\in\P_n$ to be the partition
satisfying the property that $x,y\in\bn\cup\bn'$ belong to the same block of
$\al\be$ if and only if $x$ and~$y$ are connected by a path in $\Pi(\al,\be)$.
This process is illustrated in Figure~\ref{fig:P6}.  
The operation is associative and the partition
$\1=\big\{\{1,1'\},\ldots,\{n,n'\}\big\}$ is the identity element, so $\P_n$ is
a monoid. 
It is worth noting that even though elements of $\P_n$ naturally correspond to binary relations (indeed, equivalences) on a set of size $2n$, the product on $\P_n$ does not correspond to composition of binary relations.

\begin{figure}[ht]
\begin{center}
\begin{tikzpicture}[scale=.5]
\begin{scope}[shift={(0,0)}]	
\uvs{1,...,6}
\lvs{1,...,6}
\uarcx14{.6}
\uarcx23{.3}
\uarcx56{.3}
\darc12
\darcx26{.6}
\darcx45{.3}
\stline34
\draw(0.6,1)node[left]{$\al=$};
\draw[->](7.5,-1)--(9.5,-1);
\end{scope}
\begin{scope}[shift={(0,-4)}]	
\uvs{1,...,6}
\lvs{1,...,6}
\uarc12
\uarc34
\darc45
\darc56
\stline31
\stline55
\draw(0.6,1)node[left]{$\be=$};
\end{scope}
\begin{scope}[shift={(10,-1)}]	
\uvs{1,...,6}
\lvs{1,...,6}
\uarcx14{.6}
\uarcx23{.3}
\uarcx56{.3}
\darc12
\darcx26{.6}
\darcx45{.3}
\stline34
\draw[->](7.5,0)--(9.5,0);
\end{scope}
\begin{scope}[shift={(10,-3)}]	
\uvs{1,...,6}
\lvs{1,...,6}
\uarc12
\uarc34
\darc45
\darc56
\stline31
\stline55
\end{scope}
\begin{scope}[shift={(20,-2)}]	
\uvs{1,...,6}
\lvs{1,...,6}
\uarcx14{.6}
\uarcx23{.3}
\uarcx56{.3}
\darc14
\darc45
\darc56
\stline21
\draw(6.4,1)node[right]{$=\al\be$};
\end{scope}
\end{tikzpicture}
\end{center}
\vspace{-5mm}
\caption{Two partitions $\al,\be\in\P_6$ (left), the product graph
$\Pi(\al,\be)$ (middle), and their product $\al\be\in\P_6$ (right).}
\label{fig:P6}
\end{figure}
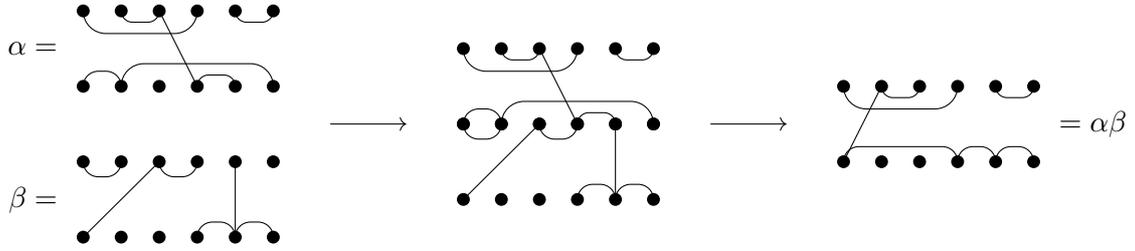

A block $A$ of a partition is referred to as a \emph{transversal} if
$A\cap\bn\not=\emptyset$ and $A\cap\bn'\not=\emptyset$, or a
\emph{non-transversal} otherwise.  
If $\al\in\P_n$, we will write
\[
\al=\partABCD
\]
to indicate that $\al$ has transversals $A_i\cup B_i'$ ($1\leq i\leq q$), upper non-transversals $C_j$ ($1\leq j\leq r$) and lower non-transversals $D_k'$ ($1\leq k\leq s$).  Note here that all of the sets $A_i,B_i,C_j,D_k$ are contained in $\bn$.
For example, $\al\in\P_6$ defined above has $\{2,3,4',5'\}$ as its only
transversal, and has upper non-transversals $\{1,4\}$ and
$\{5,6\}$, and lower non-transversals $\{1',2',6'\}$ and $\{3'\}$,
so $\al = \partXVII{2,3}{1,4}{5,6}{4,5}{1,2,6}{3}$.
Note that in the expression $\al=\partABCD$, the parameters $r$ and $s$ need
not be equal, and that any (but not all) of $q,r,s$ can be $0$.

The (\emph{co})\emph{domain}, (\emph{co})\emph{kernel} and \emph{rank} of a partition $\al=\partABCD\in\P_n$ are defined as follows:~
\bit
\item
$\dom(\al) =A_1\cup\dots\cup A_q= \set{i\in\bn}{\text{$i$ belongs to a transversal of $\al$}}$,
\item
$\codom(\al) =B_1\cup\dots\cup B_q =\set{i\in\bn}{\text{$i'$ belongs to a transversal of $\al$}}$,
\item
$\ker(\al) = \set{(i,j)\in\bn\times\bn}{\text{$i$ and $j$ belong to the same block of $\al$}}$,
the equivalence relation on $\bn$ with equivalence classes $A_1,\dots, A_q,C_1, \dots, C_r$,
\item
$\coker(\al) = \set{(i,j)\in\bn\times\bn}{\text{$i'$ and $j'$ belong to the same block of $\al$}}$,
the equivalence relation on $\bn$ with equivalence classes $B_1,\dots, B_q,D_1, \dots, D_s$,
\item
$\rank(\al)=q$, the number of transversals of $\alpha$.
\eit
For example, when $\al = \partXVII{2,3}{1,4}{5,6}{4,5}{1,2,6}{3}$, we have
$\rank(\al)=1$, $\dom(\al)=\{2,3\}$, $\codom(\al)=\{4,5\}$, the $\ker(\al)$-classes are $\{1,4\},\{2,3\},\{5,6\}$, and the $\coker(\al)$-classes are $\{1,2,6\},\{3\},\{4,5\}$.

One may easily check that for any $\al,\be,\ga\in\P_n$,
\begin{gather*}
\dom(\al\be)\sub\dom(\al) \COMMA \codom(\al\be)\sub\codom(\be) , \\
\ker(\al\be)\supseteq\ker(\al) \COMMA \coker(\al\be)\supseteq\coker(\be) \COMMA
\rank(\al\be\ga)\leq\rank(\be).
\end{gather*}
The group of units of $\P_n$ is the set $\set{\al\in\P_n}{\rank(\al)=n}$; this
group is isomorphic to (and will be identified with) the symmetric group $\S_n$.

The partition monoid
admits a natural anti-involution ${}^*:\P_n\to\P_n$ defined by 
\[
\partI{A_1}{A_q}{C_1}{C_r}{B_1}{B_q}{D_1}{D_s}^*=\partI{B_1}{B_q}{D_1}{D_s}{A_1}{A_q}{C_1}{C_r}.
\]
Roughly speaking, if $\al\in\P_n$, then $\al^*$ is obtained by reflecting (a graph representing)~$\al$ in the horizontal axis midway between the two rows of vertices.  It is easy to see that
\[
\al^{**}=\al \COMMA (\al\be)^*=\be^*\al^* \COMMA \al\al^*\al=\al,
\]
for all $\al,\be\in\P_n$.  It follows that $\P_n$ is a \emph{regular $*$-semigroup}, in the sense of Nordahl and Scheiblich \cite{NS1978}, with respect to this operation.  
We have several obvious identities, such as $\dom(\al^*)=\codom(\al)$ and $\ker(\al^*)=\coker(\al)$. 
This symmetry/duality will allow us to shorten several proofs.

\subsection{Other diagram monoids}\label{sect:prelim_other}

In this subsection, we introduce a number of important submonoids of the partition monoid $\P_n$.
Following~\cite{Maz1998},
the \emph{Brauer} and \emph{partial Brauer monoid} are defined by
\[
\B_n = \set{\al\in\P_n}{\text{all blocks of $\al$ have size $2$}} \ANd
\PB_n = \set{\al\in\P_n}{\text{all blocks of $\al$ have size at most $2$}},
\]
respectively.  Note that $\S_n\suq\B_n\suq\PB_n\suq\P_n$.  Note also that an
element of $\PB_n$ has a \emph{unique} representation as a 
graph with no loops or multiple edges and with vertex set $\bn\cup\bn'$.

Following \cite{Jones1994_2}, we say that a partition $\al\in\P_n$ is
\emph{planar} if it has a graphical representation where 
the edges are drawn within the rectangle spanned by the vertices and do not intersect.
For example, with $\al,\be\in\P_6$ as in
Figure \ref{fig:P6}, $\be$ is planar but $\al$ is not.  The set
$\mathscr{P}\P_n$ of all planar partitions is clearly a submonoid of $\P_n$.
The \emph{Jones} and \emph{Motzkin monoids} are defined by
\[
\J_n = \B_n \cap \PP_n \AND \M_n=\PB_n\cap\PP_n;
\]
see \cite{HLP2013,LF2006}.  It is well known \cite{HR2005,Jones1994_2} that $\PP_n$ is isomorphic to $\J_{2n}$; see for example \cite[p873]{HR2005}.  We will have more to say about this isomorphism in Section \ref{sect:Jn}, where it will play a crucial role in our analysis of the even degree Jones monoids $\J_{2n}$.  

The partition monoid $\P_n$ contains a number of other submonoids arising from the theory of \emph{transformation monoids}.  
Of particular importance to us here are:
\bit
\item
$\I_n = \set{\al\in\PB_n}{\text{every non-transversal of $\al$ is a singleton}}$,
which is isomorphic to (and will be identified with) the \emph{symmetric inverse monoid},
consisting of all partial permutations of $\bn$;
\item
$\O_n=\I_n\cap\PP_n=\I_n\cap\M_n$, the monoid of all order-preserving partial permutations of $\bn$.
\eit
It is clear that $\I_n\cap\B_n=\S_n$.
The submonoids of $\P_n$ defined above---$\P_n$, $\PB_n$, $\PP_n$, $\B_n$,
$\I_n$, $\M_n$, $\S_n$, $\J_n$ and~$\O_n$---are pictured in
Figure~\ref{fig:submonoids}, which displays their inclusions, intersections,
and representative elements.  Each of these monoids is closed under the
operation ${}^*$ defined above.

The full transformation monoid $\T_n$ is also naturally contained as a submonoid of $\P_n$, as
the set of all $\al\in\P_n$ with $\dom(\al)=\bn$ and $\coker(\al)$ equal to the
trivial relation on $\bn$. However $\T_n$ does not play a role in
this paper, and so it is not included in Figure~\ref{fig:submonoids}.  We also note that the partial transformation monoid~$\PT_n$ does not (naturally) embed as a submonoid of $\P_n$; see \cite[Section 3.2]{JEgrpm}.

It will be convenient to have a special notation for the elements of $\I_n$.  The unique element $\al\in\I_n$ with transversals $\{a_1,b_1'\},\ldots,\{a_q,b_q'\}$ will be denoted by $\al=\partpermIII{a_1}\cdots{a_q}{b_1}\cdots{b_q}$.  With this notation, we have $\al^{-1}=\al^*=\partpermIII{b_1}\cdots{b_q}{a_1}\cdots{a_q}$.  As a special case, when $\al=\big\{\{1\},\ldots,\{n\},\{1'\},\ldots,\{n'\}\big\}$ is the unique element of~$\I_n$ with no transversals, we write $\al=\emptypartperm$.

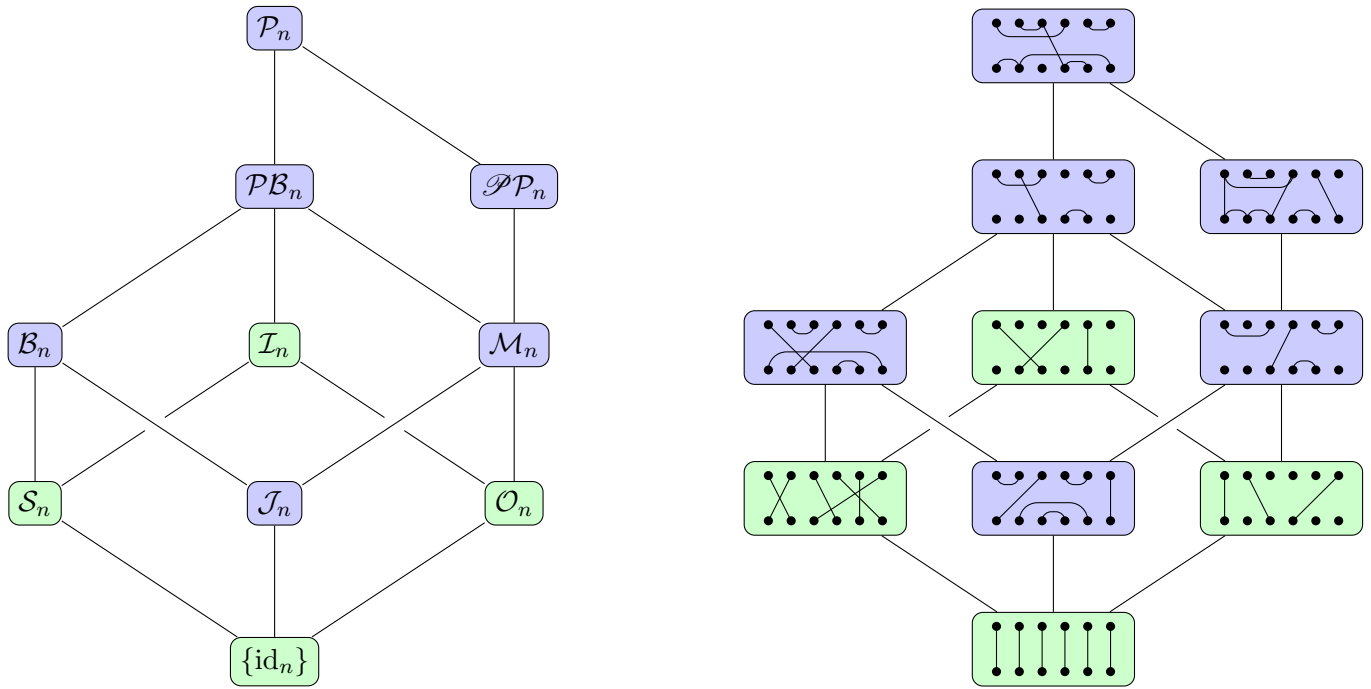
\begin{figure}[ht]
\begin{center}
\begin{tikzpicture}[scale=1.05]
\node[rounded corners,rectangle,draw,fill=blue!20] (Z) at (3,8) {$\P_n$};
\node[rounded corners,rectangle,draw,fill=blue!20] (A) at (3,6) {$\PB_n$};
\node[rounded corners,rectangle,draw,fill=blue!20] (AA) at (6,6) {$\PP_n$};
\node[rounded corners,rectangle,draw,fill=blue!20] (C) at (0,4) {$\B_n$};
\node[rounded corners,rectangle,draw,fill=green!20] (B) at (3,4) {$\I_n$};
\node[rounded corners,rectangle,draw,fill=blue!20] (D) at (6,4) {$\M_n$};
\node[rounded corners,rectangle,draw,fill=green!20] (E) at (0,2) {$\S_n$};
\node[rounded corners,rectangle,draw,fill=blue!20] (G) at (3,2) {$\J_n$};
\node[rounded corners,rectangle,draw,fill=green!20] (F) at (6,2) {$\O_n$};
\node[rounded corners,rectangle,draw,fill=green!20] (H) at (3,0) {$\{\1\}$};
\draw (Z)--(A)--(B)--(E)--(H)--(F)--(D)--(AA)--(Z) (B)--(F);
\fill[white] ($(B)!0.5!(E)$) circle(.15);
\fill[white] ($(B)!0.5!(F)$) circle(.15);
\draw (H)--(G)--(C)--(A)--(D)--(G) (C)--(E);
\end{tikzpicture}
\qquad\qquad\qquad
\begin{tikzpicture}[scale=1]
\node[rounded corners,rectangle,draw,fill=blue!20] (Z) at (3,8) {$\custpartn{1,2,3,4,5,6}{1,2,3,4,5,6}{\uarcx14{.6}\uarcx23{.3}\uarcx56{.3}\darc12\darcx26{.6}\darcx45{.3}\stline34}$};
\node[rounded corners,rectangle,draw,fill=blue!20] (AA) at (6,6) {$\custpartn{1,2,3,4,5,6}{1,2,3,4,5,6}{\uarcx14{.6}\uarcx23{.2}\darc45\darc12\darc23\stline43\stline11\stline56}$};
\node[rounded corners,rectangle,draw,fill=blue!20] (A) at (3,6) {$\custpartn{1,2,3,4,5,6}{1,2,3,4,5,6}{\uarcx13{.5}\uarc56\darc45\stline23}$};
\node[rounded corners,rectangle,draw,fill=blue!20] (C) at (0,4) {$\custpartn{1,2,3,4,5,6}{1,2,3,4,5,6}{\stline13\stline42\uarc23\uarc56\darcx16{.8}\darc45}$};
\node[rounded corners,rectangle,draw,fill=green!20] (B) at (3,4) {$\custpartn{1,2,3,4,5,6}{1,2,3,4,5,6}{\stline13\stline42\stline55}$};
\node[rounded corners,rectangle,draw,fill=blue!20] (D) at (6,4) {$\custpartn{1,2,3,4,5,6}{1,2,3,4,5,6}{\uarcx13{.5}\uarc56\darc45\stline43}$};
\node[rounded corners,rectangle,draw,fill=green!20] (E) at (0,2) {$\custpartn{1,2,3,4,5,6}{1,2,3,4,5,6}{\stlines{1/2,2/1,3/4,4/6,5/5,6/3}}$};
\node[rounded corners,rectangle,draw,fill=blue!20] (G) at (3,2) {$\custpartn{1,2,3,4,5,6}{1,2,3,4,5,6}{\uarcx12{.4}\uarc45\darc34\darcx25{.8}\stlines{3/1,6/6}}$};
\node[rounded corners,rectangle,draw,fill=green!20] (F) at (6,2) {$\custpartn{1,2,3,4,5,6}{1,2,3,4,5,6}{\stlines{1/1,2/3,6/4}}$};
\node[rounded corners,rectangle,draw,fill=green!20] (H) at (3,0) {$\custpartn{1,2,3,4,5,6}{1,2,3,4,5,6}{\stlines{1/1,2/2,3/3,4/4,5/5,6/6}}$};
\draw (Z)--(A)--(B)--(E)--(H)--(F)--(D)--(AA)--(Z) (B)--(F);
\fill[white] ($(B)!0.5!(E)$) circle(.15);
\fill[white] ($(B)!0.5!(F)$) circle(.15);
\draw (H)--(G)--(C)--(A)--(D)--(G) (C)--(E);
\end{tikzpicture}
\end{center}
\vspace{-5mm}
\caption{Important submonoids of $\P_n$ (left) and representative elements from each submonoid (right). The congruence lattices are already known for the monoids shaded green (see Section \ref{sect:SnInOn}); in this article, we describe the congruence lattices for all the other monoids.}
\label{fig:submonoids}
\end{figure}

\subsection{Green's equivalences and ideals of diagram monoids}
\label{sect:prelim_Green}

Green's equivalences $\R$, $\L$, $\gJ$, $\H$ and $\D$ reflect the ideal
structure of a semigroup $S$, and are the fundamental structural tool in
semigroup theory.  They are defined as follows.  We write $S^1=S$ if $S$ is a monoid; otherwise $S^1$ is the monoid
obtained from $S$ by adjoining an identity element to $S$.  Then, for $a,b\in S$,
\[
a\R b \iff aS^1=bS^1 \COMMA a\L b \iff S^1a=S^1b \COMMA a\gJ b \iff S^1aS^1 = S^1bS^1;
\]
further, ${\H}={\R}\cap{\L}$, and $\D$ is the join ${\R}\vee{\L}$: i.e., the least equivalence containing $\R$ and $\L$.
It is well known that ${\D}={\R}\circ{\L}={\L}\circ{\R}$ for any semigroup $S$, and that ${\D}={\gJ}$ when $S$ is finite (as is the case for all semigroups considered in this article).  
If $\K$ is any of Green's relations, and if $a\in S$, we write $K_a=\set{b\in S}{a\K b}$ for the $\K$-class of $a$ in $S$.
The set $S/{\gJ}=\set{J_a}{a\in S}$ of all $\gJ$-classes of $S$ is partially ordered as follows.  
For $a,b\in S$, we say that $J_a\leq J_b$ if $a\in S^1bS^1$.  If $T$ is a subset of $S$ that is a union of $\gJ$-classes, we write $T/{\gJ}$ for the set of all $\gJ$-classes of $S$ contained in~$T$.
The reader is referred to \cite[Chapter 2]{CPbook}, \cite[Chapter 2]{Howie} or \cite[Appendix A]{RSbook} for a more detailed introduction to Green's relations.

Green's equivalences on all diagram monoids considered in this article are
governed by (co)domains, (co)kernels and ranks, as specified in the following
proposition.  For $\P_n$ this  was first proved in \cite{Wilcox2007}, though
the terminology there was different.  For the other monoids see \cite[Theorem
2.4]{DEG2017} and also \cite{Fernandes2001,FL2011,GMbook,Wilcox2007}.  The
proposition will be used frequently throughout the paper without explicit
reference.

\begin{prop}\label{prop:green_all_inclusive}
Let $\mathcal{K}_n$ be any of the monoids $\P_n,\PB_n,\B_n,\PP_n,\M_n,\I_n,\mathcal J_n,\O_n$.  If $\al,\be\in\mathcal K_n$, then
  \bit
\itemit{i} $\al\R\be\iff\dom(\al)=\dom(\be)$ and $\ker(\al)=\ker(\be)$,
\itemit{ii} $\al\L\be\iff\codom(\al)=\codom(\be)$ and $\coker(\al)=\coker(\be)$,
\itemit{iii} $\al\gJ\be\iff\al\D\be\iff\rank(\al)=\rank(\be)$,
\itemit{iv} $J_\al\leq J_\be\iff\rank(\al)\leq\rank(\be)$.
\epfres
\eit
\end{prop}

\begin{rem}\label{rem:green_sumbonoids}
A number of consequences and simplifications arise 
from Proposition \ref{prop:green_all_inclusive}.  
We list them here, again to be used throughout, usually without explicit reference.
\bit 
\item[(i)]    \label{rem:gs:item1}
If $\al,\be$ belong to any of the monoids we consider, then $\alpha\R\beta$ (respectively, $\alpha\L\beta$, $\alpha\gJ\beta$, $\alpha\H\beta$)
if and only if $\alpha^\ast\L\beta^\ast$ (respectively, $\alpha^\ast\R\beta^\ast$, $\alpha^\ast\gJ\beta^\ast$, $\alpha^\ast\H\beta^\ast$).
\item[(ii)] \label{rem:gs:item2}
Elements of $\I_n$ have trivial (co)kernels, so
\[
\al\R\be \iff \dom(\al)=\dom(\be)\ \  
\text{and}\ \  
\al\L\be \iff \codom(\al)=\codom(\be) 
\ \ \ \text{for} \  \al,\be\in\I_n \text{ or } \O_n.
\]
\item[(iii)] \label{rem:gs:item3}
The (co)domain of an element of $\B_n$ is completely determined by its (co)kernel, so
\[
\al\R\be \iff \ker(\al)=\ker(\be) \ \ 
\text{and}\ \ 
 \al\L\be \iff \coker(\al)=\coker(\be)\ \ \ 
 \text{for}\  \al,\be\in\B_n \  \text{or}\ \J_n.
\]
\item[(iv)] \label{rem:gs:item4}
$\H$ is the trivial relation on $\PP_n$, $\M_n$, $\J_n$ and $\O_n$.  
\item[(v)] \label{rem:gs:item5}
The $\gJ$-classes in all the monoids under consideration are in one-one
  correspondence with the available ranks. In $\P_n$, $\PB_n$, $\PP_n$, $\M_n$,
  $\I_n$, $\O_n$, there is a $\gJ$-class for every $r\in\{0,\dots,n\}$
  and in $\B_n$, $\J_n$ for every $r\in\{0,\dots,n\}$ with $r\equiv n\pmod{2}$.
\item[(vi)] \label{rem:gs:item6}
All $\gJ$-classes in all these monoids are regular: equivalently, every
  $\gJ$-class contains an idempotent.
\item[(vii)] \label{rem:gs:item7}
The maximal subgroup (i.e., group $\H$-class)  containing an idempotent of rank $r$ is isomorphic to $\S_r$ in
$\P_n,\PB_n,\B_n,\I_n$, and is trivial in $\PP_n,\M_n,\J_n, \O_n$.
\item[(viii)] \label{rem:gs:item8}
All of Green's relations coincide with the universal relation on $\S_n$, or indeed on any group.
\eit
\end{rem}

Let $\mathcal{K}_n\in\{\P_n,\PB_n,\B_n,\PP_n,\M_n,\I_n,\mathcal J_n,\O_n\}$.
We will denote the $\gJ$-class of $\mathcal{K}_n$ of partitions $\alpha\in\mathcal{K}_n$ with $\rank(\alpha)=r$ by $J_r(\mathcal{K}_n)$, or, if there is no danger of confusion, simply by $J_r$.
Corresponding to~$J_r$ is the ideal 
$I_r=I_r(\mathcal{K}_n)=\set{\alpha\in\mathcal{K}_n}{\rank(\alpha)\leq r}$.
Since the $\gJ$-classes of $\mathcal K_n$ form a chain under the ordering discussed above, the ideals of $\mathcal K_n$ are precisely the sets $I_r$, and these form a chain under inclusion \cite[Proposition 2.6]{DEG2017}.

The structure of the minimal ideal will turn out to have a commanding influence on the congruence lattice of individual diagram monoids.  In almost all of the above monoids, the minimal ideal is $I_0=J_0$, the set of all elements of rank $0$; the only two exceptions are $\B_n$ and $\J_n$ for $n$ odd, where the minimal ideal is $I_1=J_1$; see Table \ref{tab:minimal}.  Furthermore, the minimal ideal is $\{[\emptyset]\}$ in $\I_n$ and $\O_n$, but in every other case it contains multiple $\R$- and $\L$-classes, with trivial exceptions for small $n$.  In every case the minimal ideal consists entirely of idempotents: i.e., it is a rectangular band.  It will transpire in the course of the paper that this is precisely the cause behind the non-linear structure of the congruence lattices.

\begin{table}[H]
\begin{center}
\renewcommand{\arraystretch}{1.8}
\begin{tabular}{|c|c|}
\hline
Monoid $\mathcal{K}_n$ & Minimal ideal of $\mathcal{K}_n$ \\
\hline
%$\P_n$, $\PP_n$, & \\
%$\PB_n$, $\M_n$, & $I_0(\mathcal{K}_n) = \set{\al\in\mathcal{K}_n}{\rank(\al)=0}$ \\
$\P_n$, $\PP_n$, $\PB_n$, $\M_n$, $\B_n$ ($n$ even), $\J_n$ ($n$ even) & $I_0(\mathcal{K}_n) = J_0(\mathcal{K}_n) = \set{\al\in\mathcal{K}_n}{\rank(\al)=0}$ \\
%\hline
$\B_n$ ($n$ odd), $\J_n$ ($n$ odd) & $I_1(\mathcal{K}_n) = J_1(\mathcal{K}_n) = \set{\al\in\mathcal{K}_n}{\rank(\al)=1}$ \\
%\hline
$\I_n$, $\O_n$ & $I_0(\mathcal{K}_n) = J_0(\mathcal{K}_n) = \{[\emptyset]\}$ \\
\hline
\end{tabular}
\end{center}
\vspace{-5mm}
\caption{Minimal ideals of diagram monoids.}
\label{tab:minimal}
\end{table}

Figure \ref{fig:minimal} pictures the elements of the minimal ideals of $\P_3$ and $\PB_3\sub\P_3$ (left), and of $\B_3$ and $\J_3\sub\B_3$ (right).  In both diagrams, each row (respectively, column) consists of an entire $\R$-class (respectively, $\L$-class).  Note that $I_0(\PP_3)=I_0(\P_3)$ and $I_0(\M_3)=I_0(\PB_3)$, so that Figure \ref{fig:minimal} (left) also pictures the minimal ideals of $\PP_3$ and $\M_3$; for larger $n$, we have $I_0(\PP_n)\subsetneq I_0(\P_n)$ and $I_0(\M_n)\subsetneq I_0(\PB_n)$.

\newcommand\nc\newcommand
\nc\Rone{}
\nc\Rtoo{\uarc12}
\nc\Rthree{\uarc23}
\nc\Rfour{\uarc13}
\nc\Rfive{\uarc12\uarc23}
\nc\Lone{}
\nc\Ltwo{\darc12}
\nc\Lthree{\darc23}
\nc\Lfour{\darc13}
\nc\Lfive{\darc12\darc23}

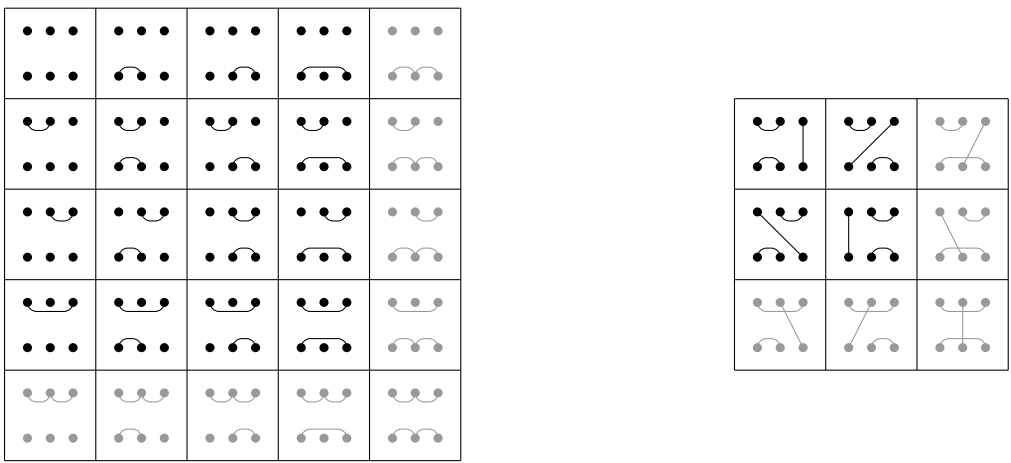
\begin{figure}[ht]
\begin{center}
\begin{tikzpicture}[scale=.6]
%\draw[help lines] (-1,-1) grid (9,9);
%
\node () at (0,8) {$\custpartn{1,2,3}{1,2,3}{\Rone\Lone}$};
\node () at (2,8) {$\custpartn{1,2,3}{1,2,3}{\Rone\Ltwo}$};
\node () at (4,8) {$\custpartn{1,2,3}{1,2,3}{\Rone\Lthree}$};
\node () at (6,8) {$\custpartn{1,2,3}{1,2,3}{\Rone\Lfour}$};
\node[black!40] () at (8,8) {$\custpartn{1,2,3}{1,2,3}{\Rone\Lfive}$};
\node () at (0,6) {$\custpartn{1,2,3}{1,2,3}{\Rtoo\Lone}$};
\node () at (2,6) {$\custpartn{1,2,3}{1,2,3}{\Rtoo\Ltwo}$};
\node () at (4,6) {$\custpartn{1,2,3}{1,2,3}{\Rtoo\Lthree}$};
\node () at (6,6) {$\custpartn{1,2,3}{1,2,3}{\Rtoo\Lfour}$};
\node[black!40] () at (8,6) {$\custpartn{1,2,3}{1,2,3}{\Rtoo\Lfive}$};
\node () at (0,4) {$\custpartn{1,2,3}{1,2,3}{\Rthree\Lone}$};
\node () at (2,4) {$\custpartn{1,2,3}{1,2,3}{\Rthree\Ltwo}$};
\node () at (4,4) {$\custpartn{1,2,3}{1,2,3}{\Rthree\Lthree}$};
\node () at (6,4) {$\custpartn{1,2,3}{1,2,3}{\Rthree\Lfour}$};
\node[black!40] () at (8,4) {$\custpartn{1,2,3}{1,2,3}{\Rthree\Lfive}$};
\node () at (0,2) {$\custpartn{1,2,3}{1,2,3}{\Rfour\Lone}$};
\node () at (2,2) {$\custpartn{1,2,3}{1,2,3}{\Rfour\Ltwo}$};
\node () at (4,2) {$\custpartn{1,2,3}{1,2,3}{\Rfour\Lthree}$};
\node () at (6,2) {$\custpartn{1,2,3}{1,2,3}{\Rfour\Lfour}$};
\node[black!40] () at (8,2) {$\custpartn{1,2,3}{1,2,3}{\Rfour\Lfive}$};
\node[black!40] () at (0,0) {$\custpartn{1,2,3}{1,2,3}{\Rfive\Lone}$};
\node[black!40] () at (2,0) {$\custpartn{1,2,3}{1,2,3}{\Rfive\Ltwo}$};
\node[black!40] () at (4,0) {$\custpartn{1,2,3}{1,2,3}{\Rfive\Lthree}$};
\node[black!40] () at (6,0) {$\custpartn{1,2,3}{1,2,3}{\Rfive\Lfour}$};
\node[black!40] () at (8,0) {$\custpartn{1,2,3}{1,2,3}{\Rfive\Lfive}$};
\foreach \x in {-1,1,3,5,7,9} {\draw (-1,\x)--(9,\x); \draw(\x,-1)--(\x,9);}
\begin{scope}[shift={(14,0)}]
\node () at (2,6) {$\custpartn{1,2,3}{1,2,3}{\Rtoo\Ltwo\stline33}$};
\node () at (4,6) {$\custpartn{1,2,3}{1,2,3}{\Rtoo\Lthree\stline31}$};
\node[black!40] () at (6,6) {$\custpartn{1,2,3}{1,2,3}{\Rtoo\Lfour\stline32}$};
\node () at (2,4) {$\custpartn{1,2,3}{1,2,3}{\Rthree\Ltwo\stline13}$};
\node () at (4,4) {$\custpartn{1,2,3}{1,2,3}{\Rthree\Lthree\stline11}$};
\node[black!40] () at (6,4) {$\custpartn{1,2,3}{1,2,3}{\Rthree\Lfour\stline12}$};
\node[black!40] () at (2,2) {$\custpartn{1,2,3}{1,2,3}{\Rfour\Ltwo\stline23}$};
\node[black!40] () at (4,2) {$\custpartn{1,2,3}{1,2,3}{\Rfour\Lthree\stline21}$};
\node[black!40] () at (6,2) {$\custpartn{1,2,3}{1,2,3}{\Rfour\Lfour\stline22}$};
\foreach \x in {1,3,5,7} {\draw (1,\x)--(7,\x); \draw(\x,1)--(\x,7);}
\end{scope}
\end{tikzpicture}
\end{center}
\vspace{-5mm}
\caption{Left:~the minimal ideal of $\P_3$ (black and gray elements) and of $\PB_3$ (black elements only).  Right:~the minimal ideal of $\B_3$ (black and gray elements) and of $\J_3$ (black elements only).}
\label{fig:minimal}
\end{figure}

\subsection{Congruence lattices}\label{sect:prelim_congruences}

Recall that an equivalence relation $\xi$ on a semigroup $S$ is a \emph{right
congruence} if, for all $a,b,x\in S$, $(a,b)\in\xi$ implies $(ax,bx)\in\xi$.
\emph{Left congruences} are defined analogously.  A relation on $S$ is a
\emph{congruence} if it is a left and right congruence.  The \emph{join} of two
equivalence relations $\xi,\zeta$ on the same set is the smallest equivalence
relation containing $\xi\cup\zeta$; we denote this by $\xi\vee\zeta$.  The join
of two congruences on a semigroup is also a congruence, and so too is the
intersection.  The set $\Cong(S)$ of all congruences on $S$ forms a lattice,
known as the \emph{congruence lattice of $S$}, under the operations of
intersection and join.  The maximum and minimum elements of $\Cong(S)$ are 
the \emph{universal congruence} ${\nabla_S=S\times S}$, and the \emph{trivial congruence} $\De_S=\bigset{(a,a)}{a\in S}$, respectively.

If $\Om\sub S\times S$, we write $\Om^\sharp_S$ (or just $\Om^\sharp$ if $S$ is
understood from context) for the congruence on $S$ generated by~$\Om$: that is,
the smallest congruence on $S$ containing $\Om$.  So $\Om^\sharp_S$ is the
least equivalence on $S$ containing the set $\bigset{(axb,ayb)}{(x,y)\in\Om,\
a,b\in S^1}$; see for example \cite[Section~1.5]{Howie}.  In the special case
that $\Om=\{(x,y)\}$ consists of a single pair, we will write $\cg
xy_S=\Omega^\sharp_S$; such a congruence is called \emph{principal}.
Principal congruences will play a major role in our investigations of finite diagram monoids; in fact, almost all congruences on the diagram monoids we consider are principal, as can be seen in Figures~\ref{fig-CongPn},~\ref{fig-CongPPn},~\ref{fig:Hasse_Bn} and~\ref{fig:Hasse_Jn}.
Clearly, every congruence on a semigroup is a (possibly infinite) join, and indeed a union, of principal congruences.
In particular, if every congruence on $S$ is finitely generated (i.e., of the form $\Om^\sharp$ for some finite subset $\Om$ of $S\times S$), then $\Cong(S)$ is generated as a $\vee$-semilattice by the set of all principal congruences.  The latter observation leads to the following simple lemma, which plays a crucial role in our proofs.

\begin{lemma}\label{lem:principal}
Suppose $S$ is a semigroup for which every congruence is finitely generated.  If $\Sigma$ is a subset of $\Cong(S)$ that is closed under $\vee$, and contains every principal congruence, then $\Sigma=\Cong(S)$.~\epfres
\end{lemma}

If $T$ is a subsemigroup of a semigroup $S$, then the congruence lattices $\Cong(S)$ and $\Cong(T)$ may be related in the following way.  Any congruence $\xi$ on $S$ induces a congruence on $T$ defined by $\xi^T=\xi\cap(T\times T)$.  And any congruence $\zeta$ on $T$ induces a congruence on $S$: namely, $\zeta_S^\sharp$, the congruence on $S$ generated by $\zeta$.  So we have two maps
\[
\Phi_{S,T}:\Cong(S)\to\Cong(T):\xi\mt\xi^T \AND \Psi_{S,T}:\Cong(T)\to\Cong(S):\zeta\mt\zeta_S^\sharp.
\]
In general, neither map need be injective or surjective.  Also, although both $\Phi_{S,T}$ and $\Psi_{S,T}$ preserve $\sub$, and although $\Phi_{S,T}$ preserves $\cap$, and $\Psi_{S,T}$ preserves $\vee$, neither $\Phi_{S,T}$ nor $\Psi_{S,T}$ need preserve all three of $\sub$, $\cap$ and~$\vee$.  In other words, neither $\Phi_{S,T}$ nor $\Psi_{S,T}$ need be a lattice homomorphism.  But we note that
\[
\xi(\Phi_{S,T}\circ\Psi_{S,T}) \sub \xi \AND \zeta\sub\zeta(\Psi_{S,T}\circ\Phi_{S,T}) \qquad\text{for all $\xi\in\Cong(S)$ and $\zeta\in\Cong(T)$.}
\]
The maps $\Phi_{S,T}$ and $\Psi_{S,T}$ will play an important role here,
in that they will establish links between the different diagram monoids we consider:
perhaps even more significantly, with the three monoids for which the
congruence lattices are already known: namely, $\S_n$, $\I_n$ and $\O_n$; see
Section \ref{sect:SnInOn}.

Because of its inherent importance, the involution $\alpha\mapsto\alpha^\ast$
is sometimes given the status of an additional fundamental operation on a
diagram monoid $\mathcal{K}_n$, turning the latter into a so-called
$\ast$-semigroup.  In this context, a \emph{$\ast$-congruence} is a semigroup
congruence $\xi$ that additionally satisfies $(\alpha,\beta)\in\xi\implies(\alpha^\ast,\beta^\ast)\in\xi$.  The set $\Cong^*(\mathcal{K}_n)$ of all
$*$-congruences of $\mathcal K_n$ is a sublattice of $\Cong(\mathcal{K}_n)$.

\section{Congruences from ideals, normal subgroups and retractions}\label{sec-fam}

This section is pivotal for the rest of the paper.  In it, we introduce a number of constructions yielding families of congruences on certain classes of semigroups, including all finite semigroups.  
All of the congruences discussed in subsequent sections will be instances of these constructions.  The Rees congruence associated to an ideal (Definition~\ref{defn:Rees}) can be viewed as a starting point for each construction.  The key building blocks are what we shall call \emph{retractable ideals} (Definition~\ref{defn:retractable}) and \emph{liftable congruences} (Definition~\ref{defn:lift}), and a certain relation associated to each normal subgroup of a maximal subgroup in a stable, regular $\gJ$-class (Definitions~\ref{defn:stableJ} and~\ref{Rel-nu}).  
In the current paper, the results of this section will be exclusively applied to finite semigroups.  However, we state and prove them under more general hypotheses so as be applicable to certain classes of infinite semigroups as well; see Remark \ref{rem:infinite} for further comments on the infinite case, and also \cite{ER2018} for an application to infinite diagram monoids.
For completeness, we begin by recalling the definition of a Rees congruence:

\begin{defn}\label{defn:Rees}
A non-empty subset $I$ of a semigroup $S$ is an \emph{ideal} if $ax,xa\in I$ for all $x\in I$ and $a\in S$.  To an ideal $I$, we associate the \emph{Rees congruence} $R_I=\De_S\cup(I\times I)$.
\end{defn}

Recall that a semigroup has at most one minimal ideal, and that if a minimal ideal exists, then it is a single $\gJ$-class and is the intersection of all ideals; see for example \cite[Section 3.1]{Howie}.  Since the intersection of any finite collection of ideals is non-empty, every finite semigroup has a minimal ideal.  
%Not every infinite semigroup has a minimal ideal, however; consider the natural numbers under addition, for example.

\begin{defn}\label{defn:retractable}
  Let $S$ be a semigroup with a minimal ideal $M$.  An ideal $I$ of $S$ is \textit{retractable} if there exists a
  homomorphism $f: I \to M$ such that $xf = x$ for all $x\in M$; such a
  homomorphism is called a \textit{retraction}. 
\end{defn}

Clearly, if $I,J$ are ideals of $S$ with $I\sub J$ and~$J$ retractable, then $I$ is also retractable.  Also, note that if $S$ has a zero element, then $\{0\}$ is the minimal ideal and all ideals are trivially retractable.

Before we describe the first congruence construction, we prove some preliminary results concerning retractions.  
Recall that an element $x$ of a semigroup $S$ is \emph{regular} if there exists $y\in S$ such that $x=xyx$.  Then with $a=yxy$, note that $x=xax$ and $a=axa$, and that $xa$ and $ax$ are idempotents of $S$ with $x\R xa$ and $x\L ax$.  We say a subset of $S$ is regular if all its elements are regular.  If any element of a $\D$-class of $S$ is regular, then \emph{every} element of the $\D$-class is regular, but this is not true of $\gJ$-classes in general; see \cite[Proposition 2.3.1]{Howie}.

The next sequence of results concern semigroups with regular minimal ideals.  Finite semigroups always have a regular minimal ideal; this is well known, and also follows from Lemma \ref{lem:Mcr} below.  
%On the other hand, infinite semigroups can have a non-regular minimal ideal; see \cite[Chapter 2, Exercise 1]{Howie}.

\begin{lemma}
\label{lemma-RetractAux}
Let $S$ be a semigroup with a regular minimal ideal $M$, and let $I$ be an ideal of $S$. If $f: I\rightarrow M$ is a retraction, then $(sxt)f=s(xf)t$ for all $x\in I$ and $s,t\in S^1$.
\end{lemma}

\pf
We will prove the lemma by showing that $(sx)f=s(xf)$ for any $s\in S ^ 1$ and
$x\in I$.  The equality $(xt)f=(xf)t$ is dual, and then $(sxt)f=s(xt)f=s(xf)t$.
Let $e\in M$ be any right identity for $xf$; such an~$e$ exists because $M$
is regular. Then, since $f$ is a retraction and $e,xe\in M$, we have
$xf=(xf)e=(xf)(ef)=(xe)f=xe$.  Next, let $e_1\in M$ be a left identity for
$(sx)f$.  Then
$
(sx)fe=e_1(sx)fe=(e_1f)(sx)f(ef)=(e_1sxe)f=(e_1s)f(xf)(ef)=(e_1s)f(xf)=(e_1sx)f=(e_1f)(sx)f=e_1(sx)f=(sx)f
$.
So $e$ is a right identity for $(sx)f$ as well, and hence $(sx)f=sxe=s(xf)$, as
required.
\epf

\begin{cor}\label{cor:retract_unique}
Let $S$ be a semigroup with a regular minimal ideal $M$.  If $I$ is a retractable ideal of $S$, then there exists a unique retraction from $I$ to $M$.
\end{cor}

\pf
Suppose $f,g:I\to M$ are retractions, and let $x\in I$.  Let $e\in M$ be a
left identity for $xf$.  Using Lemma \ref{lemma-RetractAux}, we see that
$xf = e (xf) = (ex)f = ex = (ex)g = e (xg)$.  A dual argument shows that
$xg=(xf) e'$ for some $e'$.  But then $xf = e (xg) = e (xf)  e' = (xf)e'=xg$.
\epf

Although we will make no subsequent use of the next result, we believe it is of independent interest, and so include it for completeness.

\begin{prop}
Let $S$ be a semigroup with a regular minimal ideal $M$.  Then the union of any non-empty family of retractable ideals of $S$ is a retractable ideal.  Consequently,~$S$ contains a unique maximal retractable ideal.
\end{prop}

\pf The second claim clearly follows from the first.  Suppose $\set{I_k}{k\in K}$ is a non-empty collection of retractable ideals of $S$, and put $I=\bigcup_{k\in K}I_k$; clearly $I$ is an ideal of $S$.  For each $k\in K$, let $f_k:I_k\to M$ be the unique retraction.  For any non-empty subset $L\sub K$, define $I_L=\bigcap_{l\in L}I_l$.  Since $M$ is contained in each~$I_l$, it follows that $I_L$ is non-empty, and so an ideal of $M$.  Furthermore, for any $l\in L$, the restriction of $f_l$ to~$I_L$ is a retraction; thus, by Corollary \ref{cor:retract_unique}, all such restrictions are equal.  It follows that there is a well-defined map $f:I\to M$ defined, for $x\in I$, by $xf=xf_k$ if $x\in I_k$.  Clearly $f$ acts identically on $M$, so to complete the proof, it remains to show that $(xy)f=(xf)(yf)$ for all $x,y\in I$.  With this in mind, suppose $x,y\in I$, and let $k,l\in K$ be such that $x\in I_k$ and $y\in I_l$.  Then, using Lemma~\ref{lemma-RetractAux} and the facts that $xy\in I_l$ and $(xy)f_l\in M$, we have $(xf)(yf) = (xf_k)(yf_l) = (x(yf_l))f_k = ((xy)f_l)f_k = (xy)f_l = (xy)f$, as required.
\epf

Our first construction combines the idea of a retractable ideal with the following.

\begin{defn}\label{defn:lift}
Let $S$ be a semigroup with a minimal ideal $M$.  A congruence $\xi$ on $M$ is \emph{liftable} if any, and hence all, of the following equivalent conditions are satisfied:
\bit
\item[(i)] $\De_S\cup\xi$ is a congruence on $S$, or
\item[(ii)] there exists a congruence $\zeta$ on $S$ such that $\xi=\zeta\cap(M\times M)$, or
\item[(iii)] for all $(x,y)\in\xi$ and $s\in S$, $(xs,ys),(sx,sy)\in\xi$.
\eit
\end{defn}

\begin{defn}\label{defn:lift2}
Let $S$ be a semigroup with a regular minimal ideal $M$.  To any retractable ideal $I$ of $S$ (Definition~\ref{defn:retractable}), and to any liftable congruence $\xi$ on $M$ (Definition~\ref{defn:lift}), we
associate the relation
\[
  \zeta_{I,\xi}=\De_S\cup\set{(x,y)\in I\times I}{(xf,yf) \in \xi},
\]
where $f:I\to M$ is the unique retraction.
\end{defn}

%\begin{rem}\label{rem:I<J}
Clearly the trivial and universal congruences on $M$, $\De_M$ and $\nabla_M$, are liftable: in particular, when $I=M$ is the minimal ideal itself, we have $\zeta_{M,\De_M} = \De_S$ and $\zeta_{M,\nabla_M}=R_M$.  Note also that if $\xi$ is a liftable congruence, and if $I_1,I_2$ are retractable ideals with $I_1\sub I_2$, then $\zeta_{I_1,\xi}\sub\zeta_{I_2,\xi}$.
%\end{rem}

\begin{prop}\label{prop:lift}
Let $S$ be a semigroup with a regular minimal ideal $M$.  If $I$ is a retractable ideal of $S$, and $\xi$ a liftable congruence on $M$, then the relation $\zeta_{I,\xi}$ is a congruence on $S$.
\end{prop}

\pf Write $\zeta=\zeta_{I,\xi}$ for brevity, and let $f:I\to M$ be the retraction.  Let $(x,y)\in\zeta$ and $s\in S$ be
arbitrary.  We must show that $(xs,ys),(sx,sy)\in\zeta$.  This is clear if
$x=y$, so suppose $x,y\in I$ and $xf\mathrel\xi yf$.  Since~$I$ is an
ideal, we have $xs,ys\in I$.  By Lemma \ref{lemma-RetractAux}, and condition (iii) of Definition~\ref{defn:lift}, we have $(xs)f = (xf)s \mathrel\xi
(yf)s = (ys)f$, showing that $(xs,ys)\in\zeta$.  A dual argument shows that
$(sx,sy)\in\zeta$. 
\epf

%\begin{rem}
In general, the minimal ideal may have many congruences, only few of which are liftable congruences. 
In a full transformation
semigroup $\T_X$, for example, the minimal ideal $M$ consists of the
constant transformations and is a right zero semigroup of size $|X|$. Hence, every
equivalence relation on $M$ is a congruence. On the other hand, for any $(f_1,
f_2), (g_1,g_2)\in M \times M$ with $f_1\not = f_2$ and $g_1 \not =g_2$  
there exists $h\in \T_X$ such that $(f_1h, f_2h) = (g_1,g_2)$, and so 
the only liftable congruences on $M$ are $\De_M$ and $\nabla_M$. 
%\end{rem}

In what follows, we will be mostly concerned with certain liftable congruences
that exist in \emph{every} finite semigroup, beyond the obvious $\De_M$ and
$\nabla_M$.  These liftable congruences are present in many infinite semigroups as well; the key concept is \emph{stability}.

\begin{defn}\label{defn:stableJ}
A $\gJ$-class $J$ of a semigroup $S$ is \emph{stable} if, for all $x\in J$ and $a\in S$,
\[
xa\gJ x \implies xa\R x \AND ax\gJ x \implies ax\L x.
\]
\end{defn}

It is well known that every $\gJ$-class of a finite semigroup is stable; see for example \cite[Theorem A.2.4]{RSbook}.  We require two preliminary results concerning stability; Lemmas \ref{lem:stableJD} and \ref{lem:Mcr} are probably well known, but we include simple proofs for convenience.  The proof of the next result utilises the partial order on $\gJ$-classes discussed in Section \ref{sect:prelim_Green}.

\begin{lemma}\label{lem:stableJD}
If $J$ is a stable $\gJ$-class of a semigroup $S$, then $J$ is a $\D$-class.
\end{lemma}

\pf
Let $x,y\in J$.  We must show that $x\D y$.  Since $x\gJ y$, we have $y=axb$ for some $a,b\in S^1$.  Then $J_y = J_{axb} \leq J_{ax} \leq J_x = J_y$, so that, in fact, all of these $\gJ$-classes are equal (to $J$).  In particular, $ax\gJ x$ and $(ax)b\gJ ax$.  Since $x,ax\in J$, it follows from stability that $ax\L x$, and $ax \R (ax)b=y$.  Thus, $x\L ax\R y$, which gives $x\D y$.
\epf

As noted above, if a semigroup has a minimal ideal, then this ideal is a single $\gJ$-class.  Recall that a semigroup is \emph{completely regular} if it is a union of groups; in particular, any completely regular semigroup is regular.  Recall that a \emph{rectangular band} is a semigroup consisting of idempotents, all of which are $\D$-related.

\begin{lemma}\label{lem:Mcr}
Let $S$ be a semigroup with a stable minimal ideal $M$.  Then $M$ is completely regular.  In particular, if $M$ is $\H$-trivial, then $M$ is a rectangular band.
\end{lemma}

\pf 
Let $x\in M$.  Since $M$ is an ideal, $x^2\in M$.  Since~$M$ is a single $\gJ$-class, it follows that $x^2\gJ x$.  Stability then gives $x^2\R x$ and $x^2\L x$, so that $x^2\H x$.  Green's Theorem (see \cite[Theorem 2.2.5]{Howie}) then says that the $\H$-class of $x$ is a group.  This completes the proof of the first claim.  The second follows from the first, together with the fact that $M$ is a $\D$-class, which itself follows from Lemma \ref{lem:stableJD}.
\epf

As a consequence of Lemma \ref{lem:Mcr}, all the results above concerning semigroups with a \emph{regular} minimal ideal apply to any semigroup with a \emph{stable} minimal ideal; as noted above, this includes all finite semigroups.
If $\K$ is any of Green's relations on a semigroup $S$ with a minimal ideal $M$, we denote by ${\K^M}={\K}\cap(M\times M)$ the restriction of~$\K$ to $M$.

\begin{lemma}\label{lem:lift}
Let $S$ be a semigroup with a stable minimal ideal $M$.  Then $\L^M$, $\R^M$ and $\H^M$ are all liftable congruences on $M$.
\end{lemma}

\pf By duality, and since ${\H^M}={\L^M}\cap{\R^M}$, it suffices to prove the
result for $\L^M$.  Let $(x,y)\in{\L^M}$ and $s\in S$.  Since $M$ is an ideal,
$xs,ys,sx,sy\in M$.  Since $\L$ is a right congruence, it follows that
$(xs,ys)\in{\L}$, and hence $(xs,ys)\in{\L^M}$.  Since $M$ is a single
$\gJ$-class, $x\gJ sx$, and stability gives $x\L sx$; similarly, $y\L sy$.  It
follows that $sx\L x\L y\L sy$, and so $(sx,sy)\in{\L^M}$. \epf

%\begin{rem}
%Lemma \ref{lem:lift} does not hold for arbitrary semigroups.  For example, the \emph{bicyclic semigroup}, given by the presentation $\la a,b:ba=1\ra$, is $\gJ$-universal and so equal to its own minimal ideal.  It is regular but not stable, and neither $\L$ nor $\R$ is a (liftable) congruence.
%\end{rem}

As a result of Proposition \ref{prop:lift} and Lemma \ref{lem:lift}, we obtain
the following family of congruences associated to an arbitrary retractable
ideal in a semigroup with a stable minimal ideal.

\begin{defn}
\label{defn:lrmI}
Let $S$ be a semigroup with a stable minimal ideal $M$.  Given a retractable ideal $I$ of $S$ (Definition \ref{defn:retractable}), define the congruences
\begin{align*}
\lambda_I=\zeta_{I,{\L^M}}&=\Delta_S\cup\set{ (x,y)\in I\times I}{ xf\L yf},
  &\mu_I=\zeta_{I,{\H^M}}&=\Delta_S\cup\set{ (x,y)\in I\times I}{ xf\H yf},\\
\rho_I=\zeta_{I,{\R^M}}&=\Delta_S\cup\set{ (x,y)\in I\times I}{ xf\R yf},
  &\eta_I=\zeta_{I,\De_M}&=\Delta_S\cup\set{ (x,y)\in I\times I}{ xf= yf},
\end{align*}
where $f:I\to M$ is the unique retraction, and where the relations $\zeta_{I,\xi}$ are as in Definition~\ref{defn:lift2}.
\end{defn}

%\begin{rem}
It is clear that $\eta_I\sub\mu_I=\lam_I\cap\rho_I$, with $\eta_I=\mu_I$ if and only if $M$ is $\H$-trivial, and that all four relations from Definition~\ref{defn:lrmI} are contained in $R_I=\zeta_{I,\nabla_M}$.  
It is also clear that if $I_1,I_2$ are retractable ideals with $I_1\sub I_2$, then $R_{I_1}\sub R_{I_2}$, $\lam_{I_1}\sub \lam_{I_2}$, $\rho_{I_1}\sub \rho_{I_2}$, $\mu_{I_1}\sub \mu_{I_2}$ and $\eta_{I_1}\sub \eta_{I_2}$.  Since $M$ is a single $\D$-class (by Lemma \ref{lem:stableJD}), and since ${\D}={\L}\vee{\R}$, it is easy to show that $R_I=\lam_I\vee\rho_I$ (see also Propositions~\ref{prop-CR3} and~\ref{prop-CR4}).  If $M$ is $\L$-trivial, then $\lam_I=\mu_I=\eta_I$ and $R_I=\rho_I$; if $M$ is $\L$-universal, then $\lam_I=R_I$ and~$\mu_I=\rho_I$.
%\end{rem}

The remaining families of congruences involve the interaction between an ideal and a stable, regular $\gJ$-class directly above it.  Each will make use of the following auxiliary relation:

\begin{defn}
\label{Rel-nu}
Suppose $J$ is a stable, regular $\gJ$-class of a semigroup $S$.  For a normal subgroup $N$ of a maximal subgroup contained in $J$, define the relation
\[
\nu_N = S^1(N\times N)S^1\cap(J\times J) = \set{ (sxt,syt)\in J\times J}{ x,y\in N,\ s,t\in S^1}.
\]
\end{defn}

\begin{rem}
An alternative, perhaps more intuitive, way of viewing the relation $\nu_N$ is
as follows.  
Since~$J$ is stable and regular, the principal factor $\overline{J}$ is a regular
completely $0$-simple semigroup and, hence, isomorphic to a Rees $0$-matrix
semigroup $\M^0[G;K,L; P]$, where $G$ is the maximal subgroup of $S$ containing~$N$, and~$P$ is an $L\times K$ matrix over $G\cup\{0\}$; see \cite[Section 3.2]{Howie}.  The natural
epimorphism $G\rightarrow G/N$ induces an epimorphism ${\M^0[G;K,L;P]\rightarrow
\M^0[G/N;K,L;P/N]}$, where $P/N$ is obtained from $P$ by replacing every entry by
the coset it represents, whose kernel corresponds to a congruence
$\overline{\nu}_N$ on $\overline{J}$.  It is possible to show that
$\nu_N=\overline{\nu}_N\cap(J\times J)=\overline{\nu}_N\setminus\{(0,0)\}$.  
Among other things, once this correspondence has been established, it quickly follows that $\nu_N$ is an equivalence relation, but we will prove this directly in Lemma \ref{lem:nu_H_equivalence}(ii) below.
Since establishing this correspondence is not crucial to our purposes, we omit the~details.
\end{rem}

%\begin{rem}
Although the $\nu_N$ relation from Definition \ref{Rel-nu} could be defined for a normal subgroup $N$ of \emph{any} maximal subgroup, stability of the $\gJ$-class containing $N$ is essential in showing that certain relations built from $\nu_N$ are in fact conguences; see Lemma \ref{lem:nu} and Remark \ref{rem:infinite}.
%\end{rem}

%
If $J$ is a stable, regular $\gJ$-class of a semigroup $S$, then it is a $\D$-class by Lemma \ref{lem:stableJD}, so all maximal subgroups of $S$ contained in $J$ are isomorphic; see \cite[Theorem 2.20]{CPbook}.  The next result shows that all maximal subgroups in such a $\gJ$-class produce the same collection of $\nu_N$ relations; thus, in the applications that follow, we need only consider a single fixed maximal subgroup of each stable, regular $\gJ$-class.

\begin{lemma}\label{lem:G1G2}
Let $G_1,G_2$ be maximal subgroups contained in the same stable, regular $\gJ$-class $J$ of a semigroup $S$.  If $N_1$ is a normal subgroup of $G_1$, then there exists a normal subgroup $N_2$ of $G_2$ such that $\nu_{N_1}=\nu_{N_2}$.
\end{lemma}

\pf
By Lemma \ref{lem:stableJD}, $J$ is a $\D$-class.
By \cite[Theorem 2.20]{CPbook}, there exist $a,b\in S$ such that the map ${\phi:G_1\to G_2:x\mt axb}$ is an isomorphism, with inverse $x\mt bxa$.  
Put $N_2=N_1\phi=aN_1b$; normality of $N_2$ follows from that of $N_1$.  We also have
\[
S^1(N_2\times N_2)S^1 = S^1(aN_1b\times aN_1b)S^1 = S^1(a(N_1\times N_1)b)S^1 \sub S^1(N_1\times N_1)S^1,
\]
which gives $\nu_{N_2} = S^1(N_2\times N_2)S^1 \cap (J\times J) \sub S^1(N_1\times N_1)S^1 \cap (J\times J) = \nu_{N_1}$.  Similarly, $\nu_{N_1}\sub\nu_{N_2}$.
\epf

Since we intend to use the $\nu_N$ relations to construct congruences, we need to know that they are equivalences.

\begin{lemma}\label{lem:nu_H_equivalence}
Suppose $J$ is a stable, regular $\gJ$-class of a semigroup $S$, and $N$ a normal subgroup of a maximal subgroup contained in $J$.  Then
\bit\begin{multicols}2
\itemit{i} $\nu_N\sub{\H}$,
\itemit{ii} $\nu_N$ is an equivalence on $J$.
\end{multicols}\eit
\end{lemma}

\pf
We write $\nu=\nu_N$ for brevity, $e$ for the identity of $N$, and $G=H_e$ for the maximal subgroup of $S$ containing $N$.  If $x\in G$, we will write $x^{-1}$ for the inverse of $x$ in $G$.  

To prove (i), let $(u,v)\in\nu$, so that $u,v\in J$ and $(u,v)=(sxt,syt)$ for some $x,y\in N$ and $s,t\in S^1$.  As in the proof of Lemma \ref{lem:stableJD}, since $x\gJ u=sxt$, it follows that $x\L sx \R sxt=u$.  By \cite[Theorem 2.3]{CPbook},  it follows that the map
\[
\phi:H_x\to H_u:z\mt szt
\]
is a bijection.  Consequently, since $y\in N\sub G=H_x$, we have ${v=syt=y\phi\in H_u}$, so that $(u,v)\in{\H}$, completing the proof of (i).    

For (ii), we first note that $\nu$ is clearly symmetric.  For any $u\in J$, we have $u=set$ for some $s,t\in S^1$ (since $u\gJ e$), and so $(u,u)=(set,set)\in\nu$, demonstrating reflexivity.  Before we establish transitivity, we first claim that
\[
\nu = S^1(N\times\{e\})S^1\cap(J\times J) = S^1(\{e\}\times N)S^1\cap(J\times J).
\]
To prove this, first note that $\nu$ clearly contains the two stated relations.  Conversely,  if $(u,v)\in\nu$, then $u,v\in J$ and $(u,v)=(sxt,syt)$ for some $x,y\in N$ and $s,t\in S^1$; but then also
\[
(u,v)=(s\cdot xy^{-1}\cdot yt,s\cdot e\cdot yt)\in S^1(N\times\{e\})S^1 \ANDSIM (u,v)\in S^1(\{e\}\times N)S^1,
\]
completing the proof of the claim.

To establish transitivity of $\nu$, suppose $(u,v),(v,w)\in\nu$.  We must show that $(u,w)\in\nu$.  Certainly ${u,w\in J}$.  In light of the claim proved in the previous paragraph, we may write 
\[
(u,v)=(axb,aeb) \AND (v,w)=(ced,cyd) \qquad\text{for some $x,y\in N$ and $a,b,c,d\in S^1$.}
\]
Again, as in the proof of Lemma \ref{lem:stableJD}, stability gives
\[
e \L ae \R aeb \L eb \R e \AND
e \L ce \R ced \L ed \R e.
\]
Let $a',b',c',d'\in S$ be such that $e=a'(ae)=(eb)b'=c'(ce)=(ed)d'$.  
Keeping in mind that $aeb=v=ced$, it follows again from \cite[Theorem 2.3]{CPbook} that the maps
\begin{align*}
& \phi_1:H_e\to H_{ae}:z\mt az, &
& \phi_2:H_{eb}\to H_{v}:z\mt az, &
& \phi_3:H_e\to H_{eb}:z\mt zb, &
& \phi_4:H_{ae}\to H_{v}:z\mt zb, &
\\[2truemm]
& \psi_1:H_{ae}\to H_e:z\mt a'z, &
& \psi_2:H_{v}\to H_{eb}:z\mt a'z, &
& \psi_3:H_{eb}\to H_e:z\mt zb', &
& \psi_4:H_{v}\to H_{ae}:z\mt zb' &
\end{align*}
are all bijections, with $\phi_i^{-1}=\psi_i$ for each $i$.  
By (i), we have $(v,w)\in\nu\sub{\H}$, and so $w\in H_v$.  Thus, $w=w\psi_2\phi_2\psi_4\phi_4=aa'wb'b$; consequently, $(u,w)=(axb,a(a'wb')b)$, so the proof will be complete if we can show that $a'wb'\in N$.  Now, $a'wb'=a'(cyd)b'=(a'ce)y(edb')$; thus, since $N$ is a normal subgroup of $G$, it suffices to show that $a'ce,edb'\in G=H_e$ and $edb'=(a'ce)^{-1}$.  
Now, $ce\L e\L ae$ and $ce\R ced=v=aeb\R ae$; together these imply $ce\in H_{ae}$, from which it follows that $a'ce=(ce)\psi_1\in H_e$.
A similar calculation gives $ed\in H_{eb}$, and so $edb'=(ed)\psi_3\in H_e$.  Finally, to show that $edb'=(a'ce)^{-1}$, it suffices to observe that  $(a'ce)(edb') = a'(ced)b' = a'vb' = a'(aeb)b' = e\phi_1\psi_1\phi_3\psi_3 = e$.
\epf

In addition to the relations $\nu_N$, defined above, we need one more key concept in order to describe the remaining families of congruences.

\begin{defn}
\label{defn:RC}
  An \emph{IN-pair} on a semigroup $S$ is a pair $\C=(I,N)$, where
  $I$ is an ideal of~$S$, and $N$ is a normal subgroup of a maximal subgroup of
  $S$ contained in some stable, regular $\gJ$-class $J$ that is minimal in the set
  $(S\setminus I)/{\gJ}$.
  To such an IN-pair, we associate the relation
  \[
  R_\C = R_{I,N} = R_I\cup\nu_N=\Delta_S \cup \nu_N \cup (I\times I),
  \]
  where $\nu_N$ is given in Definition~\ref{Rel-nu}. 
Clearly, when $|N|=1$, $\nu_N=\De_J\sub\De_S$, and so $R_{I,N}=R_I$ is the Rees congruence associated to $I$ (Definition \ref{defn:Rees}); thus, we say an IN-pair $\C=(I,N)$ is \emph{proper} if $|N|\geq2$.
\end{defn}

Our final family of congruences involve a special kind of IN-pair.

\begin{defn}
\label{defn:lrmC0}
Let $S$ be a semigroup with a stable minimal ideal.
  We say that an IN-pair $\C=(I,N)$ on $S$ is
  \emph{retractable} if $I$ is a retractable ideal (Definition \ref{defn:retractable}), and all the elements of $N$ act the same way on~$M$: i.e., $|xN|=|Nx|=1$ for all $x\in M$.
  For such a retractable IN-pair, and for a liftable congruence $\xi$ on $M$ (Definition \ref{defn:lift}), we
  define the relation
\[
\zeta_{I,N,\xi} = \zeta_{I,\xi}\cup\nu_N = \De_S\cup\nu_N\cup\set{(x,y)\in I\times I}{(xf,yf) \in \xi},
\]
where $f:I\to M$ is the unique retraction, and where $\zeta_{I, \xi}$ and $\nu_N$ are given in Definitions~\ref{defn:lift2} and \ref{Rel-nu}, respectively.
\end{defn}

By Lemma \ref{lem:lift}, every retractable IN-pair on a semigroup with a stable minimal ideal yields a natural family of $\zeta_{I,N,\xi}$ relations:

\begin{defn}
\label{defn:lrmC}
Let $S$ be a semigroup with a stable minimal ideal.  To every retractable IN-pair $\C=(I,N)$ on $S$ (Definition \ref{defn:lrmC0}), we associate four relations:
\begin{align*}
\lambda_\C=\lambda_{I,N}=\zeta_{I,N,{\L^M}}=\lam_I\cup\nu_N&=\Delta_S\cup\nu_N\cup\set{
  (x,y)\in I\times I}{ xf\L yf},\\
\rho_\C=\rho_{I,N}=\zeta_{I,N,{\R^M}}=\rho_I\cup\nu_N&=\Delta_S\cup\nu_N\cup\set{
  (x,y)\in I\times I}{ xf\R yf},\\
\mu_\C=\mu_{I,N}=\zeta_{I,N,{\H^M}}=\mu_I\cup\nu_N&=\Delta_S\cup\nu_N\cup\set{ (x,y)\in I\times I}{ xf\H yf},\\
\eta_\C=\eta_{I,N}=\zeta_{I,N,\De_M}=\eta_I\cup\nu_N&=\Delta_S\cup\nu_N\cup\set{ (x,y)\in I\times I}{ xf= yf},
\end{align*}
where $f:I\to M$ is the unique retraction.
\end{defn}

Again, if $|N|=1$, then the relations from Definitions \ref{defn:lrmC0} and
\ref{defn:lrmC} reduce to those from Definitions \ref{defn:lift2} and \ref{defn:lrmI}, respectively.  In what
follows, we will prove that the relations given in Definitions \ref{defn:RC}--\ref{defn:lrmC} are congruences, and discuss their inclusions, meets and
joins.

\begin{lemma}
\label{lem:nu}
Let $S$ be a semigroup with a stable minimal ideal.  If $\C=(I,N)$ is an IN-pair on $S$, and if $(x,y)\in\nu_N$ and~$s\in S$, then $(xs,ys),(sx,sy)\in R_\C$.  Furthermore, if $\C$ is retractable, then ${(xs,ys),(sx,sy)\in\eta_\C}$.
\end{lemma}

\begin{proof}
It suffices to prove the statements for $(xs,ys)$, as the proofs for $(sx, sy)$ are dual.  Denote by $J$ the (stable, regular) $\gJ$-class containing $N$.  Since $(x,y)\in\nu_N$, we may write $(x,y)=(puq,pvq)$ where $u,v\in N$ and $p,q\in S^1$.  
%
%As in the proof of Lemma \ref{lem:stableJD}, since $u\gJ x=puq$, it follows that $u\L pu \R puq=x$.  By \cite[Theorem 2.3]{CPbook},
%%Green's Lemma (see Lemmas~2.2.1--2.2.3 and their proofs, in \cite{Howie}), 
%it follows that the map
%\[
%H_u\to H_x:z\mt pzq
%\]
%is a bijection (recall that $H_u$ and $H_x$ denote the $\H$-classes of $u$ and $x$, respectively).  Consequently, since $v\in H_u$, we have ${y=pvq\in H_x}$.  
%%
Since $x\H y$ (by Lemma \ref{lem:nu_H_equivalence}(i)), it follows in particular that $x\L y$, and so $xs\L ys$.  Thus, we either have $xs, ys \in J$ or else $xs, ys\not\in J$.  If $xs,ys\in J$, then
\[
(xs,ys)=(pu(qs),pv(qs))\in\nu_N \subseteq \eta_\C \sub R_{\C},
\]
completing the proof in this case.
For the remainder of the proof, we assume that $xs,ys\not\in J$.  By the minimality of $J$ in $(S\setminus I)/{\gJ}$,
it follows that $xs,ys\in I$ and so $(xs, ys)\in I\times I \sub R_{\C}$, which
concludes the proof of the first statement.

Suppose from now on that $\C$ is retractable. Denote by $M$ the minimal ideal, and by $f:I\to M$ the retraction.  Since $xs\L ys$, we may write $xs=ays$ and $ys=bxs$ for some $a,b\in S^1$.  But then $(xs)f=(ays)f=a(ys)f$, by Lemma~\ref{lemma-RetractAux}, and similarly $(ys)f=b(xs)f$, so that $(xs)f\L (ys)f$.  Since $M$ is regular (by Lemma \ref{lem:Mcr}), there exists an idempotent $e$ in the $\L$-class of $(xs)f$, and we note that $e$ is a right identity for both $(xs)f$ and $(ys)f$.  Since $u,v\in N$ and $qse\in M$, and since every element of $N$ has the same action on $M$, it follows that $uqse=vqse$. Hence, $(xs)f=(xs)fe=(xse)f=(puqse)f=(pvqse)f=(yse)f=(ys)fe=(ys)f$, and so $(xs,ys)\in \eta_{\C}$, completing the proof.
\end{proof}

\begin{rem}
If $S$ has a stable minimal ideal $M$ (which is regular, by Lemma \ref{lem:Mcr}), and if $N$ is a normal subgroup of any maximal subgroup $G$ contained in $M$, then $\nu_N$ is in fact a liftable congruence, as follows from the proof of Lemma \ref{lem:nu} with $J=M$ (the case in which $xs,ys\not\in J$ does not arise).  The induced congruence $\zeta_{M,\nu_N}=\De_S\cup\nu_N$ from Definition \ref{defn:lift2} satisfies $\eta_M=\De_S\sub\zeta_{M,\nu_N}\sub\zeta_{M,\nu_G}=\mu_M$, and the interval in $\Cong(S)$ from $\De_S$ to $\mu_M$ is isomorphic to the lattice of normal subgroups of $G$.  Since the minimal ideal is $\H$-trivial in all the diagram monoids we consider, such liftable congruences will play no part in this paper.  
\end{rem}

%\begin{rem}\label{rem:TX}
%Stability of $J$ was a crucial ingredient in the proof of Lemma \ref{lem:nu}.  Indeed, consider the full transformation semigroup $\T_X$ on an infinite set $X$.  Let $X=Y\cup Z$, where $Y\cap Z=\emptyset$ and~${|X|=|Y|=|Z|}$.  Let~$J=\set{f\in\T_X}{\rank(f)=|X|}$, and let $N=\S_X$ be the symmetric group on $X$.  So $J$ is a non-stable, regular $\gJ$-class of $\T_X$, and $N$ is a maximal subgroup of $\T_X$ contained in $J$.  We define transformations~$u,v,p,q,s\in\T_X$ as follows.  First, let $u=q=\id_X$ be the identity map on $X$.  Next, let $v$ be any transformation that maps $Y$ bijectively onto $Z$, and $Z$ bijectively onto $Y$.  Fix some element $w\in Y$, and let $p=s$ be the transformation that maps $Y$ identically, and maps all of $Z$ onto $w$.  Then $u,v\in N$ and $(x,y)=(puq,pvq)\in\nu_N$.  However, $xs=p\in J$ but $ys$ is the constant map with image $\{w\}$, so that~$ys\not\in J$, whence $(xs,ys)$ belongs neither to $\nu_N$ nor to $I\times I$, where $I$ is the ideal $\T_X\sm J$.  Consequently, the relation~$R_I\cup\nu_N$ is not a congruence on $\T_X$.
%\end{rem}

\begin{prop}
\label{prop-AreCongs}
Let $S$ be a semigroup with a stable minimal ideal.  If $\C$ is an IN-pair on $S$, then $R_\C$ is a congruence on $S$.  Furthermore, if $\C=(I,N)$ is retractable, and if $\xi$ is a liftable congruence on the minimal ideal of $S$, then $\zeta_{I,N,\xi}$ is also a congruence on $S$.  In particular, $\lambda_\C$, $\rho_\C$, $\mu_\C$ and $\eta_\C$ are all congruences if $\C$ is retractable.
\end{prop}

\pf
Since $R_\C=R_I\cup\nu_N$, and since $R_I$ and $\nu_N$ are equivalences (using Lemma \ref{lem:nu_H_equivalence}(ii) for $\nu_N$), $R_\C$ is symmetric and reflexive.  To show transitivity, suppose $(x,y),(y,z)\in R_\C$.  We must show that $(x,z)\in R_\C$.  As this is obvious if $x=y$ or $y=z$, we assume neither of these hold, so that $(x,y),(y,z)\in (I\times I)\cup\nu_N$.  If $y\in I$, then we must also have $x,z\in I$ (since $\nu_N\sub J\times J$), and so $(x,z)\in I\times I\sub R_\C$.  Otherwise, $y\in J$, and so $(x,y),(y,z)\in\nu_N$; Lemma \ref{lem:nu_H_equivalence}(ii) then gives $(x,z)\in\nu_N\sub R_\C$.  From Lemma \ref{lem:nu} and the fact that $R_I$ is a congruence, it follows immediately that $R_\C=R_I\cup\nu_N$ is a congruence.  

Now suppose $\C$ is retractable, and that $\xi$ is a liftable congruence on the minimal ideal.  We prove that $\zeta_{I,N,\xi}$ is an equivalence in the same way as for $R_\C$.
%; if $(x,y)$ and $(y,z)$ both belong to $\zeta_{I,N,\xi}\sm\De$, then either both pairs belong to $\zeta_{I,\xi}$ or both to $\nu_N$.  
Now let $(x,y)\in\zeta_{I,N,\xi}=\zeta_{I,\xi}\cup\nu_N$ and $s\in S$.  We must show that $(sx,sy),(xs,ys)\in\zeta_{I,N,\xi}$.  This follows from Proposition \ref{prop:lift} if $(x,y)\in\zeta_{I,\xi}$, or from Lemma \ref{lem:nu} if $(x,y)\in\nu_N$, using the fact that~$\eta_\C\sub\zeta_{I,N,\xi}$.  \epf

\begin{rem}
If $N_1,\ldots,N_k$ are normal subgroups of maximal subgroups of $S$ contained in distinct stable, regular $\gJ$-classes that are each minimal in $S\sm I$, then, by a similar argument to that used in the proof of Proposition~\ref{prop-AreCongs}, $R_I\cup\nu_{N_1}\cup\cdots\cup\nu_{N_k}$ is a congruence; a similar comment may also be made concerning the case in which $I$ is retractable.  However, since the $\gJ$-classes form a chain in every diagram monoid we consider, this situation does not arise in our subsequent investigations.
\end{rem}

We now elucidate the relationships between the congruences defined in this section within the lattice~$\Cong(S)$.

\begin{prop}
\label{prop-CR2}
Let $S$ be a semigroup with a stable minimal ideal.  Suppose $\C_1=(I_1,N_1)$ and ${\C_2=(I_2, N_2)}$ are IN-pairs on~$S$, and that $J_1$ and $J_2$ are the $\gJ$-classes of $N_1$
  and $N_2$, respectively. If $I_1\cup J_1\subseteq I_2$, or $I_1=I_2$ and
  $N_1\subseteq N_2$, then $R_{\C_1}\subseteq R_{\C_2}$.
\end{prop}

\pf  If $I_1\cup J_1\subseteq I_2$, then $R_{\C_1}=R_{I_1}\cup \nu_{N_1} \sub R_{I_2}\cup(J_1\times J_1)\sub R_{I_2}\cup(I_2\times I_2)=R_{I_2}\sub R_{\C_2}$.  
If $I_1=I_2$ and $N_1\subseteq N_2$, then $R_{\C_1}=R_{I_1}\cup\nu_{N_1}\sub R_{I_1}\cup\nu_{N_2}=R_{\C_2}$.
\epf

\begin{prop}
\label{prop-CR3}
Let $S$ be a semigroup with a stable minimal ideal.  Suppose $\C_1=(I_1,N_1)$ and ${\C_2=(I_2, N_2)}$ are retractable IN-pairs on $S$, and that $J_1$ and $J_2$ are the $\gJ$-classes of $N_1$ and $N_2$, respectively. If $I_1\cup J_1\subseteq I_2$, or $I_1=I_2$ and $N_1\subseteq N_2$, then
\begin{itemize}
\begin{multicols}{2}
\itemit{i} $\lambda_{\C_1}\subseteq \lambda_{\C_2}$,
$\rho_{\C_1}\subseteq \rho_{\C_2}$,
$\mu_{\C_1}\subseteq \mu_{\C_2}$,
$\eta_{\C_1}\subseteq \eta_{\C_2}$,
\itemit{ii} $\lambda_{\C_1}\cap \rho_{\C_2}=\rho_{\C_1}\cap \lambda_{\C_2}=\mu_{\C_1}$, 
\itemit{iii} $\lambda_{\C_1}\vee \rho_{\C_2}=\rho_{\C_1}\vee \lambda_{\C_2}=R_{\C_2}$,
\itemit{iv} $\lam_{\C_1}\cap\mu_{\C_2} = \rho_{\C_1}\cap\mu_{\C_2} = \mu_{\C_1}$,
\itemit{v} $\lam_{\C_1}\vee\mu_{\C_2} = \lam_{\C_2}$ and $\rho_{\C_1}\vee\mu_{\C_2} = \rho_{\C_2}$,
\itemit{vi} $R_{\C_1}\cap\mu_{\C_2}=\mu_{\C_1}$ and $R_{\C_1}\vee\mu_{\C_2}=R_{\C_2}$.
\end{multicols}
\eit
\end{prop}

\pf
Let $f:I_2\to M$ be the unique retraction.  Since $I_1\sub
  I_2$, the retraction of $I_1$ is the restriction of $f$ to~$I_1$.
  Part (i) follows from the definition of the congruences in a manner
  similar to the proof of Proposition~\ref{prop-CR2}.  For the rest of the proof, we abbreviate $\xi_{\C_i}$ to $\xi_i$, for $i\in\{1,2\}$ and $\xi\in\{R,\lam,\rho,\mu\}$.

For (ii), it is sufficient to prove one of the two assertions: say, $\lambda_1\cap \rho_2=\mu_1$.
Clearly, $\mu_1\subseteq \lambda_1$ and $\mu_1\subseteq \rho_1\subseteq \rho_2$, whence $\mu_1\subseteq \lambda_1\cap\rho_2$.
For the converse inclusion, let $(x,y)\in \lambda_1\cap\rho_2$.
Following the definition of $\lambda_1$, if $x=y$ then clearly $(x,y)\in \mu_1$, while if $(x,y)\in\nu_1$, then $(x,y)\in \mu_1$ since $\nu_1\subseteq\mu_1$.
Finally, suppose $(x,y)\not\in\De_S\cup\nu_1$, so that $x,y\in I_1\sub I_2$ with $xf\L yf$.
Since also $(x,y)\in\rho_2$, it follows that $(x,y)\not\in\De_S\cup\nu_2$, so that $xf\R yf$, and hence $xf\H yf$: i.e., $(x,y)\in\mu_1$.

To prove (iii), it is sufficient to show that $\lambda_1\vee\rho_2=R_2$.
Clearly, $\rho_2\subseteq R_2$ and $\lambda_1\subseteq \lambda_2\subseteq R_2$, whence $\lambda_1\vee\rho_2\subseteq R_2$.
For the converse inclusion, let $(x,y)\in R_2$ be arbitrary.
Following the definition of $R_2$, if $x=y$ then clearly $(x,y)\in \lambda_1\vee\rho_2$, while if $(x,y)\in \nu_2$, then
$(x,y)\in \rho_2\subseteq \lambda_1\vee\rho_2$.
Finally, consider $x,y\in I_2$.
Since $xf,yf\in M$, and since $M$ is a ${\D}$-class (by Lemma \ref{lem:stableJD}), there exists
$z\in M$ such that $xf\L z\R yf$.
From $M\subseteq I_1\subseteq I_2$ and $z=zf$, we have $(x,xf)\in\rho_2$,
$(xf,z)\in \lambda_1$ and $(z,y)\in \rho_2$, implying $(x,y)\in\lambda_1\vee\rho_2$. 

The proof of (iv) is similar to that of (ii), so we omit it.  For (v), it suffices to prove that $\rho_1\vee\mu_2=\rho_2$.  The forwards inclusion is again straightforward.  For the converse, let $(x,y)\in\rho_2$.  Since $\De_S\cup\nu_{N_2}\sub\mu_2\sub\rho_1\vee\mu_2$, it suffices to assume that $x,y\in I_2$ and $xf\R yf$.  If $I_1=I_2$, then it follows that $(x,y)\in\rho_1\sub\rho_1\vee\mu_2$, so suppose instead that $I_1\cup J_1\sub I_2$.  Then, again using $M\sub I_1\sub I_2$, we have $(x,xf)\in\mu_2$, $(xf,yf)\in\rho_1$ and $(yf,y)\in\mu_2$, so that $(x,y)\in\rho_1\vee\mu_2$.

The proof of (vi) is similar to (but easier than) the proofs of (ii) and (v), so we omit it.
\epf

\begin{rem}\label{rem:nonprincipal}
Consider the case that $I_1=I_2$ and $N_1\sub N_2$, in the notation of Proposition \ref{prop-CR3}.  From the proof of part (v), it follows that $\lam_{\C_2}$ and $\rho_{\C_2}$ decompose as \emph{unions}, not just joins:
\[
\lam_{\C_2} = \lam_{\C_1}\cup\mu_{\C_2} \AND \rho_{\C_2} = \rho_{\C_1}\cup\mu_{\C_2}.
\]
It is also easy to see (under the above assumption on $I_1,I_2,N_1,N_2$) that $R_{\C_2}=R_{\C_1}\cup\lam_{\C_2}=R_{\C_1}\cup\rho_{\C_2}$.
\end{rem}

Specialising Proposition \ref{prop-CR3} to the case where $\C_1=\C_2$ we obtain
the following.

\begin{prop}
\label{prop-CR4}
Let $S$ be a semigroup with a stable minimal ideal $M$, and let $\C$ be a retractable IN-pair on~$S$.  Then 
\bit
\itemit{i} $\eta_\C\sub\mu_\C=\lambda_{\C}\cap \rho_{\C}$ and~$R_\C=\lambda_{\C}\vee \rho_{\C}$,
\itemit{ii} $\eta_\C=\mu_\C$ if and only if $M$ is $\H$-trivial,
\itemit{iii} if $M$ has at least two $\R$-classes and at least two $\L$-classes, then the congruences $\lambda_{\C}$ and $\rho_{\C}$ are incomparable in $\Cong(S)$.
\eit
\end{prop}

\pf Part (i) is a direct consequence of Proposition \ref{prop-CR3}.  Parts (ii) and (iii) follow from the fact that the restrictions of $\lambda_\C,\rho_\C,\mu_\C$ to $M$ are equal to the restrictions of $\L,\R,\H$ to~$M$, respectively. \epf

Proposition \ref{prop-CR4} can be interpreted as saying that the congruences
$R_{\C}$, $\lambda_{\C}$, $\rho_{\C}$, and $\mu_{\C}$ corresponding to a
retractable IN-pair form a sublattice inside $\Cong(S)$, as
depicted in Figure \ref{fig-CR} (left).
This diamond structure forms the basic building block from which the
non-chain parts of the specific congruence lattices considered in this paper are
built; see Figures \ref{fig-CongPn}, \ref{fig-CongPPn}, \ref{fig:Hasse_Bn} and
\ref{fig:Hasse_Jn}. Our next result describes how two such diamonds compare to
each other inside $\Cong(S)$; it follows from Propositions~\ref{prop-CR2}--\ref{prop-CR4}.

\begin{prop}
\label{prop-CR4a}
Let $S$ be a semigroup with a stable minimal ideal $M$.  Suppose $\C_1=(I_1,N_1)$ and $\C_2=(I_2, N_2)$ are retractable IN-pairs on $S$, that $J_1$ and $J_2$ are the $\gJ$-classes of $N_1$ and $N_2$, respectively, and that either $I_1\cup J_1\subseteq I_2$, or $I_1=I_2$ and $N_1\subseteq N_2$. If $M$ has at least two $\R$-classes and at least two $\L$-classes, then the eight $R,\lam,\rho,\mu$ congruences arising from $\C_1$ and $\C_2$ (see Definitions \ref{defn:RC} and \ref{defn:lrmC}) form a sublattice of $\Cong(S)$ whose Hasse diagram is depicted in Figure \ref{fig-CR} (right). \epfres
\end{prop}

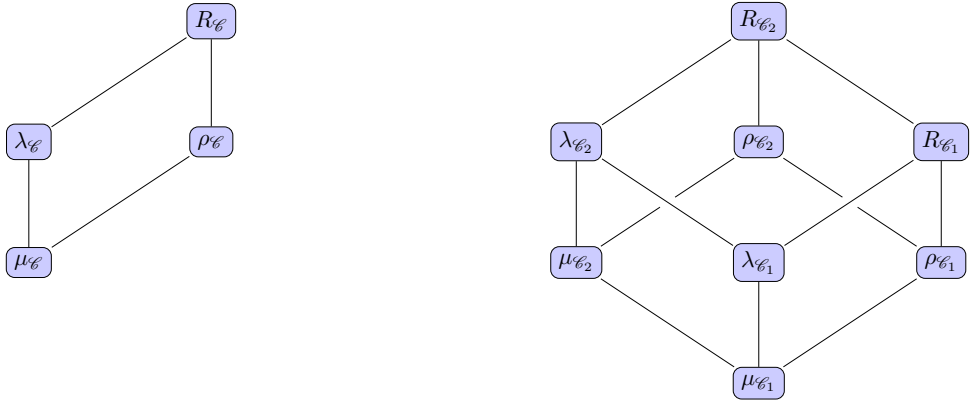
\begin{figure}[ht]
\begin{center}
\scalebox{0.8}{
\begin{tikzpicture}[scale=1]
\node[rounded corners,rectangle,draw,fill=blue!20] (rl) at (6,4) {$R_\C$};
\node[rounded corners,rectangle,draw,fill=blue!20] (r) at (3,2) {$\lambda_\C$};
\node[rounded corners,rectangle,draw,fill=blue!20] (l) at (6,2) {$\rho_{\C}$};
\node[rounded corners,rectangle,draw,fill=blue!20] (d) at (3,0) {$\mu_{\C}$};
\draw
(d)--(l)
(d)--(r)
(l)--(rl)
(r)--(rl)
;
\begin{scope}[shift={(12,-2)}]
\node[rounded corners,rectangle,draw,fill=blue!20] (mrl) at (3,6) {$R_{\C_2}$};
\node[rounded corners,rectangle,draw,fill=blue!20] (mr) at (0,4) {$\lambda_{\C_2}$};
\node[rounded corners,rectangle,draw,fill=blue!20] (ml) at (3,4) {$\rho_{\C_2}$};
\node[rounded corners,rectangle,draw,fill=blue!20] (rl) at (6,4) {$R_{\C_1}$};
\node[rounded corners,rectangle,draw,fill=blue!20] (m) at (0,2) {$\mu_{\C_2}$};
\node[rounded corners,rectangle,draw,fill=blue!20] (r) at (3,2) {$\lambda_{\C_1}$};
\node[rounded corners,rectangle,draw,fill=blue!20] (l) at (6,2) {$\rho_{\C_1}$};
\node[rounded corners,rectangle,draw,fill=blue!20] (d) at (3,0) {$\mu_{\C_1}$};
\draw
(d)--(l) (m)--(ml) 
(d)--(m) (l)--(ml) (rl)--(mrl)
(d)--(r) (m)--(mr)
(l)--(rl) (ml)--(mrl)
;
\fill[white] (1.5,5)circle(.15);
\fill[white] (1.5,3)circle(.15);
\fill[white] (4.5,3)circle(.15);
\draw
(r)--(rl) (mr)--(mrl)  (r)--(mr)
;
\end{scope}
\end{tikzpicture}
}
\caption{Left: the sublattice of $\Cong(S)$ formed by the four congruences
  arising from a retractable IN-pair~$\C$.  Right: the sublattice of
  $\Cong(S)$ consisting of the eight congruences arising from two retractable
  IN-pairs $\C_1,\C_2$ satisfying the conditions of Proposition
  \ref{prop-CR4a}.}
\label{fig-CR}
\end{center}
\end{figure}

Finally, we record the following result, concerning the very least elements of
the lattice; its proof is trivial, and omitted.

\begin{prop}
\label{prop-CR5}
Let $S$ be a semigroup with a stable minimal ideal $M$.  Then the congruences from Definition \ref{defn:lrmI} associated to $M$ are
\[
\lam_M = \De_S\cup {\L^M} \COMMA
\rho_M = \De_S\cup {\R^M} \COMMA
\mu_M = \De_S\cup {\H^M} \COMMA
\eta_M = \De_S.
\]
  In particular, if $M$ is $\H$-trivial, then $\mu_M=\eta_M=\Delta_S$. 
  \epfres
\end{prop}

\begin{rem}
Let $S$ be a semigroup with a stable minimal ideal $M$, and suppose $I$ is a retractable ideal of~$S$.  By Definitions \ref{defn:Rees} and \ref{defn:lrmI}, the congruences $R_I$, $\lambda_I$, $\rho_I$, $\mu_I$ and $\eta_I$ can be associated to $I$.  As noted above, if $S\setminus I$ has a stable, regular, minimal $\gJ$-class, then these congruences coincide with the corresponding congruences from Definitions~\ref{defn:RC} and~\ref{defn:lrmC}, for any trivial group~$N$ contained in any of these~$\gJ$-classes.  But this condition on $\gJ$-classes might fail to hold: namely, if $I=S$, or if all the minimal $\gJ$-classes in~$S\sm I$ are non-stable or non-regular.  For such a semigroup~$S$, one could prove results along the lines of Proposition~\ref{prop-CR4a} to describe the relationships between different collections of congruences.  However, this does not arise in any of the applications we consider, as all of our diagram monoids are regular and finite (so stable) and their maximum retractable ideals are always proper if $n$ is not trivially small, so we omit these details. 
\end{rem}

\begin{rem}\label{rem:infinite}
We end this section with a number of comments related to infinite semigroups.
First note that an infinite semigroup need not have a minimal ideal; consider the natural numbers under addition, for example.
Even if an infinite semigroup has one, the minimal ideal need not be regular; see \cite[Chapter 2, Exercise 1]{Howie} for an example.
Moreover, even if an infinite semigroup has a regular minimal ideal, it need not be stable.  For example, the \emph{bicyclic semigroup}, given by the presentation $\la a,b:ba=1\ra$, is equal to its own minimal ideal; it is regular but not stable, and neither $\L$ nor $\R$ is a (liftable) congruence.

Recall also that the $\nu_N$ relation (Definition \ref{Rel-nu}) was defined in the case that $N$ is contained in a \emph{stable}, regular $\gJ$-class of a semigroup $S$.  
It turns out that stability of this $\gJ$-class (which is not guaranteed if $S$ is infinite) was a crucial ingredient in the proof of Lemma \ref{lem:nu}.  Indeed, consider the full transformation semigroup $\T_X$ on an infinite set $X$.  Let $X=Y\cup Z$, where $Y\cap Z=\emptyset$ and~${|X|=|Y|=|Z|}$.  Let~${J=\set{f\in\T_X}{\rank(f)=|X|}}$, and let $N=\S_X$ be the symmetric group on $X$.  So $J$ is a non-stable, regular $\gJ$-class of~$\T_X$, and $N$ is a maximal subgroup of $\T_X$ contained in $J$.  We define transformations~$u,v,p,q,s\in\T_X$ as follows.  First, let $u=q=\id_X$ be the identity map on $X$.  Next, let $v$ be any transformation that maps $Y$ bijectively onto $Z$, and $Z$ bijectively onto $Y$.  Fix some element $w\in Y$, and let $p=s$ be the transformation that maps~$Y$ identically, and maps all of $Z$ onto $w$.  Then $u,v\in N$ and $(x,y)=(puq,pvq)\in\nu_N$.  However, $xs=p\in J$ but $ys$ is the constant map with image $\{w\}$, so that~$ys\not\in J$, whence $(xs,ys)$ belongs neither to $\nu_N$ nor to $I\times I$, where $I$ is the ideal $\T_X\sm J$.  Consequently, the relation~$R_I\cup\nu_N$ is not a congruence on $\T_X$.
\end{rem}

\section{Congruence lattices of \boldmath{$\S_n$, $\I_n$ and $\O_n$}}
\label{sect:SnInOn}

In this section we present the results of Liber \cite{Liber1953} and Fernandes
\cite{Fernandes2001}, describing the congruence lattices of the symmetric
inverse monoid $\I_n$ and the monoid $\O_n$ of order-preserving
partial permutations of~$\bn$.  To prepare for our subsequent exposition, we
present these results within the framework of the previous section.

First, however, let us recall the situation for the symmetric group $\S_n$, as this will also play a crucial role in many of our investigations.
Congruences on any group $G$ are uniquely determined by normal subgroups of $G$, 
and hence the congruence lattice of $G$ is isomorphic to the lattice of normal subgroups of $G$.  
The normal subgroups of the symmetric group~$\S_n$ are easily described.  They
always form a chain:
\begin{itemize}\begin{multicols}{2}
\item for $n=1$: $\{\id_1\}=\S_1$,
\item for $n=2$: $\{\id_2\}\suq\S_2$,
\item for $n=4$: $\{\id_4\}\suq K\suq\A_4\suq\S_4$,
\item for $n=3$ or $n\geq5$: $\{\id_n\}\suq\A_n\suq\S_n$.
\end{multicols}\eit
where $\A_n$ denotes the \emph{alternating group}, and
$K=\{\id_4,(1,2)(3,4),(1,3)(2,4),(1,4)(2,3)\}$ is the \emph{Klein 4-group}.

We now turn to the symmetric inverse monoid $\I_n$.
Recall that the $\gJ$-classes and ideals of $\I_n$ are the sets
\[
J_q=J_q(\I_n)=\set{\al\in\I_n}{\rank(\al)=q} \AND I_q=I_q(\I_n)=\set{\alpha\in\I_n}{\rank(\alpha)\leq q}, 
\]
respectively, for $q=0,\ldots,n$.
For each ideal $I_q$, we have the Rees congruence $R_{I_q}$ (Definition~\ref{defn:Rees}), which we will denote by $R_q$ for simplicity.
The remaining congruences on $\I_n$ may all be described in terms of IN-pairs (Definition \ref{defn:RC}).  Recall from Lemma \ref{lem:G1G2} that to describe all such pairs, it suffices to identify a single distinguished maximal subgroup from each $\gJ$-class (all $\gJ$-classes of $\I_n$ are regular and stable).  With this in mind, for $q=0,\dots,n$, we define a mapping
\[
\S_q\to \I_n:\sigma\mapsto\overline{\sigma}=\partpermIII{1}\cdots{q}{{1\si}}\cdots{{q\si}},
\]
using the notation of Subsection \ref{sect:prelim_other}.
It is then easy to see that $\overline{\S}_q=\set{\overline\si}{\si\in\S_q}$ is a maximal subgroup of~$\I_n$ contained in $J_q$.
Furthermore, for $q=1,\ldots,n$ and for any normal subgroup $N$ of $\S_q$, the pair $(I_{q-1},\overline N)$ is an IN-pair, and so yields a congruence $R_{I_{q-1},\overline N}$, as given in Definition \ref{defn:RC}; again, for simplicity, we will abbreviate $R_{I_{q-1},\overline N}$ to $R_N$.  Since the minimal ideal $I_0=J_0$ of $\I_n$ has only one element, all ideals are trivially retractable, but no $\lam,\rho,\mu$ congruences arise from any of these IN-pairs.

\begin{thm}[Liber \cite{Liber1953}; see also {\cite[Section 6.3]{GMbook}}]\label{thm:In}
Let $n\geq0$, and let $\I_n$ be the symmetric inverse monoid of degree $n$.  
\bit
\itemit{i} The ideals of $\I_n$ are the sets $I_q$, for $q=0,\ldots,n$, yielding the Rees congruences $R_q=R_{I_q}$, as in Definition \ref{defn:Rees}.
\itemit{ii} The proper IN-pairs of $\I_n$ are of the form $(I_{q-1},\overline N)$, for $q=2,\ldots,n$ and for any non-trivial normal subgroup $N$ of $\S_q$, yielding the congruences $R_N=R_{I_{q-1},\overline N}$, as in Definition \ref{defn:RC}.
\itemit{iii} The above congruences are distinct, and they exhaust all the congruences of $\I_n$.
\itemit{iv} The congruence lattice $\Cong(\I_n)$ forms the following chain:
\[
\epfreseq
\Delta=R_0 \suq R_1 \suq R_{\S_2} \suq R_2 \suq R_{\A_3} \suq R_{\S_3} \suq R_3 \suq R_K \suq R_{\A_4} \suq R_{\S_4}  \suq \cdots \suq R_{\S_n}\suq R_n = \nabla.
\]
\eit
\end{thm}

For the monoid $\O_n$, the situation is even simpler.  Here, we just have the Rees congruences $R_{I_q}$, abbreviated to $R_q$, associated to the ideal $I_q=I_q(\O_n)=\set{\al\in\O_n}{\rank(\al)\leq q}$.

\begin{thm}[Fernandes \cite{Fernandes2001}]
\label{thm:On}
Let $n\geq0$, and let $\O_n$ be the monoid of order-preserving partial permutations of the set $\bn=\{1,\ldots,n\}$.  
\bit
\itemit{i} The ideals of $\O_n$ are the sets $I_q$, for $q=0,\ldots,n$, yielding the Rees congruences $R_q=R_{I_q}$, as in Definition \ref{defn:Rees}.
\itemit{ii} The above congruences are distinct, and they exhaust all the congruences of $\O_n$.
\itemit{iii} The congruence lattice $\Cong(\O_n)$ forms the following chain:
\[
\epfreseq
\Delta = R_0 \suq R_1 \suq R_2 \suq\cdots\suq R_n = \nabla.
\]
\eit
\end{thm}

\begin{rem}
\label{rem:InOnPrinc}
  From the fact that $\Cong(\I_n)$ and $\Cong(\O_n)$ are chains, it follows
  that all congruences on $\I_n$ and $\O_n$ are principal.  Furthermore, a non-trivial
  congruence $\xi$ on $\I_n$ or $\O_n$ is generated by any pair $(\alpha,\beta)$ from the set
  $\xi\setminus\xi^-$, where $\xi^-$ is the maximal congruence properly contained in
  $\xi$.  In actual fact, proving this is a major component in the proofs of
  Theorems~\ref{thm:In} and~\ref{thm:On}.
\end{rem}

\section{The partition monoid \boldmath{$\P_n$}}
\label{sect:Pn}

In this section, we show how the theory developed in Section \ref{sec-fam} can be applied to describe all the congruences of the partition monoid $\P_n$. 
It will serve as a blue-print for all the subsequent sections.  Since $\P_1$ has size $2$, its congruence lattice is easily described, so we assume that $n\geq2$ for the remainder of this section.

The basic strategy is to start with the chain of ideals of $\P_n$,
identify the IN-pairs associated with these ideals, and determine which of them are retractable. 
Propositions \ref{prop-CR2}--\ref{prop-CR5} then enable us to calculate the lattice formed by the congruences arising from these ideals and IN-pairs. Finally, the brunt of the work in this section will be the proof that there are no further congruences on $\P_n$.

From Proposition \ref{prop:green_all_inclusive} we know that the $\gJ$-classes of $\P_n$ are the sets
$J_q=\set{\al\in\P_n}{\rank(\alpha)=q}$, for $q=0,\dots,n$.
The corresponding ideals are $I_q=J_0\cup\dots\cup J_q$,
giving rise to the Rees congruences $R_{I_q}$ (Definition \ref{defn:Rees}), which again will be denoted by $R_q$. 
The minimal ideal is $I_0=J_0$, and it is a rectangular band.

For each proper ideal $I_{q-1}$ ($q=1,\dots,n$) there is a unique minimal $\gJ$-class 
contained in $\P_n\setminus I_{q-1}$: namely,~$J_{q}$. All maximal subgroups contained in $J_q$ are isomorphic to $\S_q$, the symmetric group of rank $q$. 
If for $\sigma\in\S_q$ we let 
$\overline{\sigma}=\partpermIII{1}\cdots{q}{{1\si}}\cdots{{q\si}}$, as in Section \ref{sect:SnInOn}, the mapping $\sigma\mapsto \overline{\sigma}$ is an isomorphism between $\S_q$ and the maximal subgroup $\overline{\S}_q=\set{\overline\si}{\si\in\S_q}$ contained in $J_q$.  

For $q=1,\ldots,n$ and for every normal subgroup $N$ of $\S_q$, the pair $(I_{q-1},\overline{N})$ is an IN-pair. Such a pair yields the congruence~$R_{I_{q-1},\overline N}$, as in Definition \ref{defn:RC}, which throughout this section we will denote simply by~$R_N$; so
\[
R_N=R_{I_{q-1},\overline N} = R_{q-1}\cup\nu_N=\Delta\cup \nu_{N}\cup (I_{q-1}\times I_{q-1}),
\]
where $\De=\De_{\P_n}$ and
\[
\nu_{N}=\nu_{\overline N}=\set{ (\alpha\overline\si\be,\alpha\overline\tau\be)\in J_q\times J_q}{\si,\tau\in N,\ \alpha,\be\in\P_n},
\]
as in Definition \ref{Rel-nu}.
When $N=\{\id_q\}$ is trivial, we obtain the Rees congruence $R_{q-1}$; the \emph{proper} IN-pairs occur when $|N|\geq2$.

The above families of congruences, $R_q$ and $R_N$, closely parallel the situation in  $\Cong(\I_n)$
as described in Section \ref{sect:SnInOn}
(and also that in $\Cong(\T_n)$ originally discovered by Mal$'$cev).
In particular, they form a chain, by Proposition~\ref{prop-CR2}.
However,~$\P_n$ also has some $\lambda,\rho,\mu$ congruences, arising from a retraction that we now describe.

\begin{defn}
\label{defn:hat}
For $\al\in\P_n$, let $\widehat{\alpha}$ be the unique partition of rank $0$ with the same kernel and cokernel as~$\alpha$.~
\end{defn}

\begin{lemma}
\label{HatRetract}
The map $f: I_1\rightarrow I_0:\al\mt\alh$ is a retraction.
\end{lemma}

\pf
Clearly $f$ acts identically on $I_0$.  It remains to prove that $\alh\beh=\widehat{\al\be}$ for all ${\al,\be\in I_1=J_0\cup J_1}$.  This is clearly the case if ${\al,\be\in J_0}$, so suppose without loss of generality that $\al\in J_1$ and $\be\in I_1$.  We may write $\al=\partIV{A_0}{A_1}\cdots{A_r}{B_0}{B_1}\cdots{B_s}$ and $\be=\partIV{C_0}{C_1}\cdots{C_t}{D_0}{D_1}\cdots{D_u}$ or $\partIII{C_0}{C_1}\cdots{C_t}{D_0}{D_1}\cdots{D_u}$.  Then $\alh=\partIII{A_0}{A_1}\cdots{A_r}{B_0}{B_1}\cdots{B_s}$ and $\beh = \partIII{C_0}{C_1}\cdots{C_t}{D_0}{D_1}\cdots{D_u}$, while $\al\be$ is equal to either $\partIV{A_0}{A_1}\cdots{A_r}{D_0}{D_1}\cdots{D_u}$ or $\partIII{A_0}{A_1}\cdots{A_r}{D_0}{D_1}\cdots{D_u}$.  The identity $\alh\beh=\widehat{\al\be}$ quickly follows. \epf

%\begin{rem}\label{rem:retractable_ideals_Pn}
The ideal $I_1$ is in fact the \emph{largest} retractable ideal of $\P_n$.  It is not hard to prove this directly, but it also follows from Theorem \ref{thm-CongPn}.  Indeed, if any larger ideal was retractable, then such an ideal would lead to additional $\lam,\rho,\mu$ congruences, as in Definition \ref{defn:lrmI}.
%\end{rem}

All IN-pairs associated to the retractable ideals $I_0,I_1$ of $\P_n$ are retractable, as we now show.

\begin{lemma}
\label{ThreeTriples}
The IN-pairs $\C_0=(I_0,\{\overline{\id}_1\})$,
$\C_1=(I_1,\{\overline{\id}_2\})$ and
$\C_{\S_2}=(I_1,\overline{\S}_2)$ are all retractable.
\end{lemma}

\pf
By Lemma \ref{HatRetract}, the stated ideals are all retractable, so it remains to prove that for every $\alpha\in I_0$, we have
$|\alpha \overline{\S}_2|=1=|\overline{\S}_2\alpha|$. 
In fact, it suffices to prove the first equality, the second being dual.
Recall that $\overline{\S}_2=\{\beta,\gamma\}$,
where $\beta=\partpermII1212$ and $\gamma=\partpermII1221$.
Suppose $\alpha=\partIII{A_1}{A_2}\cdots{A_r}{B_1}{B_2}\cdots{B_s}$.
If $(1,2)\in\coker(\alpha)$, then
$\alpha\beta=\partIII{A_1}{A_2}\cdots{A_r}{1,2}{3}\cdots{n}=\alpha\gamma$; otherwise,
$\alpha\beta=\partIII{A_1}{A_2}\cdots{A_r}{1}{2}\cdots{n}=\alpha\gamma$.
\epf

To each of these three pairs, we can associate the $\lambda,\rho,\mu$ congruences given in Definition \ref{defn:lrmC} (but note that the $\eta$ and $\mu$ congruences are equal in each case, since $I_0=J_0$ is $\H$-trivial).  
To simplify the notation, we will abbreviate these to $\lam_0=\lam_{\C_0}$, $\lam_1=\lam_{\C_1}$, $\lam_{\S_2}=\lam_{\C_{\S_2}}$, with similar notation also for the $\rho$ and $\mu$ congruences.
We can now state the main result of this section.

%\newpage

\begin{thm}\label{thm-CongPn}
Let $n\geq2$, and let $\P_n$ be the partition monoid of degree $n$.  
\bit
\itemit{i} The ideals of $\P_n$ are the sets $I_q$, for $q=0,\ldots,n$, yielding the Rees congruences $R_q=R_{I_q}$, as in Definition \ref{defn:Rees}.
\itemit{ii} The proper IN-pairs of $\P_n$ are of the form $(I_{q-1},\overline N)$, for $q=2,\ldots,n$ and for any non-trivial normal subgroup $N$ of $\S_q$, yielding the congruences $R_N=R_{I_{q-1},\overline N}$, as in Definition \ref{defn:RC}.
\itemit{iii} The retractable IN-pairs of $\P_n$ are $\C_0=(I_0,\{\overline{\id}_1\})$, $\C_1=(I_1,\{\overline{\id}_2\})$ and $\C_{\S_2}=(I_1,\overline{\S}_2)$, yielding the congruences $\lambda_0$, $\rho_0$, $\mu_0$, $\lambda_1$, $\rho_1$, $\mu_1$, $\lambda_{\S_2}$, $\rho_{\S_2}$, $\mu_{\S_2}$, respectively, as in Definition \ref{defn:lrmC}.
\itemit{iv} The above congruences are distinct, and they exhaust all the congruences of $\P_n$.
\itemit{v} The Hasse diagram of the congruence lattice $\Cong(\P_n) $ is shown in Figure \ref{fig-CongPn}.
\eit
\end{thm}

\pf
That all the listed relations are congruences follows from Proposition \ref{prop-AreCongs}, and that they form the lattice depicted in Figure \ref{fig-CongPn} follows from Propositions \ref{prop-CR2}--\ref{prop-CR5}.
What therefore remains to be proved is that they exhaust all the congruences of $\P_n$.
By Lemma \ref{lem:principal}, this can be accomplished by showing that our list contains all the principal congruences. 
The principal congruences are shaded blue in Figure \ref{fig-CongPn}, and below we will for each of them describe all generating pairs. 
The proof is then completed by observing that this covers all possible pairs, as summarised in Table \ref{PnCongGens}.
\epf

\begin{figure}[ht]
\begin{center}
\scalebox{0.8}{
\begin{tikzpicture}[scale=.97]
\dottedinterval0{16}11
\node[rounded corners,rectangle,draw,fill=blue!20] (N) at (0,19) {$R_n$};
  \draw(0.87,18.98) node {$=\nabla$};
\node[rounded corners,rectangle,draw,fill=blue!20] (R3) at (0,16) {$R_3$};
\node[rounded corners,rectangle,draw,fill=blue!20] (S3) at (0,14) {$R_{\S_3}$};
\node[rounded corners,rectangle,draw,fill=blue!20] (A3) at (0,12) {$R_{\A_3}$};
\node[rounded corners,rectangle,draw,fill=blue!20] (R2) at (0,10) {$R_2$};
\node[rounded corners,rectangle,draw,fill=green!20] (krl) at (0,8) {$R_{\S_2}$};
\node[rounded corners,rectangle,draw,fill=green!20] (kr) at (-3,6) {$\lambda_{\S_2}$};
\node[rounded corners,rectangle,draw,fill=green!20] (kl) at (0,6) {$\rho_{\S_2}$};
\node[rounded corners,rectangle,draw,fill=blue!20] (mrl) at (3,6) {$R_1$};
\node[rounded corners,rectangle,draw,fill=blue!20] (k) at (-3,4) {$\mu_{\S_2}$};
\node[rounded corners,rectangle,draw,fill=blue!20] (mr) at (0,4) {$\lambda_1$};
\node[rounded corners,rectangle,draw,fill=blue!20] (ml) at (3,4) {$\rho_1$};
\node[rounded corners,rectangle,draw,fill=blue!20] (rl) at (6,4) {$R_0$};
\node[rounded corners,rectangle,draw,fill=blue!20] (m) at (0,2) {$\mu_1$};
\node[rounded corners,rectangle,draw,fill=blue!20] (r) at (3,2) {$\lambda_0$};
\node[rounded corners,rectangle,draw,fill=blue!20] (l) at (6,2) {$\rho_0$};
\node[rounded corners,rectangle,draw,fill=blue!20] (d) at (3,0) {$\mu_0$};
   \draw(3.8,0.05) node {$=\Delta$};
\draw
(d)--(l) (m)--(ml) (k)--(kl)
(d)--(m)--(k) (l)--(ml)--(kl) (rl)--(mrl)--(krl)
(d)--(r) (m)--(mr) (k)--(kr)
(l)--(rl) (ml)--(mrl) (kl)--(krl)
(krl)--(R2)
(R2)--(A3)--(S3)--(R3)
;
\fill[white] (-1.5,5)circle(.15);
\fill[white] (1.5,5)circle(.15);
\fill[white] (1.5,3)circle(.15);
\fill[white] (4.5,3)circle(.15);
\draw
(r)--(rl) (mr)--(mrl) (kr)--(krl) (r)--(mr)--(kr) 
;
\begin{scope}[shift={(12,0)}]
\dottedinterval0{16}11
\node[rounded corners,rectangle,draw,fill=blue!20] (A3) at (0,12) {$R_{\A_3}$};
\node[rounded corners,rectangle,draw,fill=blue!20] (S3) at (0,14) {$R_{\S_3}$};
\node[rounded corners,rectangle,draw,fill=blue!20] (R3) at (0,16) {$R_3$};
\node[rounded corners,rectangle,draw,fill=blue!20] (N) at (0,19) {$R_n$};
  \draw(0.87,18.98) node {$=\nabla$};
\node[rounded corners,rectangle,draw,fill=blue!20] (mrl) at (3,6) {$R_1$};
\node[rounded corners,rectangle,draw,fill=blue!20] (rl) at (6,4) {$R_0$};
\node[rounded corners,rectangle,draw,fill=blue!20] (m) at (0,4) {$\mu_1$};
\node[rounded corners,rectangle,draw,fill=green!20] (krl) at (0,8) {$R_{\S_2}$};
\node[rounded corners,rectangle,draw,fill=blue!20] (d) at (3,2) {$\mu_0$};
  \draw(3.87,1.98) node {$=\Delta$};
\node[rounded corners,rectangle,draw,fill=blue!20] (k) at (-3,6) {$\mu_{\S_2}$};
\node[rounded corners,rectangle,draw,fill=blue!20] (R2) at (0,10) {$R_2$};
\draw
(d)--(m)--(k) 
(k)--(krl) (m)--(mrl) (d)--(rl)
(rl)--(mrl)--(krl)
(krl)--(R2)
(R2)--(A3)--(S3)--(R3)
;
\end{scope}
\end{tikzpicture}
}
\vspace{-3mm}
\caption{
Hasse diagrams for the congruence lattice (left) and the $*$-congruence lattice (right) of the partition monoid $\P_n$.  Congruences shaded blue are principal;  those shaded green are minimally generated by two pairs of partitions.  
The diagrams also serve for the partial Brauer monoid $\PB_n$ (see Section \ref{sect:PBn}).
}
\label{fig-CongPn}
\end{center}
\end{figure}
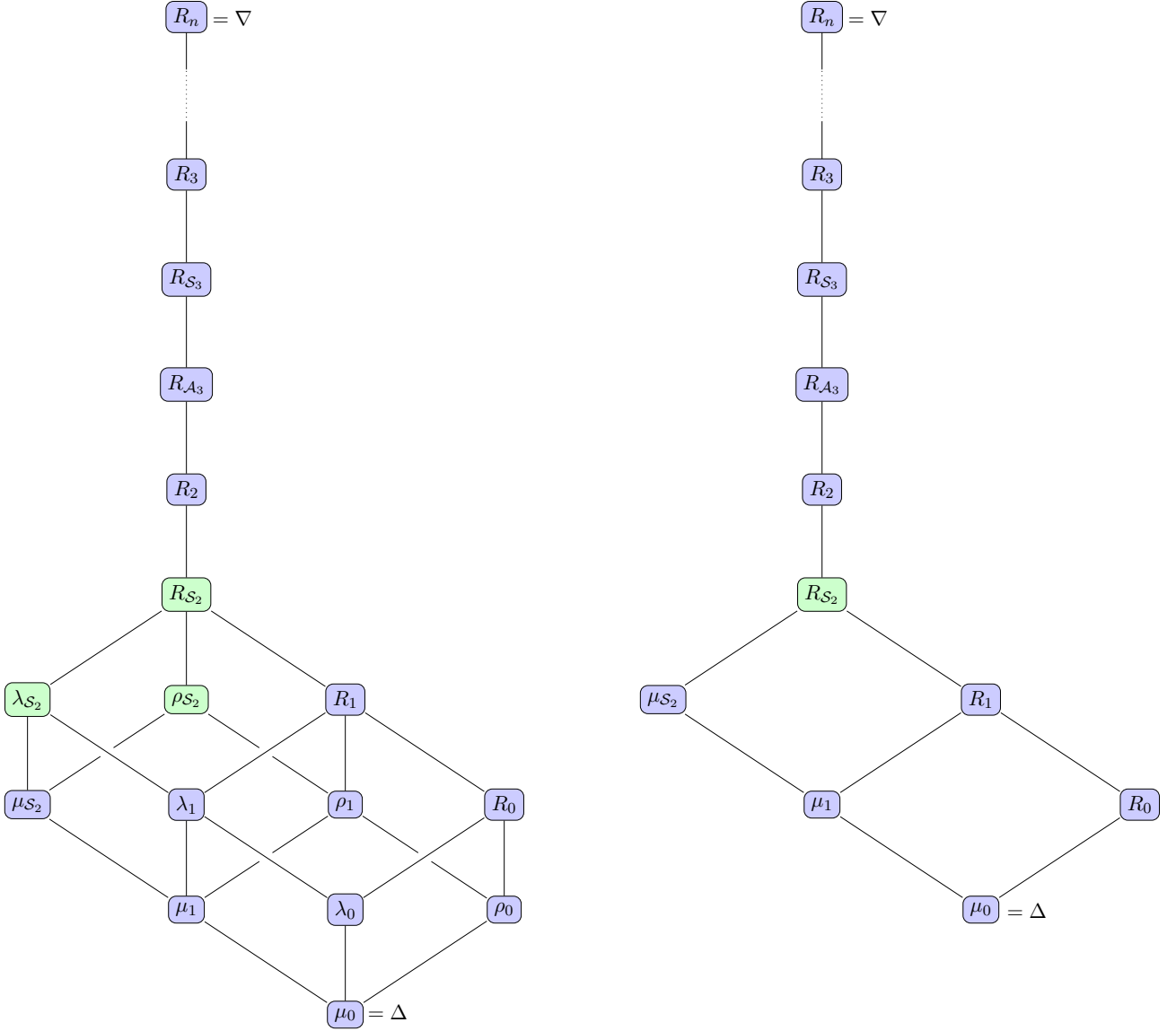

\begin{table}[ht]
\begin{center}
\begin{tabular}{|r|l|c|l|} \hline
\multicolumn{1}{|c|}{\boldmath{$q$}} & \textbf{Additional conditions} & \boldmath{$(\alpha,\beta)^\sharp$} & \textbf{Reference} \\ \hline\hline
&$\alpha=\beta$ & $\Delta$ & \\ \hline
0& $(\alpha,\beta)\in{\R}\sm\Delta$ &
$\rho_0$ & Proposition \ref{prop:small_congruences}(i)\\ \hline
0& $(\alpha,\beta)\in{\L}\sm\Delta$ &
$\lambda_0$ & Proposition \ref{prop:small_congruences}(ii)\\ \hline
0& $(\alpha,\beta)\not\in{\L}\cup{\R}$ &
$R_0$ & Proposition \ref{prop:joins2}(i)\\ \hline
1& $\alh=\beh$, $\al\not=\be$ &
$\mu_1$ & Proposition \ref{prop:small_congruences}(iii)\\ \hline
1& $(\widehat{\alpha},\widehat{\beta})\in{\R}\sm\De$ &
$\rho_1$ & Proposition \ref{prop:joins2}(ii)\\ \hline
1& $(\widehat{\alpha},\widehat{\beta})\in{\L}\sm\De$ &
$\lambda_1$ & Proposition \ref{prop:joins2}(iii)\\ \hline
1& $(\widehat{\alpha},\widehat{\beta})\not\in{\L}\cup{\R}$ &
$R_1$ & Proposition \ref{prop:joins2}(iv)\\ \hline
2& $(\alpha,\beta)\in{\H}\sm\Delta$ &
$\mu_{\S_2}$ & Proposition \ref{prop:small_congruences}(iv)\\ \hline
$\geq2$&  $(\alpha,\beta)\not\in{\H}$ & $R_q$ & Proposition \ref{prop-aa3}\\ \hline
$\geq3$&  $(\alpha,\beta)\in{\H}\setminus \Delta$ & $R_N$ & Proposition \ref{prop-aa3}\\ \hline
\end{tabular}

\caption{The principal congruences $(\alpha,\beta)^\sharp$ on $\P_n$, 
with $q=\rank(\alpha)\geq\rank(\beta)$.
In the last row, $N$ stands for the normal closure in $\S_q$ of $\phi(\alpha,\beta)$.}
\label{PnCongGens}
\end{center}
\end{table}

As noted in the above proof, the rest of this section is devoted to describing the generating pairs for the principal congruences on $\P_n$.
Throughout, we will refer to the congruences forming the ``prism shaped'' part of the congruence lattice (i.e., $R_{\S_2}$ and all the congruences contained in it) as the \emph{lower} congruences, and those forming the ``wick'' (i.e., congruences containing $R_{\S_2}$) as the \emph{upper congruences}.

Before we move on, it will be convenient to give an alternative description of the relations $\nu_N$ on~$\P_n$ from Definition \ref{Rel-nu}.  To do so, we first make the following definition.

\begin{defn}\label{defn:phi}
Let $\al,\be\in J_q$ with $\al\H\be$ and $q\geq1$, and suppose the kernel-classes of $\al$, and hence also of~$\be$, are $A_1,\ldots,A_q,C_1,\ldots,C_r$, where $\dom(\al)=\dom(\be)=A_1\cup\cdots\cup A_q$ and $\min(A_1)<\cdots<\min(A_q)$.  Then
$\al\be^* = \partI{A_1}{A_q}{C_1}{C_r}{A_{1\si}}{A_{q\si}}{C_1}{C_r}$ for some permutation $\si\in\S_q$, and we define $\phi(\al,\be)=\si$.  
\end{defn}

We now establish the connection between the relations $\nu_N$ and the permutations $\phi(\al,\be)$ from Definitions~\ref{Rel-nu} and \ref{defn:phi}.  In what follows, we will generally use the next result without explicit reference.

\begin{lemma}\label{lem:nu_phi}
If $q\geq1$, and if $N$ is a normal subgroup of $\S_q$, then
\[
\nu_N=\set{ (\alpha,\beta)\in J_q\times J_q}{\alpha\H \beta,\ \phi(\alpha,\beta)\in N}.
\]
\end{lemma}

\pf Let $\Om=\set{ (\alpha,\beta)\in J_q\times J_q}{\alpha\H \beta,\
\phi(\alpha,\beta)\in N}$ be the set in the statement of the lemma.  First
suppose $(\al,\be)\in\Om$, and write $\al=\partI{A_1}{A_q}{C_1}{C_r}{B_1}{B_q}{D_1}{D_s}$ where $\min(A_1)<\cdots<\min(A_q)$.  Since $\al\H\be$, we have $\be=\partI{A_1}{A_q}{C_1}{C_r}{B_{1\si}}{B_{q\si}}{D_1}{D_s}$ for some $\si\in\S_q$, and it is clear that $\si=\phi(\al,\be)^{-1}\in N$.  But then $(\al,\be)=(\ga\overline{\id}_q\de,\ga\overline{\si}\de)$, where $\ga=\partI{A_1}{A_q}{C_1}{C_r}{1}{q}{q+1}{n}$ and $\de=\partI{1}{q}{q+1}{n}{B_1}{B_q}{D_1}{D_s}$, giving $(\al,\be)\in\nu_N$.

Conversely, suppose $(\al,\be)\in\nu_N$, so that $(\al,\be)=(\ga\overline\si\de,\ga\overline\tau\de)\in J_q\times J_q$ for some $\ga,\de\in\P_n$ and ${\si,\tau\in N}$.  Since $(\al,\be)=(\ga\cdot\overline{\id}_q\cdot\overline\si\de,\ga\cdot\overline\tau\;\!\overline\si^*\cdot\sib\de)$, with $\overline\tau\;\!\overline\si^*\in \overline N$, it suffices to assume that $\si=\id_q$.  So ${(\al,\be)=(\ga\overline\id_q\de,\ga\overline\tau\de)}$.
Since $\al\in J_q$, we may write $\al=\partI{A_1}{A_q}{C_1}{C_r}{B_1}{B_q}{D_1}{D_s}$.  As noted in the proof of Proposition~\ref{prop-AreCongs}, $\al\H\be$, so we may write $\be=\partI{A_1}{A_q}{C_1}{C_r}{B_{1\theta}}{B_{q\theta}}{D_1}{D_s}$ for some $\theta\in\S_q$.
For each $1\leq i\leq q$, choose some $a_i\in A_i$ and $b_i\in B_i$.  
Since $\idb_q$ and $\al=\ga\idb_q\de$ both belong to $J_q$, it follows that $\ga\idb_q$ and $\idb_q\de$ do as well.  Thus, renaming the sets $A_i$ if necessary, we may assume that there is a path $\pi_i$ in the product graph $\Pi(\ga,\overline{\id}_q)$ from $a_i$ to $i''$, and a path $\pi_i'$ in $\Pi(\overline{\id}_q,\de)$ from $i''$ to $b_i'$.  
Since $\ker(\idb_q)=\coker(\idb_q)$ is trivial, these paths involve only edges from $\ga$ or from $\de$, respectively.  In particular, the paths $\pi_i$ and $\pi_i'$ are also in the product graphs $\Pi(\ga,\taub)$ and $\Pi(\taub,\de)$, respectively, for each $i$.  It then follows that $a_i$ and $b_{i\tau}'$ belong to the same block of $\be=\ga\overline\tau\de$: i.e., that $\theta=\tau$.  Now let $\om\in\S_q$ be such that $\min(A_{1\om})<\cdots<\min(A_{q\om})$, and put $E_i=A_{i\om}$ for each $i$.  Then
\[
\al=\partI{E_1}{E_q}{C_1}{C_r}{B_{1\om}}{B_{q\om}}{D_1}{D_s} \COMMA
\be=\partI{E_1}{E_q}{C_1}{C_r}{B_{1\om\tau}}{B_{q\om\tau}}{D_1}{D_s} \COMMA
\al\be^*=\partI{E_1}{E_q}{C_1}{C_r}{E_{1\om\tau^{-1}\om^{-1}}}{E_{q\om\tau^{-1}\om^{-1}}}{C_1}{C_r},
\]
so that $\phi(\al,\be)=\om\tau^{-1}\om^{-1}\in N$. \epf

For ease of reference in the proofs that follow, it will be convenient to list the explicit criteria for membership in each of the lower congruences (these follow from Definitions \ref{defn:RC}, \ref{defn:lrmC}, \ref{defn:hat}, Proposition \ref{prop:green_all_inclusive} and Lemma \ref{lem:nu_phi}):
\newpage
\begin{itemize}\begin{multicols}{2}
\item $\rho_0=\De \cup \set{(\al,\be)\in I_0\times I_0}{\ker(\al)=\ker(\be)}$,

\item $\rho_1=\De \cup \set{(\al,\be)\in I_1\times I_1}{\ker(\al)=\ker(\be)}$,

\item $\rho_{\S_2}=\rho_1\cup\big( {\H}\cap(J_2\times J_2)\big)$,

\item $\mu_1= \De \cup \set{(\al,\be)\in I_1\times I_1}{\alh=\beh}$,

\item $\lam_0=\De \cup \set{(\al,\be)\in I_0\times I_0}{\coker(\al)=\coker(\be)}$,

\item $\lam_1=\De \cup \set{(\al,\be)\in I_1\times I_1}{\coker(\al)=\coker(\be)}$,

\item $\lam_{\S_2}=\lam_1\cup\big( {\H}\cap(J_2\times J_2)\big)$,

\item $\mu_{\S_2}=\mu_1\cup\big( {\H}\cap(J_2\times J_2)\big)$.
\end{multicols}\eit
We will repeatedly refer to Liber's classification of congruences on the symmetric inverse monoid $\I_n$ as presented in Theorem \ref{thm:In}.
Whenever there is notational conflict between congruences on $\P_n$ and on $\I_n$, we will resolve it by placing bars above the latter.
Thus, for example, for a normal subgroup $N$ of $\S_q$, we will write $R_N$ for the congruence it gives rise to on $\P_n$, and $\overline{R}_N=R_N\cap (\I_n\times\I_n)$ for its $\I_n$ analogue.

\begin{lemma}\label{lem:mn}
Let $\xi\in\Cong(\P_n)$.
\begin{itemize}
\begin{multicols}{2}
\itemit{i}  If $\Rb_1\sub\xi$, then $\mu_1\sub\xi$.  Thus, $\mu_1=\Rb_1^\sharp$.
\itemit{ii}  If $\Rb_{\S_2}\sub\xi$, then $\mu_{\S_2}\sub\xi$.  Thus, $\mu_{\S_2}=\Rb_{\S_2}^\sharp$.
\end{multicols}
\end{itemize}
\end{lemma}

\pf To prove (i), suppose $\Rb_1\sub\xi$.  Let $\al\in\P_n$ with $\rank(\al)\leq1$.  
Consulting the above membership criterion for $\mu_1$, it suffices to show that
$(\al,\alh)\in\xi$.  This is clearly the case if $\rank(\al)=0$, so suppose $\rank(\al)=1$.  Write $\al=\partIV{A_0}{A_1}\cdots{A_r}{B_0}{B_1}\cdots{B_s}$, so that $\alh=\partIII{A_0}{A_1}\cdots{A_r}{B_0}{B_1}\cdots{B_s}$, and define $\ga=\partIV{A_0}{A_1}\cdots{A_r}{1}{2}\cdots{n}$ and $\de=\partIV{1}{2}\cdots{n}{B_0}{B_1}\cdots{B_s}$.
Also define elements $\si,\tau\in\I_n\sub\P_n$ by $\si=\partpermI11$ and $\tau=\emptypartperm$, noting that $(\si,\tau)\in\Rb_1\sub\xi$.  It then follows that $\al=\ga\si\de \mathrel\xi \ga\tau\de = \alh$, completing the proof of the first assertion.  To prove the second assertion, since certainly $\Rb_1\sub\Rb_1^\sharp$, it follows from the first assertion that $\mu_1\sub\Rb_1^\sharp$.  But it is also clear that $\Rb_1\sub\mu_1$, so since $\mu_1$ is a congruence, it follows that $\Rb_1^\sharp\sub\mu_1$.

To prove (ii), suppose $\Rb_{\S_2}\sub\xi$. Note that $\mu_{\S_2}=\mu_1\cup(\nu_{\S_2}\sm\De)$.
 Since $\Rb_1\sub\Rb_{\S_2}\sub\xi$, part (i) gives $\mu_1\sub\xi$.  
Now let $(\al,\be)\in\nu_{\S_2}\sm\De$.  Since $\rank(\al)=\rank(\be)=2$ and $\al\H\be$, we may write $\al=\partV{A_1}{A_2}{C_1}{C_r}{B_1}{B_2}{D_1}{D_s}$ and $\be=\partV{A_1}{A_2}{C_1}{C_r}{B_2}{B_1}{D_1}{D_s}$.  Define $\ga=\partV{A_1}{A_2}{C_1}{C_r}{1}{2}{3}{n}$ and $\de=\partV{1}{2}{3}{n}{B_1}{B_2}{D_1}{D_s}$.
Also put $\si=\partpermII1212$ and $\tau=\partpermII1221$, noting that $(\si,\tau)\in\Rb_{\S_2}\sub\xi$.  It then follows that $\al=\ga\si\de \mathrel\xi \ga\tau\de = \be$.  The second assertion of~(ii) is proved in a similar way to the second assertion of (i).  \epf

We can now describe the generating pairs for the congruences $\lambda_0$, $\rho_0$, $\mu_1$ and $\mu_{\S_2}$.  As can be seen from Figure \ref{fig-CongPn}, these are the only non-trivial lower congruences on $\P_n$ that are not joins of proper sub-congruences.

\begin{prop}\label{prop:small_congruences}
The relations $\rho_0,\lam_0,\mu_1,\mu_{\S_2}$ are all principal congruences on $\P_n$.  Moreover, if $\al,\be\in\P_n$, then
\begin{itemize}\begin{multicols}{2}
\itemit{i} $\rho_0=\cg{\al}{\be}\iff (\al,\be)\in\rho_0\sm\De$,
\itemit{ii} $\lam_0=\cg{\al}{\be}\iff (\al,\be)\in\lam_0\sm\De$,
\itemit{iii} $\mu_1=\cg{\al}{\be}\iff (\al,\be)\in\mu_1\sm\De$, 
\itemit{iv} $\mu_{\S_2}=\cg{\al}{\be}\iff (\al,\be)\in\mu_{\S_2}\sm\mu_1$.
\end{multicols}\eit
\end{prop}

\pf The ``only if'' parts are clear for each statement.
Before we prove (i), we note that (ii) will then follow by duality.  Indeed, given (i),
\[
\lam_0=\cg\al\be \iff \rho_0=\cg{\al^*}{\be^*} \iff (\al^*,\be^*)\in\rho_0\sm\De \iff (\al,\be)\in\lam_0\sm\De.
\]
To prove (i), 
fix some $(\al,\be)\in\rho_0\sm\De$ and, for simplicity, write ${\mathrel\xi}=\cg{\al}{\be}$.  For $\si\in\P_n$ with $\rank(\si)=0$, write $\si'$ for the unique partition of rank~$0$ with $\ker(\si')=\ker(\si)$ and such that $\coker(\si')$ is the trivial relation on $\bn$.  
We first claim that $\si\mathrel\xi\si'$ for all $\si\in\P_n$ with $\rank(\si)=0$.  We prove the claim by descending induction on $r$, the number of cokernel-classes of~$\si$.  If $r=n$, then $\si=\si'$, so there is nothing to prove.  So suppose
$r\leq n-1$, and write $\si=\partII{A_1}\cdots{A_q}{B_1}\cdots{B_r}$.
Since $\rank(\al)=\rank(\be)=0$ and $\ker(\al)=\ker(\be)$, yet $\al\not=\be$, it follows that $\coker(\al)\not=\coker(\be)$.  Reversing the roles of $\al$ and $\be$, if necessary, we may assume that there exists $(i,j)\in\coker(\al)\sm\coker(\be)$.  
Write $\bn\sm\{i,j\}=\{k_1,\ldots,k_{n-2}\}$.
Since $r\leq n-1$, we may assume that $|B_1|\geq2$.  Fix some $m\in B_1$, put $C=B_1\sm\{m\}$, and define $\tau=\partV ij{k_1}{k_{n-2}}{m}{C}{B_2}{B_r}$.
Then one easily checks that $\si\al\tau = \si$ and $\si\be\tau = \partXIX{A_1}{A_2}{A_3}{A_q}{m}{C}{B_2}{B_r}$.  It follows that $\si=\si\al\tau\mathrel\xi\si\be\tau$.  But $\rank(\si\be\tau)=0$, $\ker(\si\be\tau)=\ker(\si)$, and $\si\be\tau$ has $r+1$ cokernel-classes so, by induction, $\si'\mathrel\xi\si\be\tau\mathrel\xi\si$.  This completes the proof of the claim.  
Returning to the main proof, suppose $(\ga,\de)\in\rho_0$ is arbitrary.  If $\ga=\de$, then clearly $\ga\mathrel\xi\de$.  If $\ga\not=\de$, then $\rank(\ga)=\rank(\de)=0$ and $\ker(\ga)=\ker(\de)$, so the claim gives $\ga\mathrel\xi\ga'=\de'\mathrel\xi\de$. 

To prove (iii), let $(\al,\be)\in\mu_1\sm\De$.
Renaming $\al,\be$ if necessary, we may assume that $\rank(\al)=1$, and we also have $\rank(\be)\leq1$, $\ker(\al)=\ker(\be)$, $\coker(\al)=\coker(\be)$ but $\al\not=\be$.  Write $\al=\partIV{A_0}{A_1}\cdots{A_r}{B_0}{B_1}\cdots{B_s}$.  Then, renaming $\ker(\al)$-classes if necessary, one of the following holds:
\begin{itemize}\begin{multicols}{2}
\item[(a)] $\be=\partIII{A_0}{A_1}\cdots{A_r}{B_0}{B_1}\cdots{B_s}$,
\item[(b)] $\be = \partIV{A_r}{A_0}{\cdots}{A_{r-1}}{B_0}{B_1}{\cdots}{B_s}$ with $r\not=0$, 
\item[(c)] $\be = \partIV{A_0}{A_1}{\cdots}{A_r}{B_s}{B_0}{\cdots}{B_{s-1}}$ with $s\not=0$, or
\item[(d)] $\be = \partIV{A_r}{A_0}{\cdots}{A_{r-1}}{B_s}{B_0}{\cdots}{B_{s-1}}$ with $r,s\not=0$.
\end{multicols}\eit
Choose some $i\in A_0$ and $j\in B_0$, and put $\si=\partpermI ii$ and $\tau=\partpermI jj$.  In all cases, we have $\si\al\tau=\partpermI ij$ and $\si\be\tau = \emptypartperm$.  In particular, $(\si\al\tau,\si\be\tau)\in\Rb_1\sm\Deb$, so Theorem \ref{thm:In} (and Remark \ref{rem:InOnPrinc}) gives $\Rb_1=\cgI{\si\al\tau}{\si\be\tau}\sub\cg{\si\al\tau}{\si\be\tau}\sub\cg\al\be$.  Lemma~\ref{lem:mn}(i) then gives $\mu_1\sub\cg\al\be$, and the reverse containment is clear.

To prove (iv), let $(\al,\be)\in\mu_{\S_2}\sm\mu_1=\nu_{\S_2}\sm\Delta$, and write $\al=\partV{A_1}{A_2}{C_1}{C_r}{B_1}{B_2}{D_1}{D_s}$ and $\be=\partV{A_1}{A_2}{C_1}{C_r}{B_2}{B_1}{D_1}{D_s}$.  Choose any $a_i\in A_i$ and $b_i\in B_i$ ($i=1,2$), and put $\si=\partpermII{a_1}{a_2}{a_1}{a_2}$ and $\tau=\partpermII{b_1}{b_2}{b_1}{b_2}$.  Then $\si\al\tau=\partpermII{a_1}{a_2}{b_1}{b_2}$ and $\si\be\tau=\partpermII{a_1}{a_2}{b_2}{b_1}$.  In particular, $(\si\al\tau,\si\be\tau)\in\Rb_{\S_2}\sm\Rb_1$, so Theorem \ref{thm:In} (and Remark \ref{rem:InOnPrinc}) gives $\Rb_{\S_2}=\cgI{\si\al\tau}{\si\be\tau}\sub\cg\al\be$.  Lemma~\ref{lem:mn}(ii) then gives $\mu_{\S_2}\sub\cg\al\be$, and again this is enough.  \epf

Next we move onto the remaining four lower principal congruences.

\begin{prop}\label{prop:joins2}
The congruences $R_0$, $\rho_1$, $\lam_1$ and $R_1$ are all principal.  Moreover, if $\al,\be\in\P_n$, then
\begin{itemize}\begin{multicols}2
\itemit{i} $R_0=\cg\al\be\iff(\al,\be)\in R_0 \sm(\rho_0\cup\lam_0)$,
\itemit{ii} $\rho_1=\cg\al\be\iff(\al,\be)\in\rho_1\sm(\mu_1\cup\rho_0)$,
\itemit{iii} $\lam_1=\cg\al\be\iff(\al,\be)\in\lam_1\sm(\mu_1\cup\lam_0)$, 
\itemit{iv} $R_1=\cg\al\be\iff(\al,\be)\in R_1\sm ( R_0 \cup \lam_1 \cup \rho_1 )$.
\end{multicols}\eit
\end{prop}

\pf Again, the ``only if'' parts are obvious.  Beginning with (i), suppose $(\al,\be)\in R_0\sm(\rho_0\cup\lam_0)$, and write $\al=\partII{A_1}\cdots{A_r}{B_1}\cdots{B_s}$ and $\be=\partII{C_1}\cdots{C_t}{D_1}\cdots{D_u}$,
noting that $\ker(\al)\not=\ker(\be)$ and $\coker(\al)\not=\coker(\be)$.  By Proposition \ref{prop-CR4}(i), $R_0=\rho_0\vee\lam_0$, so it is enough to prove that $\cg\al\be$ contains both $\rho_0$ and $\lam_0$.  
Put $\ga = \partII{A_1}\cdots{A_r}{D_1}\cdots{D_u}$.
Then $(\al,\ga)=(\ga\al,\ga\be)\in\rho_0\sm\De$, so Proposition \ref{prop:small_congruences}(i) gives $\rho_0=\cg{\al}{\ga}\sub\cg\al\be$.  Similarly, $\lam_0=\cg\ga\be=\cg{\al\ga}{\be\ga}\sub\cg\al\be$.

Next we prove (ii), and we note that (iii) will then follow by duality.  
Suppose $(\al,\be)\in\rho_1\sm(\mu_1\cup\rho_0)$.  Renaming $\al,\be$ if necessary, we may assume that
$\rank(\al)=1$, and we write $\al=\partIV{A_0}{A_1}\cdots{A_r}{B_0}{B_1}\cdots{B_s}$.
Since $\rank(\be)\leq1$ and $\ker(\al)=\ker(\be)$, we may write
\[
\text{(a)}\ \be=\partIV{A_0}{A_1}{\cdots}{A_r}{C_0}{C_1}{\cdots}{C_t},\ \text{or}\ \ \ \ \ \ 
\text{(b)}\ \be=\partIV{A_r}{A_0}{\cdots}{A_{r-1}}{C_0}{C_1}{\cdots}{C_t},\ \text{where}\ r\not=0,\ \text{or}\ \ \ \ \ \ 
\text{(c)}\ \be=\partIII{A_0}{A_1}{\cdots}{A_r}{C_0}{C_1}{\cdots}{C_t}.
\]
Since $(\al,\be)\not\in\mu_1$, we have $\coker(\al)\not=\coker(\be)$.  In any of cases (a)--(c), put $\ga=\partIII{A_0}{A_1}{\cdots}{A_r}{A_0}{A_1}{\cdots}{A_r}$.
Then $(\ga\al,\ga\be)\in\cg\al\be$, with $\ga\al=\partIII{A_0}{A_1}{\cdots}{A_r}{B_0}{B_1}\cdots{B_s}$ and $\ga\be=\partIII{A_0}{A_1}{\cdots}{A_r}{C_0}{C_1}{\cdots}{C_t}$.
In particular, $(\ga\al,\ga\be)\in\rho_0\sm\De$, so Proposition \ref{prop:small_congruences}(i) gives $\rho_0=\cg{\ga\al}{\ga\be}\sub\cg\al\be$. 
To complete the proof of~(ii), since $\rho_1=\mu_1\vee\rho_0$ (by Proposition \ref{prop-CR4a}), it suffices to show that $\mu_1\sub\cg\al\be$.  To do this, we consider cases (a)--(c) separately.

First, suppose (a) holds.  We consider two separate subcases.  If $B_0\not=C_0$, then (reversing the roles of $\al$ and $\be$, if necessary) choose some $i\in B_0\sm C_0$, write $\bn\sm\{i\}=\{i_1,\ldots,i_{n-1}\}$, let $\de=\partpermI ii$,
and note that $\al\de=\partIV{A_0}{A_1}\cdots{A_r}{i}{i_1}\cdots{i_{n-1}}$ and $\be\de=\partIII{A_0}{A_1}\cdots{A_r}{i}{i_1}\cdots{i_{n-1}}$.
On the other hand, if $B_0=C_0$, then (reversing the roles of $\al$ and $\be$, if necessary) choose some $i\in B_0=C_0$, some $j,k\in\bn$ with $(j,k)\in\coker(\al)\sm\coker(\be)$, write $\bn\sm\{i,j,k\}=\{j_1,\ldots,j_{n-3}\}$, let $\de=\partVI{k}{i,j}{j_1}{j_{n-3}}{k}{i,j}{j_1}{j_{n-3}}$,
and note that $\al\de=\partVI{A_0}{A_1}{A_2}{A_r}{k}{i,j}{j_1}{j_{n-3}}$ and $\be\de=\partVII{A_0}{A_1}{A_2}{A_r}{k}{i,j}{j_1}{j_{n-3}}$.
In either subcase, Proposition \ref{prop:small_congruences}(iii) gives $\mu_1=\cg{\al\de}{\be\de}\sub\cg\al\be$.

Next, suppose (b) holds.  Again, we consider two separate subcases.  If $B_0\cap C_0\not=\emptyset$, then choose some ${i\in B_0\cap C_0}$, write $\bn\sm\{i\}=\{i_1,\ldots,i_{n-1}\}$, let $\de=\partIV{i}{i_1}\cdots{i_{n-1}}{B_0}{B_1}\cdots{B_s}$,
and note that $\al\de=\al $ and $\be\de=\partIV{A_r}{A_0}\cdots{A_{r-1}}{B_0}{B_1}\cdots{B_s}$.
On the other hand, if 
$B_0\cap C_0=\emptyset$, choose some $i\in B_0$ and $j\in C_0$, write $\bn\sm\{i,j\}=\{j_1,\ldots,j_{n-2}\}$, let 
$\de=\partpermII ijij$, 
and note that $\al\de=\partVI{A_0}{A_1}{A_2}{A_r}{i}{j}{j_1}{j_{n-2}}$ and $\be\de=\partVI{A_r}{A_0}{A_1}{A_{r-1}}{j}{i}{j_1}{j_{n-2}}$.
In either case, $(\al\de,\be\de)\in\mu_1\sm\De$, so Proposition \ref{prop:small_congruences}(iii) gives $\mu_1=\cg{\al\de}{\be\de}\sub\cg\al\be$.

Finally, if (c) holds, then $\be(\al^*\al)=\alh$, and Proposition \ref{prop:small_congruences}(iii) gives $\mu_1=\cg\al{\alh}=\cg{\al\al^*\al}{\be\al^*\al}\sub\cg\al\be$.
This completes the proof of (ii).

For (iv), suppose $(\al,\be)\in R_1\sm (R_0 \cup \rho_1 \cup \lam_1 )$.  
Renaming $\al,\be$ if necessary, we may assume that
\[
\rank(\al)=1\COMMA
\rank(\be)\leq1 \COMMA
\ker(\al)\not=\ker(\be)\COMMA
\coker(\al)\not=\coker(\be).
\]
Write $\al=\partIV{A_0}{A_1}\cdots{A_r}{B_0}{B_1}\cdots{B_s}$ and $\be=\partIV{C_0}{C_1}\cdots{C_t}{D_0}{D_1}\cdots{D_u}$ or $\partIII{C_0}{C_1}\cdots{C_t}{D_0}{D_1}\cdots{D_u}$.
Since $\rho_1\vee\lam_1=R_1$, by Proposition~\ref{prop-CR4}(i), to prove that $\cg\al\be=R_1$, it suffices to show that $\cg\al\be$ contains $\rho_1$ and $\lam_1$.  We just show the first of these, as the second is dual.  One of the following three statements must hold: 
\[
\text{(d)}\ \rank(\be)=0,\ \text{or}\ \ \ \ \ \ 
\text{(e)}\ \rank(\be)=1\  \text{and}\  A_0=C_0,\ \text{or}\ \ \  \ \ \ 
\text{(f)}\ \rank(\be)=1\ \text{and}\ A_0\not=C_0.
\]
Suppose first that (d) or (e) holds.  We then define $\ga=\partIV{A_0}{A_1}\cdots{A_r}{A_0}{A_1}\cdots{A_r}$,
and note that $\ga\al=\al$ and $\ga\be=\partIII{A_0}{A_1}\cdots{A_r}{D_0}{D_1}\cdots{D_u}$ or $\partIV{A_0}{A_1}\cdots{A_r}{D_0}{D_1}\cdots{D_u}$.
In either case, we see that $\cg{\al}{\ga\be}=\cg{\ga\al}{\ga\be}\sub\cg\al\be$, while ${\cg{\al}{\ga\be}=\rho_1}$ follows from case (ii), proved above.  

Now suppose (f) holds.  Reversing the roles of $\al$ and $\be$, if necessary, we may choose some $i\in A_0\sm C_0$.  Write $\bn\sm\{i\}=\{i_1,\ldots,i_{n-1}\}$, and let $\ga=\partpermI ii$.
Then $\ga\al=\partIV i{i_1}\cdots{i_{n-1}}{B_0}{B_1}\cdots{B_s}$ and $\ga\be=\partIII i{i_1}\cdots{i_{n-1}}{D_0}{D_1}\cdots{D_u}$.
In particular, $\rho_1=\cg{\ga\al}{\ga\be}\sub\cg\al\be$, again using part (ii).  \epf

\begin{rem}
\label{NotPrinc}
The congruences $\lambda_{\S_2}$, $\rho_{\S_2}$ and $R_{\S_2}$ are not principal (cf.~Remark \ref{rem:nonprincipal}). 
However, each of them is generated by two pairs of partitions, and these generating sets (up to reordering) are as follows:
\bit
\item[(i)] $\rho_{\S_2}=\cgset\al\be\ga\de\iff (\al,\be)\in\rho_{\S_2}\sm \rho_1$ and $(\ga,\de)\in \rho_{\S_2}\sm\mu_{\S_2}$,
\item[(ii)] $\lam_{\S_2}=\cgset\al\be\ga\de\iff (\al,\be)\in\lam_{\S_2}\sm \lam_1$ and $(\ga,\de)\in \lam_{\S_2}\sm\mu_{\S_2}$,
\item[(iii)] $R_{\S_2}=\cgset\al\be\ga\de\iff (\al,\be)\in R_{\S_2}\sm R_1$ and $(\ga,\de)\in R_{\S_2}\sm (\lambda_{\S_2}\cup\rho_{\S_2})$.
\eit
The forwards implications of these assertions are easily proved, using $\rho_{\S_2}=\mu_{\S_2}\cup\rho_1$,
$\lambda_{\S_2}=\mu_{\S_2}\cup\lambda_1$ and $R_{\S_2}=R_1\cup \lambda_{\S_2}=R_1\cup\rho_{\S_2}$.  
The converses follow by combinations of facts proved in Propositions \ref{prop:small_congruences} and~\ref{prop:joins2}, but we leave it to the reader to verify this, as it does not play an important role in what follows.
\end{rem}

Finally, we turn our attention to the upper congruences. Specifically, we consider the chain of relations
\begin{equation}
\label{eqaa1}
R_{\S_2}\suq R_2 \suq R_{\A_3}\suq R_{\S_3}\suq
R_3\suq R_K \suq R_{\A_4}\suq R_{\S_4}\suq
R_4\suq R_{\A_5}\suq R_{\S_5}\suq R_5\suq\cdots
\suq R_{\S_n}\suq R_n=\nabla.
\end{equation}
For an arbitrary member $\xi$ of this chain other than $R_{\S_2}$, let
$\xi^-$ be its immediate predecessor in the chain. 
We aim to show that every such $\xi$ is generated by any pair from $\xi\setminus \xi^-$.
We proceed via  a sequence of lemmas.

\begin{lemma}
\label{lemma-aa1}
For every $\alpha\in\P_n$, there exists $\beta\in\I_n$ with $\rank(\beta)=\rank(\alpha)$, $\al=\alpha\beta\alpha$ and $\be=\be\al\be$.
\end{lemma}

\pf
If $\alpha=\partI{A_1}{A_q}{C_1}{C_r}{B_1}{B_q}{D_1}{D_s}$, pick $a_i\in A_i$, $b_i\in B_i$ ($i=1,\ldots,q$), and define $\beta=\partpermIII{b_1}\cdots{b_q}{a_1}\cdots{a_q}$. Verification of the stated properties is immediate.
\epf

\begin{lemma}
\label{lemma-aa1b}
Let $\al,\be\in J_q$ with $q\geq1$.  
\bit
\itemit{i} If $(\al,\be)\not\in{\R}$, then there exists $\gamma\in J_q$ such that precisely one of $\gamma\alpha,\gamma\beta$ belongs to $J_q$.
\itemit{ii} If $(\al,\be)\not\in{\L}$, then there exists $\gamma\in J_q$  such that precisely one of $\alpha\gamma,\beta\gamma$ belongs to $J_q$.
\end{itemize}
\end{lemma}

\pf  By duality, it suffices to prove (i), so suppose $(\al,\be)\not\in{\R}$, and write $\alpha=\partI{A_1}{A_q}{C_1}{C_r}{B_1}{B_q}{D_1}{D_s}$ and $\beta=\partI{E_1}{E_q}{G_1}{G_t}{F_1}{F_q}{H_1}{H_u}$.
Since $(\alpha,\beta)\not\in{\R}$, either $\dom(\alpha)\neq\dom(\beta)$ or $\ker(\alpha)\neq\ker(\beta)$. 

\bigskip\noindent {\bf Case 1.}  Suppose first that $\dom(\alpha)\neq\dom(\beta)$.  Without loss of generality, we may assume that $A_1\sm\dom(\beta)$ is nonempty.  
Let $a_1\in A_1\sm\dom(\beta),a_2\in A_2,\ldots,a_q\in A_q$, and let $\ga=\partpermIII{a_1}\cdots{a_q}{a_1}\cdots{a_q}$.
Then $\rank(\ga\al)=q$ and $\rank(\ga\be)<q$; the latter is the case because $\dom(\ga\be)\sub\dom(\ga)=\{a_1,\ldots,a_q\}$, yet $a_1\not\in\dom(\ga\be)$.  

\bigskip\noindent {\bf Case 2.} Next suppose $\dom(\al)=\dom(\be)$, but $\ker(\al)\neq\ker(\be)$.  Without loss of generality, we may assume that there exists some $(i,j)\in\ker(\al)\sm\ker(\be)$, and that either $i,j\in A_1$ or $i,j\in C_1$.  

\bigskip\noindent {\bf Subcase 2.1.}  Suppose $i,j\in A_1$.  Since $i,j\in\dom(\al)=\dom(\be)$ and $(i,j)\not\in\ker(\be)$, we may assume without loss of generality that $i\in E_1$ and $j\in E_2$.  For each $3\leq p\leq q$, choose some $e_p\in E_p$, and put $e_1=i$ and $e_2=j$.  Define $\ga=\partpermIII{e_1}\cdots{e_q}{e_1}\cdots{e_q}$.
Then $\rank(\ga\be)=q$ and $\rank(\ga\al)<q$; the latter is the case because $(i,j)\in\ker(\ga\al)$ and $i,j\in\dom(\ga\al)=\{e_1,\ldots,e_q\}$.  

\bigskip\noindent {\bf Subcase 2.2.}  Finally, suppose $i,j\in C_1$.  For each $1\leq p\leq q$, choose some $a_p\in A_p$.  Write $\bn\sm\{a_1,\ldots,a_q\}=\{x_1,\ldots,x_{n-q}\}$, assuming that $x_1=i$ and $x_2=j$, and put 
$\ga =
\Big( 
{ \scriptsize \renewcommand*{\arraystretch}{1}
\begin{array} {\c|\c|\c|\c|\c|\c|\c|\cend}
a_1 \:&\: \cdots \:&\: a_{q-1} \:&\: i \:&\: a_q,j \:&\: x_3 \:&\: \cdots \:&\: x_{n-q} \\ \cline{5-8}
a_1 \:&\: \cdots \:&\: a_{q-1} \:&\: i \:&\: a_q,j \:&\: x_3 \:&\: \cdots \:&\: x_{n-q}
\rule[0mm]{0mm}{2.7mm}
\end{array} 
}
\hspace{-1.5 truemm} \Big)
$.
Then $\rank(\ga\al)=q$ and $\rank(\ga\be)<q$; the latter is the case because $\{i\}$ is a singleton block of $\ga\be$, and $\dom(\ga\be)\sub\dom(\ga)=\{a_1,\ldots,a_{q-1},i\}$. 
\epf

The next result concerns the permutations $\phi(\al,\be)$ from Definition \ref{defn:phi}, and the symmetric inverse monoid~$\I_n\sub\P_n$.

\begin{lemma}
\label{lemma-aa1a}
For any $(\al,\be)\in {\H}\cap (J_q\times J_q)$ with $q\geq 1$, there exist $\al_1,\be_1,\ga_1,\de_1\in J_q(\I_n)$ and $\ga_2,\de_2\in J_q$ such that~
\begin{itemize}
\begin{multicols}2
\itemit{i} $(\alpha_1,\beta_1)=(\gamma_1\alpha\delta_1,\gamma_1\beta\delta_1) \in{\H}$,
\itemit{ii} $\phi(\alpha_1,\beta_1)=\phi(\alpha,\beta)$, 
\itemit{iii} $(\al,\be)=(\ga_2\al_1\de_2,\ga_2\be_1\de_2)$.
\end{multicols}\eit
\end{lemma}

\pf
Write
$\alpha=\partI{A_1}{A_q}{C_1}{C_r}{B_{1\si}}{B_{q\si}}{D_1}{D_s}$
and
$\beta=\partI{A_1}{A_q}{C_1}{C_r}{B_{1\tau}}{B_{q\tau}}{D_1}{D_s}$, where $\min(A_1)<\cdots<\min(A_q)$, $\min(B_1)<\cdots<\min(B_q)$ and $\sigma,\tau\in\S_k$.  For each $1\leq i\leq q$, put $a_i=\min(A_i)$ and $b_i=\min(B_i)$.  Then $\ga_1=\partpermIII{a_1}\cdots{a_q}{a_1}\cdots{a_q}$, $\de_1=\partpermIII{b_1}\cdots{b_q}{b_1}\cdots{b_q}$, $\ga_2=\partI{A_1}{A_q}{C_1}{C_r}{A_1}{A_q}{C_1}{C_r}$ and $\de_2=\partI{B_1}{B_q}{D_1}{D_s}{B_1}{B_q}{D_1}{D_s}$
have the desired properties. \epf

Recall that for any congruence $\xi$ on $\P_n$, we write $\xib=\xi\cap (\I_n\times\I_n)$ for the induced congruence on $\I_n$.

\begin{lemma}
\label{lemma-aa2}
Let $\xi$ be any of the congruences from the chain \eqref{eqaa1} other than $R_{\S_2}$.  Then $\xi=\xib^\sharp$.

\end{lemma}

\pf
Since $\xib\sub\xi$, the inclusion $\xib^\sharp\sub\xi$ is clear, so only the converse needs to be proved. 

\bigskip\noindent {\bf Case 1.} Suppose first that $\xi=R_q$ for some $q\geq 2$.  We begin by showing that $R_0\subseteq \xib^\sharp$.  To do so, let $\alpha=\partpermII1212$, $\be=\emptypartperm$ and $\si=\partIII{12}{3}\cdots{n}{12}{3}\cdots{n}$.  Now $(\al,\be)\in\Rb_2\sub\xib$, so it follows that $(\si,\be)=(\al\si\al,\be\si\be)\in\xib^\sharp$.  But $(\si,\be)\in R_0\sm(\rho_0\cup\lam_0)$, so Proposition \ref{prop:joins2}(i) gives $R_0=\cg\si\be\sub\xib^\sharp$.  To complete the proof that $\xi=R_q\sub\xib^\sharp$, it suffices to show that every partition of rank at most $q$ is $\xib^\sharp$-related to a partition of rank~$0$.  So let $\ga\in\P_n$ with $\rank(\ga)=p\leq q$.  By Lemma \ref{lemma-aa1}, there exists $\de\in\I_n$ with $\rank(\de)=p$ such that $\ga=\ga\de\ga$.  But then with $\be=\emptypartperm$ as above, we have $(\de,\be)\in\Rb_q\sub\xib$, so that $(\ga,\ga\be\ga)=(\ga\de\ga,\ga\be\ga)\in\xib^\sharp$.  Since $\rank(\ga\be\ga)=0$, the proof is complete in this case.

\bigskip\noindent {\bf Case 2.}  Now suppose $\xi=R_N$ for some non-trivial normal subgroup $N$ of $\S_q$, where $q\geq3$.  Now, $R_{q-1}=\Rb_{q-1}^\sharp\sub\xib^\sharp$ by Case 1, so it remains to show that $R_N\sm R_{q-1}\sub\xib^\sharp$.  
Let $(\al,\be)\in R_N\sm R_{q-1}=\nu_N\sm \Delta$ be arbitrary.  So $\rank(\al)=\rank(\be)=q$, $\al\H\be$ and $\phi(\al,\be)\in N$.  Let $\al_1,\be_1,\ga_1,\de_1\in J_q(\I_n)$ and $\ga_2,\de_2\in J_q$ be as in Lemma \ref{lemma-aa1a}.  Then $(\al_1,\be_1)\in \overline{R}_N=\xib$, and it follows that $(\al,\be)=(\ga_2\al_1\de_2,\ga_2\be_1\de_2)\in\xib^\sharp$. \epf

The next result is the final step required to complete the proof of Theorem \ref{thm-CongPn}.

\begin{prop}
\label{prop-aa3}
Consider the chain \eqref{eqaa1} of congruences on $\P_n$, and let $\xi$ be any of these congruences other than $R_{\S_2}$. Then $\xi$ is a principal congruence on $\P_n$.  Moreover, if $\al,\be\in\P_n$, then $\xi=\cg\al\be$ if and only if $(\al,\be)\in\xi\sm\xi^-$, where $\xi^-$ denotes the congruence immediately preceding $\xi$ in the chain \eqref{eqaa1}.  
\end{prop}

\pf The ``only if'' part being clear, suppose $(\al,\be)\in\xi\sm\xi^-$.  We must show that $\xi\sub\cg\al\be$.  In fact, by Lemma \ref{lemma-aa2}, it is enough to show that $\xib\sub\cg\al\be$.

\bigskip\noindent {\bf Case 1.}  Suppose first that $\xi=R_q$ for some $q\geq2$, so that $\xi^-=R_{\S_q}$.  Renaming if necessary, we may assume that $\rank(\al)=q$.  We also have either $\rank(\be)<q$ or else $\rank(\be)=q$ but $(\al,\be)\not\in{\H}$, since the condition ``$\phi(\al,\be)\in N$'' is trivially satisfied when $N=\S_q$ (cf.~Definition \ref{Rel-nu} and Lemma \ref{lem:nu_phi}).  

\bigskip\noindent {\bf Subcase 1.1.}  Now suppose $\rank(\be)=p<q$.  By Lemma \ref{lemma-aa1}, there exists $\ga\in\I_n$ with $\rank(\ga)=q$ such that $\ga\al\ga=\ga$.  Since $\rank(\ga\be\ga)\leq\rank(\be)=p<q$, and since $(\ga\al\ga,\ga\be\ga)\in\cg\al\be$, we may assume without loss of generality that $\al\in\I_n$.  By Lemma \ref{lemma-aa1} again, there exists $\de\in\I_n$ with $\rank(\de)=p$ and $\be=\be\de\be$.  Writing ${\zeta}=\cg\al\be$, we then have
$\al \mathrel\zeta \be = \be\de\be \mathrel\zeta \al\de\al$.
But $\al\de\al\in\I_n$ and $\rank(\al\de\al)\leq\rank(\de)<q$, so $(\al,\al\de\al)\in\Rb_q\sm\overline{R}_{\S_q}$.  It then follows from Theorem \ref{thm:In} (and Remark \ref{rem:InOnPrinc}) that $\Rb_q=\cgI{\al}{\al\de\al}\sub{\zeta}=\cg\al\be$, as required.

\bigskip\noindent {\bf Subcase 1.2.}  Suppose $\rank(\be)=q$ and $(\al,\be)\not\in{\H}$.  Without loss of generality, we may assume that $(\al,\be)\not\in{\R}$.  By Lemma \ref{lemma-aa1b}(i), and renaming $\al,\be$ if necessary, there exists $\ga\in\P_n$ with $\rank(\ga\al)<q=\rank(\ga\be)$.  But then by Subcase 1.1, $\Rb_q=\cg{\ga\al}{\ga\be}\sub\cg\al\be$.

\bigskip\noindent {\bf Case 2.}  Suppose $\xi=R_N$ for some non-trivial normal subgroup $N$ of $\S_q$, where $q\geq3$.  Let $H$ be the largest normal subgroup of $\S_q$ properly contained in $N$, and note that $\xi^-=R_H$.  (For the case where $N$ is the smallest non-trivial normal subgroup of $\S_q$, recall that $R_{\{\id_q\}}=R_{q-1}$.)  Because $(\al,\be)\in R_N\sm R_H$, we have $\rank(\al)=\rank(\be)=q$, $\al\H\be$ and $\phi(\al,\be)\in N\sm H$.  Let $\al_1,\be_1,\ga_1,\de_1$ be as in Lemma \ref{lemma-aa1a}.  Then, from the conclusions of that lemma, it follows that $(\al_1,\be_1)\in(\overline{R}_N\sm\overline{R}_H)\cap\cg\al\be$.  It follows from Theorem~\ref{thm:In} (and Remark \ref{rem:InOnPrinc}) that $\Rb_N=\cgI{\al_1}{\be_1}\sub\cg\al\be$, completing the proof. \epf

Now that the proof of Theorem \ref{thm-CongPn} is complete, we conclude this section by noting that it is easy to derive a description of the lattice $\Cong^*(\P_n)$ of $\ast$-congruences on $\P_n$.  Indeed, to do this, we just need to check which of the congruences of $\P_n$ are also compatible with the unary operation $\alpha\mapsto\alpha^\ast$.  It turns out that most of them do, with the exception of the ``left/right'' congruences $\lambda_0,\lambda_1,\lambda_{\S_2},\rho_1,\rho_2,\rho_{\S_2}$.

\begin{cor}\label{cor:main_Pn}
For $n\geq 2$, $\Cong^\ast(\P_n)=\Cong(\P_n)\setminus\{\lambda_0,\lambda_1,\lambda_{\S_2},\rho_1,\rho_2,\rho_{\S_2}\}$.  Theorem \ref{thm-CongPn} describes the congruence lattice $\Cong(\P_n)$.  The Hasse diagram of $\Cong^*(\P_n)$ is shown in Figure \ref{fig-CongPn}.
\end{cor}

\pf
By Remark \ref{rem:green_sumbonoids}(i), $\alpha\gJ\alpha^\ast$ for any $\al\in\P_n$, and hence the Rees congruences $R_q$ are $\ast$-compatible.
Furthermore, from Definition \ref{Rel-nu}, it easily follows that $\alpha\mapsto\alpha^\ast$ preserves the relation $\nu_N$ (for any normal subgroup $N$ of~$ \S_q$), whence the congruences $R_N$ are also $\ast$-compatible.
Next, note that the retraction $f$ from Definition~\ref{defn:hat} agrees with the involution, in the sense that $\alpha^\ast f=(\alpha f)^\ast$; therefore the relations $\mu_0,\mu_1,\mu_{\S_2}$ are $\ast$-compatible.
Finally, by  Remark \ref{rem:green_sumbonoids}(i), we have that $\alpha\mapsto\alpha^\ast$ maps $\R$-classes of $\P_n$ to $\L$-classes, and vice versa, and hence the congruences $\lambda_0,\lambda_1,\lambda_{\S_2},\rho_1,\rho_2,\rho_{\S_2}$ are not $\ast$-compatible.
\epf

\section{The partial Brauer monoid \boldmath{$\PB_n$}}\label{sect:PBn}

The goal of this section is to give a brief treatment of the congruence lattice of the partial Brauer monoid~$\PB_n$, which, recall, consists of all partitions with blocks of sizes $\leq 2$.  In a nutshell, the lattice is isomorphic to that of $\P_n$, and the proof is almost identical  to the one for $\P_n$ given in the previous section.

To verify this, we must first check that all the ingredients needed to define various congruences forming the lattice shown in Figure \ref{fig-CongPn} are present in the context of the $\PB_n$. Indeed, $\PB_n$ has $n+1$ ideals and $\gJ$-classes, $I_q$ and $J_q$ respectively, indexed by the available ranks $q=0,\dots, n$.  They give rise to the Rees congruences $R_q$ (Definition \ref{defn:Rees}).  Next, the map $\S_q\to\P_n:\si\mt\sib=\partpermIII1\cdots q{1\si}\cdots{q\si}$, discussed in Section \ref{sect:SnInOn}, maps into $\PB_n$, so for every normal subgroup~$N$ of $ \S_q$ for $q\geq1$, we have an IN-pair $(I_{q-1},\overline{N})$, giving rise to the congruence $R_N=R_{I_{q-1},\overline N}$ (Definition \ref{defn:RC}).

The next step is to note that the retraction $f:\P_n\to\P_n:\al\mt\alh$ of $\P_n$, given in Definition \ref{defn:hat}, maps $\PB_n$ into $\PB_n$.  Indeed, $\alh$ is obtained by ``breaking'' the transversals of $\al$ into two, thereby reducing their size, and leaving the non-transversals alone.  Thus, $I_1$ is retractable (and, as with $\P_n$, no larger ideal of $\PB_n$ is retractable).  Finally, we have three retractable IN-pairs $\C_0=(I_0,\{\overline{\id}_1\})$, $\C_1=(I_1,\{\overline{\id}_2\})$, $\C_{\S_2}=(I_1,\overline{\S}_2)$, giving rise to nine $\lambda,\rho,\mu$ congruences.

\begin{thm}\label{thm-CongPBn}
Let $n\geq2$, and let $\PB_n$ be the partial Brauer monoid of degree $n$.  
\bit
\itemit{i} The ideals of $\PB_n$ are the sets $I_q$, for $q=0,\ldots,n$, yielding the Rees congruences $R_q=R_{I_q}$, as in Definition \ref{defn:Rees}.
\itemit{ii} The proper IN-pairs of $\PB_n$ are of the form $(I_{q-1},\overline N)$, for $q=2,\ldots,n$ and for any non-trivial normal subgroup $N$ of $\S_q$, yielding the congruences $R_N=R_{I_{q-1},\overline N}$, as in Definition \ref{defn:RC}.
\itemit{iii} The retractable IN-pairs of $\PB_n$ are $\C_0=(I_0,\{\overline{\id}_1\})$, $\C_1=(I_1,\{\overline{\id}_2\})$ and $\C_{\S_2}=(I_1,\overline{\S}_2)$, yielding the congruences $\lambda_0$, $\rho_0$, $\mu_0$, $\lambda_1$, $\rho_1$, $\mu_1$, $\lambda_{\S_2}$, $\rho_{\S_2}$, $\mu_{\S_2}$, respectively, as in Definition \ref{defn:lrmC}.
\itemit{iv} The above congruences are distinct, and they exhaust all the congruences of $\PB_n$.
\itemit{v} The $\ast$-congruences are the same, but with $\lambda_0,\lambda_1,\lambda_{\S_2},\rho_0,\rho_1,\rho_{\S_2}$ excluded.
\itemit{vi} The Hasse diagrams of the congruence and $*$-congruence lattices $\Cong(\PB_n)$ and $\Cong^*(\PB_n)$ are shown in Figure \ref{fig-CongPn}.
\eit
\end{thm}

The proof now proceeds in exactly the same way as the Proof of Theorem \ref{thm-CongPn}.  The reader only needs to be alert that every time a partition is constructed, it is necessary to verify that it belongs to $\PB_n$ (on the assumption, of course, that all the previous partitions came from this monoid too).

Thus, for example, in the proof of Lemma \ref{lem:mn}(i), given a partition $\al=\partIV{A_0}{A_1}\cdots{A_r}{B_0}{B_1}\cdots{B_s}$, four new partitions are constructed:
\[
\ga=\partIV{A_0}{A_1}\cdots{A_r}{1}{2}\cdots{n} \COMMA \de=\partIV{1}{2}\cdots{n}{B_0}{B_1}\cdots{B_s} \COMMA \si=\partpermI11 \COMMA \tau=\emptypartperm .
\]
Now, $\gamma$ has the same upper blocks as $\alpha$, and lower blocks of size $1$; similarly, $\delta$ has lower blocks equal to those of $\alpha$ and upper blocks of size $1$. Hence, if $\alpha\in\PB_n$, then $\gamma,\delta\in\PB_n$, and also $\sigma,\tau\in\I_n\subseteq \PB_n$.  Checks like this need to be performed throughout the proof, but they are no harder than the above example.
One last small thing perhaps worthy of mention is that in the proof of the analogue of Lemma \ref{lemma-aa1b}, Subcase 2.1 does not arise.  Indeed the situation there is that we are considering a partition $\alpha$ of the form $\partI{A_1}{A_q}{C_1}{C_r}{B_1}{B_q}{D_1}{D_s}$.  The assumption in Subcase 2.1 is that $i,j\in A_1$ and $i\neq j$. However, if $\alpha\in\PB_n$, then from $|A_1\cup B_1^\prime|\leq 2$ it follows that $|A_1|=|B_1|=1$.

\section{The planar partition monoid \boldmath{$\PP_n$} and the Motzkin monoid \boldmath{$\M_n$}}
\label{sect:PPnMn}

Recall that the planar partition monoid $\PP_n$ consists of all planar partitions,
and that the Motzkin monoid $\M_n=\PB_n\cap \PP_n$ consists of all planar partitions with blocks of sizes $\leq 2$.
In this section we describe the congruence lattices of these two monoids. The proof will follow
the model set up in Section \ref{sect:Pn}, but the following points should be noted before we begin.

The two congruence lattices $\Cong(\PP_n)$ and $\Cong(\M_n)$ turn out to be isomorphic, in a manner analogous to $\Cong(\P_n)$ and $\Cong(\PB_n)$. We will this time run the two proofs in parallel, concentrating on $\PP_n$, and making relevant remarks about $\M_n$ when necessary.

Unlike the situation with $\P_n$ and $\PB_n$, Green's $\H$ relation on $\PP_n$ is trivial, and hence all subgroups of $\PP_n$ are trivial. This has two consequences. Firstly, the role that was played by the symmetric inverse monoid $\I_n$
will now be played by $\I_n\cap \PP_n=\O_n$, the monoid of order-preserving partial permutations of~$\bn$. 
Thus, in this section, for a congruence $\xi\in\Cong(\PP_n)$ we will write $\overline{\xi}=\xi\cap(\O_n\times\O_n)$ for its restriction to $\O_n$.
And, secondly, there will be no congruences of the type $R_N,\lambda_N,\rho_N,\mu_N$, for non-trivial normal subgroups $N$ of~$\S_q$.

Yet again, for $q=0,\dots,n$, we have ideals $I_q$, $\gJ$-classes $J_q$, and Rees congruences $R_q$ (Definition \ref{defn:Rees}).
The first main step is to verify that our familiar retraction $f:\alpha\mapsto\widehat{\alpha}$ remains valid in the context of $\PP_n$ (see Corollary~\ref{cor:retract_planar}).  In order to do that, and for later use, it will be convenient to discuss a particular canonical graphical representation of a planar partition $\al\in\PP_n$.

Let $A=\{a_1,\ldots,a_k\}$ and $B=\{b_1,\ldots,b_l\}$ be (possibly equal) nonempty subsets of $\bn$ with $a_1<\cdots<a_k$ and $b_1<\cdots<b_l$.  We define three graphs $\Ga_A$, $\Ga_{B'}$ and $\Ga_{A\cup B'}$ as follows.  All three graphs have vertex set~$\bn\cup\bn'$, and the edge sets are given by
\begin{gather*}
E(\Ga_A)=\big\{\{a_1,a_2\},\ldots,\{a_{k-1},a_k\}\big\} \COMMA
E(\Ga_{B'})=\big\{\{b_1',b_2'\},\ldots,\{b_{l-1}',b_l'\}\big\} ,\\
E(\Ga_{A\cup B'}) = E(\Ga_A)\cup E(\Ga_{B'})\cup\big\{\{a_1,b_1'\},\{a_k,b_k'\}\big\}.
\end{gather*}
If $\al\in\P_n$ is an arbitrary partition, we define the \emph{canonical graph of $\al$} to be the graph $\Ga_\al$ with vertex set~$\bn\cup\bn'$ and edge set $E(\Ga_\al)=\bigcup_{X\in\al}E(\Ga_X)$, where the union is over all blocks $X$ of $\al$.  As an example, the canonical graph of 
$
\al=
\Big( 
{ \scriptsize \renewcommand*{\arraystretch}{1}
\begin{array} {\c|\c|\c|\c|\c|\cend}
1,4 \:&\: 7 \:&\: 2,3 \:&\: 5 \:&\: 6 \:&\: 8  \\ \cline{3-6}
1,2,4 \:&\: 8 \:&\: 3 \:&  \multicolumn{3}{c}{5,6,7}
\rule[0mm]{0mm}{2.7mm}
\end{array} 
}
\hspace{-1.5 truemm} \Big)
\in\PP_8
$
is given in Figure \ref{fig:PmJ2m} below (in black).

Let $A,B$ be nonempty subsets of $\bn$ as in the previous paragraph.  We say that $A$ and $B$ are \emph{separated} if $a_k<b_1$ or $b_l<a_1$; in these cases, we write $A<B$ or $B<A$, respectively.  We say that $A$ is \emph{nested} by $B$ if there exists some $1\leq i<l$ such that $b_i<a_1$ and $a_k<b_{i+1}$; we say that $A$ and $B$ are \emph{nested} if $A$ is nested by $B$ or vice versa.

\begin{lemma}\label{lem:nested_or_separated}
Let $\al=\partABCD\in\PP_n$, with $\min(A_1)<\cdots<\min(A_q)$.  Then
\bit
\itemit{i} $A_1<\cdots<A_q$ and $B_1<\cdots<B_q$,
\itemit{ii} for all $1\leq i<j\leq r$, $C_i$ and $C_j$ are either nested or separated,
\itemit{iii} for all $1\leq i<j\leq s$, $D_i$ and $D_j$ are either nested or separated,
\itemit{iv} for all $1\leq i\leq q$ and $1\leq j\leq r$, either $A_i$ and $C_j$ are separated or else $C_j$ is nested by $A_i$, 
\itemit{v} for all $1\leq i\leq q$ and $1\leq j\leq s$, either $B_i$ and $D_j$ are separated or else $D_j$ is nested by $B_i$.
\eit
Consequently, the canonical graph $\Ga_\al$ may be drawn in planar fashion (within the rectangle determined by the vertices).
\end{lemma}

\pf Clearly the last assertion follows from items (i)--(v).  Since $\al\in\PP_n$, it is represented by some planar graph, say $\Ga$.
The proofs of (i)--(v) are similar; each is proved by identifying paths in
$\Ga$ that would cross if the desired conclusion did not hold.  We just prove
(i), and leave the rest to the reader.  Suppose $1\leq i<q$, and write $a=\min(A_i)$, $b=\max(B_i)$, $c=\min(A_{i+1})$ and $d=\min(B_{i+1})$.  In $\Ga$, there is a path~$\pi_{ab}$ from $a$ to $b'$, and a path $\pi_{cd}$ from $c$ to $d'$.  By assumption, $a<c$.  If $d<b$, then the paths $\pi_{ab}$ and $\pi_{cd}$ would have to intersect.  Thus, $b<d$: i.e., $B_i<B_{i+1}$, giving the second claim in (i).  Since this clearly implies $\min(B_1)<\cdots<\min(B_q)$, the rest of the first claim follows upon applying the above argument to $\al^*=\partI{B_1}{B_q}{D_1}{D_s}{A_1}{A_q}{C_1}{C_r}$. \epf

For any $\al\in\PP_n$, the canonical graph of $\alh$ is obtained from that of $\al$ by deleting the transverse edges, so the final assertion of Lemma \ref{lem:nested_or_separated} immediately gives the following.

\begin{cor}
\label{cor:retract_planar}
The map $\al\mt\alh$ from Definition \ref{defn:hat} maps $\PP_n$ into $\PP_n$. \epfres
\end{cor}

So, as in the case of $\P_n$, the map $I_1\to I_0:\al\mt\alh$ is a retraction (and, as with $\P_n$, no larger ideal of $\PP_n$ is retractable).
This time, however, due to the fact that the maximal subgroup corresponding to $J_2$ is trivial, rather than $\S_2$, we only have two retractable IN-pairs: namely,
$\C_0=(I_0,\{\overline{\id}_1\})$,
$\C_1=(I_1,\{\overline{\id}_2\})$.
They give rise to six congruences (Definition \ref{defn:lrmC}), which we abbreviate to $\lambda_0,\rho_0,\mu_0(=\Delta),\lambda_1,\rho_1,\mu_1$.

\begin{thm}\label{thm-CongPPn}
Let $n\geq2$, and let $\PP_n$ (resp., $\M_n$) be the planar partition monoid (resp., Motzkin monoid) of degree~$n$, respectively.  
\bit
\itemit{i} The ideals of $\PP_n$ (resp., $\M_n$) are the sets $I_q$, for $q=0,\ldots,n$, yielding the Rees congruences $R_q=R_{I_q}$, as in Definition \ref{defn:Rees}.
\itemit{ii} The retractable IN-pairs of $\PP_n$ (resp., $\M_n$) are $\C_0=(I_0,\{\overline{\id}_1\})$ and $\C_1=(I_1,\{\overline{\id}_2\})$, yielding the congruences $\lambda_0$, $\rho_0$, $\mu_0$, $\lambda_1$, $\rho_1$, $\mu_1$, respectively, as in Definition \ref{defn:lrmC}.
\itemit{iii} The above congruences are distinct, and they exhaust all the congruences of $\PP_n$ (resp., $\M_n$).
\itemit{iv} The $\ast$-congruences of $\PP_n$ (resp., $\M_n$) are the same, but with $\lambda_0,\lambda_1,\rho_0,\rho_1$ excluded.
\itemit{v} The Hasse diagrams of the congruence and $*$-congruence lattices of $\PP_n$ (resp., $\M_n$) are shown in Figure~\ref{fig-CongPPn}.
\eit
\end{thm}

\pf
As in the proof of Theorem \ref{thm-CongPn}, the
listed relations are congruences by Proposition \ref{prop-AreCongs}, and they form the lattice depicted in Figure \ref{fig-CongPn} by Propositions \ref{prop-CR2}--\ref{prop-CR5}.
The proof that there are no further congruences is obtained by describing all the principal congruences, and verifying that they are precisely those listed (see Table \ref{PPnCongGens}); 
this is the content of the remainder of this section. 
The assertions concerning the $\ast$-congruences are immediate consequences, exactly as in the proof of Corollary \ref{cor:main_Pn}.
\epf

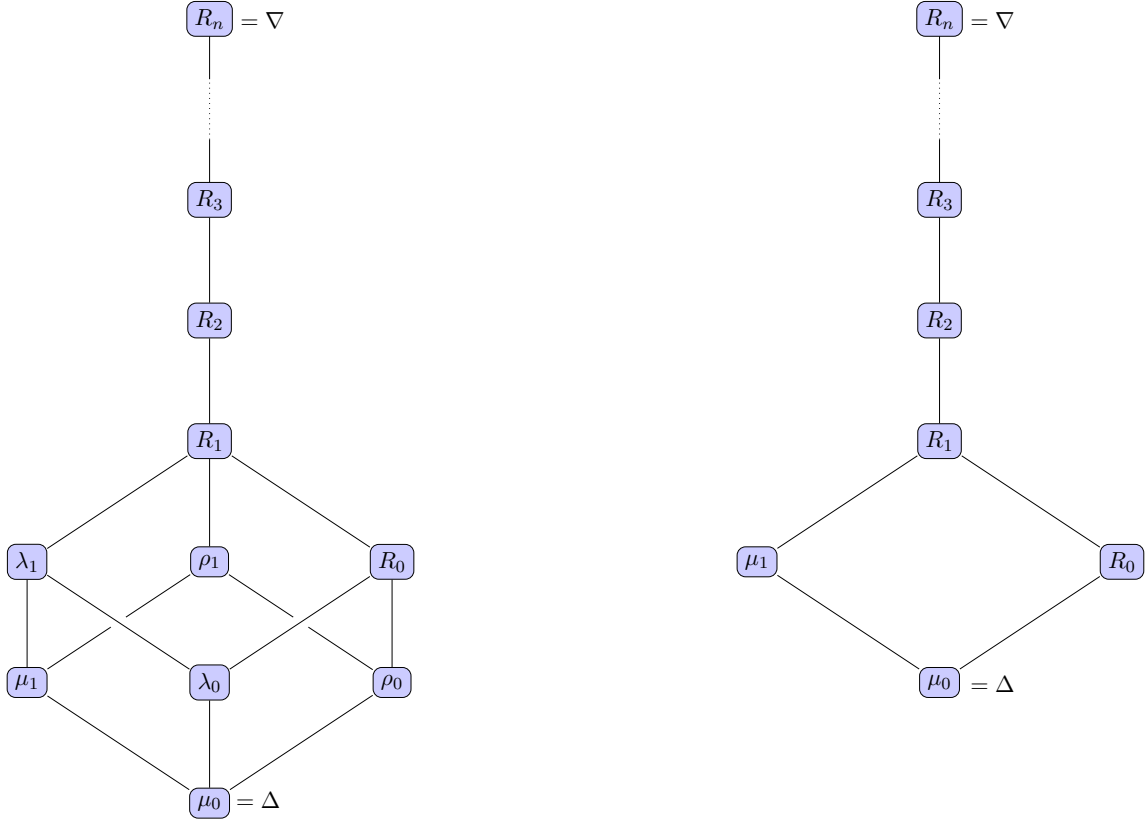
\begin{figure}[ht]
\begin{center}
\scalebox{0.8}{
\begin{tikzpicture}[scale=1]
\dottedinterval0{12}11
\node[rounded corners,rectangle,draw,fill=blue!20] (N) at (0,15) {$R_n$};
  \draw(0.87,14.98) node {$=\nabla$};
\node[rounded corners,rectangle,draw,fill=blue!20] (R3) at (0,12) {$R_3$};
\node[rounded corners,rectangle,draw,fill=blue!20] (R2) at (0,10) {$R_2$};
\node[rounded corners,rectangle,draw,fill=blue!20] (R1) at (0,8) {$R_{1}$};
\node[rounded corners,rectangle,draw,fill=blue!20] (l1) at (-3,6) {$\lambda_{1}$};
\node[rounded corners,rectangle,draw,fill=blue!20] (r1) at (0,6) {$\rho_1$};
\node[rounded corners,rectangle,draw,fill=blue!20] (R0) at (3,6) {$R_0$};
\node[rounded corners,rectangle,draw,fill=blue!20] (m1) at (-3,4) {$\mu_1$};
\node[rounded corners,rectangle,draw,fill=blue!20] (l0) at (0,4) {$\lambda_0$};
\node[rounded corners,rectangle,draw,fill=blue!20] (r0) at (3,4) {$\rho_0$};
\node[rounded corners,rectangle,draw,fill=blue!20] (m0) at (0,2) {$\mu_0$};
   \draw(0.8,2.05) node {$=\Delta$};
\draw
(m0)--(r0) (m1)--(r1)
(m0)--(m1) (r0)--(r1) (R0)--(R1)
(m0)--(l0) (m1)--(l1)
(r0)--(R0) (r1)--(R1)
(R1)--(R2)
(R2)--(R3)
;
\fill[white] (-1.5,5)circle(.15);
\fill[white] (1.5,5)circle(.15);
\draw
(l0)--(R0) (l1)--(R1) (l0)--(l1) 
;
\begin{scope}[shift={(12,0)}]
\dottedinterval0{12}11
\node[rounded corners,rectangle,draw,fill=blue!20] (A3) at (0,12) {$R_3$};
\node[rounded corners,rectangle,draw,fill=blue!20] (N) at (0,15) {$R_n$};
  \draw(0.87,14.98) node {$=\nabla$};
\node[rounded corners,rectangle,draw,fill=blue!20] (mrl) at (3,6) {$R_0$};
\node[rounded corners,rectangle,draw,fill=blue!20] (m) at (0,4) {$\mu_0$};
 \draw(0.87,3.98) node {$=\Delta$};
\node[rounded corners,rectangle,draw,fill=blue!20] (krl) at (0,8) {$R_1$};
\node[rounded corners,rectangle,draw,fill=blue!20] (k) at (-3,6) {$\mu_1$};
\node[rounded corners,rectangle,draw,fill=blue!20] (R2) at (0,10) {$R_2$};
\draw
(m)--(k) 
(k)--(krl) (m)--(mrl)
(mrl)--(krl)
(krl)--(R2)
(R2)--(A3)
;
\end{scope}
\end{tikzpicture}
}
\vspace{-3mm}
\caption{
Hasse diagrams of $\Cong(\PP_n)$ and $\Cong(\M_n)$ (left), and of
$\Cong^\ast(\PP_n)$ and $\Cong^\ast(\M_n)$ (right).
All congruences are principal.  
}
\label{fig-CongPPn}
\end{center}
\end{figure}

\begin{table}[ht]
\begin{center}
\begin{tabular}{|r|l|c|l|} \hline
\multicolumn{1}{|c|}{\boldmath{$q$}} & \textbf{Additional conditions} & \boldmath{$(\alpha,\beta)^\sharp$} & \textbf{Reference} \\ \hline\hline
&$\alpha=\beta$ & $\Delta$ & \\ \hline
0& $(\alpha,\beta)\in{\R}\sm\Delta$ &
$\rho_0$ & Proposition \ref{prop:small_congruences:PPn}(i)\\ \hline
0& $(\alpha,\beta)\in{\L}\sm\Delta$ &
$\lambda_0$ & Proposition \ref{prop:small_congruences:PPn}(ii)\\ \hline
0& $(\alpha,\beta)\not\in{\L}\cup{\R}$ &
$R_0$ & Proposition \ref{prop:joins2:PPn}(i)\\ \hline
1& $\alh=\beh$, $\al\not=\be$ &
$\mu_1$ & Proposition \ref{prop:small_congruences:PPn}(iii)\\ \hline
1& $(\widehat{\alpha},\widehat{\beta})\in{\R}\sm\De$ &
$\rho_1$ & Proposition \ref{prop:joins2:PPn}(ii)\\ \hline
1& $(\widehat{\alpha},\widehat{\beta})\in{\L}\sm\De$ &
$\lambda_1$ & Proposition \ref{prop:joins2:PPn}(iii)\\ \hline
1& $(\widehat{\alpha},\widehat{\beta})\not\in{\L}\cup{\R}$ &
$R_1$ & Proposition \ref{prop:joins2:PPn}(iv)\\ \hline
$\geq2$&  $\al\not=\be$ & $R_q$ & Proposition \ref{prop-aa3:PPn}\\ \hline
\end{tabular}
\caption{The principal congruences $(\alpha,\beta)^\sharp$ on $\PP_n$, 
with $q=\rank(\alpha)\geq\rank(\beta)$.}
\label{PPnCongGens}
\end{center}
\end{table}

We now embark on the task of describing all the principal congruences of $\PP_n$ and $\M_n$.
The proof will be couched in terms of $\PP_n$; to obtain a proof for $\M_n$ the reader should check in relevant places that various partitions belong to $\M_n$ (provided that those on which they depend do as well). These checks are trivial, and will not be explicitly indicated.
The sequence of lemmas follows a similar pattern as we had in Section \ref{sect:Pn} (minus the lemmas dealing with the congruences $R_N,\lambda_N,\rho_N$, $\mu_N$ that no longer exist in our present set-up).
We will state all the lemmas, together with the references to their counterparts from Section~\ref{sect:Pn}, but we will not rehearse all the proof details; instead we will point at the instances where planarity of partitions plays a role.
The references to Liber's Theorem \ref{thm:In} should be replaced by those to Fernandes' Theorem \ref{thm:On}.

\begin{lemma}[cf.~Lemma \ref{lem:mn}(i)]
\label{lem:mn:PPn}
Let $\xi\in\Cong(\PP_n)$.
If $\Rb_1\sub\xi$, then $\mu_1\sub\xi$.  Thus, $\mu_1=\Rb_1^\sharp$.
\end{lemma}

\pf 
To make the proof of Lemma \ref{lem:mn}(i) work here, we observe that, given
${\al=\partIV{A_0}{A_1}\cdots{A_r}{B_0}{B_1}\cdots{B_s}\in\PP_n}$,
the partitions
$\ga=\partIV{A_0}{A_1}\cdots{A_r}{1}{2}\cdots{n}$ and $\de=\partIV{1}{2}\cdots{n}{B_0}{B_1}\cdots{B_s}$ also belong to $\PP_n$ (as easily follows from Lemma~\ref{lem:nested_or_separated}),
as do
$\si=\partpermI11$ and $\tau=[\emptyset]$ trivially.
\epf

\begin{prop}[cf.~Proposition \ref{prop:small_congruences}]
\label{prop:small_congruences:PPn}
The relations $\rho_0,\lam_0,\mu_1$ are all principal congruences on $\PP_n$.  Moreover, if $\al,\be\in\PP_n$, then
\[
\textup{(i)}\ 
\rho_0=\cg{\al}{\be}\iff (\al,\be)\in\rho_0\sm\De,\ \ 
\textup{(ii)}\ \lam_0=\cg{\al}{\be}\iff (\al,\be)\in\lam_0\sm\De,\ \ 
\textup{(iii)}\ \mu_1=\cg{\al}{\be}\iff (\al,\be)\in\mu_1\sm\De.
\]
\end{prop}

\pf 
Here we need to deal with the fact that the partition $\tau=\partV ij{k_1}{k_{n-2}}{m}{C}{B_2}{B_r}$ constructed during the proof of part~(i) 
of Proposition \ref{prop:small_congruences}
may fail to be planar.  However, since $\si$ is itself planar, we can ensure that $\tau$ is planar by assuming that $i<j$, and that $m$ is the minimal element of $\bn$ such that the $\coker(\si)$-class of $m$ is nontrivial (this also ensures that $m=\min(B_1)$, where $B_1$ is the $\coker(\si)$-class of $m$).  \epf

\begin{prop}[cf.~Proposition \ref{prop:joins2}]
\label{prop:joins2:PPn}
The congruences $R_0$, $\rho_1$, $\lam_1$ and $R_1$ are all principal.  Moreover, if $\al,\be\in\PP_n$, then
\begin{itemize}\begin{multicols}{2}
\itemit{i} $R_0=\cg\al\be\iff (\al,\be)\in R_0 \sm(\rho_0\cup\lam_0)$,
\itemit{ii} $\rho_1=\cg\al\be\iff (\al,\be)\in\rho_1\sm(\mu_1\cup\rho_0)$,
\itemit{iii} $\lam_1=\cg\al\be\iff (\al,\be)\in\lam_1\sm(\mu_1\cup\lam_0)$, 
\itemit{iv} $R_1=\cg\al\be\iff (\al,\be)\in R_1\sm ( R_0 \cup \lam_1 \cup \rho_1 )$.
\end{multicols}
\end{itemize}
\end{prop}

\pf 
The partition $\de=\partVI{k}{i,j}{j_1}{j_{n-3}}{k}{i,j}{j_1}{j_{n-3}}$ constructed in the second subcase of case (a) in the proof of Proposition \ref{prop:joins2}(ii) is nonplanar if and only if $i<k<j$ or $j<k<i$.  However, since $(j,k)$ was an arbitrary pair from $\coker(\al)\sm\coker(\be)$, we could just reverse the roles of $j$ and $k$, and so define $\de=\partVI{j}{i,k}{j_1}{j_{n-3}}{j}{i,k}{j_1}{j_{n-3}}$ if either of the above inequalities hold.   \epf

\begin{lemma}[cf.~Lemma \ref{lemma-aa1}]
\label{lemma-aa1:PPn}
For every $\alpha\in\PP_n$, there exists $\beta\in\O_n$ with $\rank(\beta)=\rank(\alpha)$, $\al=\alpha\beta\alpha$ and $\be=\be\al\be$.
\end{lemma}

\pf
The planarity of $\be\in\I_n$, as constructed in the proof of Lemma \ref{lemma-aa1}, follows from Lemma \ref{lem:nested_or_separated}(i).
\epf

%\newpage

\begin{lemma}[cf.~Lemma \ref{lemma-aa1b}]
\label{lemma-aa1b:PPn}
Let $\al,\be\in J_q$ with $q\geq1$.  
\bit
\itemit{i} If $(\al,\be)\not\in{\R}$, then there exists $\gamma\in J_q$ such that precisely one of $\gamma\alpha,\gamma\beta$ belongs to $J_q$.
\itemit{ii} If $(\al,\be)\not\in{\L}$, then there exists $\gamma\in J_q$ such that precisely one of $\alpha\gamma,\beta\gamma$ belongs to $J_q$.
\end{itemize}
\end{lemma}

\pf 
 The partition
$\ga =
\Big( 
{ \scriptsize \renewcommand*{\arraystretch}{1}
\begin{array} {\c|\c|\c|\c|\c|\c|\c|\cend}
a_1 \:&\: \cdots \:&\: a_{q-1} \:&\: i \:&\: a_q,j \:&\: x_3 \:&\: \cdots \:&\: x_{n-q} \\ \cline{5-8}
a_1 \:&\: \cdots \:&\: a_{q-1} \:&\: i \:&\: a_q,j \:&\: x_3 \:&\: \cdots \:&\: x_{n-q}
\rule[0mm]{0mm}{2.7mm}
\end{array} 
}
\hspace{-1.5 truemm} \Big)
$
constructed during Subcase 2.2 of the proof of Lemma \ref{lemma-aa1b} could be nonplanar.  However, in the case that $\al$ is planar, we can easily modify the definition of $\ga$ to ensure that $\ga$ is planar and still fulfils its purpose in the proof.  Indeed, first, we assume that $i,j\in C_1$ satisfy $i<j$.  We also note that, since $\al$ is planar,  no element $x\in\dom(\al)$ can satisfy $i<x<j$.  So, for each $x\in\dom(\al)$, we have either
$x<i$ or $j<x$.
If there are no elements $x\in\dom(\al)$ with $x<i$ then, renaming the sets $A_1,\ldots,A_q$ if necessary, we may assume that $a_q=\min(\dom(\al))$, and we define $\ga$ as above.  If there is some  $x\in\dom(\al)$ with $x<i$ then, again renaming if necessary, we may assume that $a_q=\max\set{x\in\dom(\al)}{x<i}$, and we instead define 
$
\ga =
\Big( 
{ \scriptsize \renewcommand*{\arraystretch}{1}
\begin{array} {\c|\c|\c|\c|\c|\c|\c|\cend}
a_1 \:&\: \cdots \:&\: a_{q-1} \:&\: j \:&\: a_q,i \:&\: x_3 \:&\: \cdots \:&\: x_{n-q} \\ \cline{5-8}
a_1 \:&\: \cdots \:&\: a_{q-1} \:&\: j \:&\: a_q,i \:&\: x_3 \:&\: \cdots \:&\: x_{n-q}
\rule[0mm]{0mm}{2.7mm}
\end{array} 
}
\hspace{-1.5 truemm} \Big)
$.
\epf

\begin{lemma}[cf.~Lemma \ref{lemma-aa2}]
\label{lemma-aa2:PPn}
If $2\leq q\leq n$, then $R_q=\overline{R}_q^\sharp$.
\epfres
\end{lemma}

\begin{prop}[cf.~Proposition \ref{prop-aa3}]
\label{prop-aa3:PPn}
Let $2\leq q\leq n$.  Then $R_q$ is a principal congruence on~$\PP_n$.  Moreover, if $\al,\be\in\PP_n$, then $R_q=\cg\al\be\iff(\al,\be)\in R_q\sm R_{q-1}$. \epfres
\end{prop}

\section{The Brauer monoid \boldmath{$\B_n$}}
\label{sect:Bn}

We now move onto the Brauer monoid $\B_n$,
consisting of all partitions $\alpha\in\P_n$ having all the blocks of size~$2$.
Again, $\Cong(\B_n)$ is easily described for $n\leq2$, so we will assume $n\geq3$ for the duration of this section.

There are some differences between $\B_n$ and the preceding monoids that are worth paying attention to.
Firstly, there will be a difference in behaviour depending on the parity of $n$, the reason for which is in the $\gJ$-class structure of $\B_n$.
Perhaps most notably, for $n$ even, the basic diamond shapes forming the lower part of 
$\Cong(\B_n)$ are due to a new retraction.
And finally, the symmetric inverse monoid $\I_n$, which was a linchpin in several parts of the argument of the previous section, does not embed in~$\B_n$ (as can be seen by comparing sizes of maximal subgroups in the top two $\gJ$-classes of both monoids); consequently, modifications are needed for that aspect of the argument.

The $\gJ$-classes of $\B_n$ are $J_n,J_{n-2},\dots,J_z$,
where $J_q=\set{\alpha\in\B_n}{ \rank(\alpha)=q}$, and  $z\in\{0,1\}$ satisfies
$z\equiv n \pmod{2}$.
This meaning of $z$ will be fixed throughout the section.
The corresponding ideals are $I_q=J_q\cup\dots\cup J_z$, and they induce Rees congruences
$R_q$.
The minimal ideal of $\B_n$ is $J_z=I_z$.
As in $\P_n$, we single out one particular copy of $\S_q$ inside $J_q$, but the construction is slightly different:
for $\sigma\in\S_q$, we let $\overline{\sigma}=\partI{1}{q}{q+1,q+2}{n-1,n}{1\sigma}{q\sigma}{q+1,q+2}{n-1,n}\in J_q$.

For every $q\in\{z+2,\dots,n\}$ and every normal subgroup $N$ of $\S_q$ we have an IN-pair $(I_{q-2},\overline{N})$, yielding a congruence 
\[
R_N=R_{q-2}\cup\nu_N=\Delta\cup\nu_N\cup (I_{q-2}\times I_{q-2}),
\]
where $\nu_N=\set{ (\alpha,\beta)\in J_q\times J_q}{ \alpha\H\beta,\ \phi(\alpha,\beta)\in N}$.
(Note that the proof of Lemma \ref{lem:nu_phi} needs to be slightly modified to work for $\B_n$: namely, the partitions $\ga,\de$ defined in the early stages of the proof need to be replaced by $\ga=\partI{A_1}{A_q}{C_1}{C_r}{1}{q}{q+1,q+2}{n-1,n}$ and $\de=\partI{1}{q}{q+1,q+2}{n-1,n}{B_1}{B_q}{D_1}{D_s}$.)
The congruences~$R_q,R_N$ form a chain (again, due to the fact that ideals of $\B_n$ and normal subgroups of all $\S_q$ form chains; see Proposition \ref{prop-CR2}), and we will see that they account for all the upper congruences, as well as some of the lower congruences.

The remaining lower congruences arise from retractions, which depend on the parity of $n$.
For $n$ odd, we just take the identity retraction $I_1\rightarrow I_1$ to obtain a retractable IN-pair
$(I_1,\{\overline{\id}_3\})$, which in turn yields three congruences
$\lambda_1$, $\rho_1$ and $\mu_1=\Delta$.
However, for $n$ even, we have a new non-trivial retraction!

\begin{defn}
\label{defn:retr_Bn}
For a Brauer element $\al$ of rank $2$ with transversals $\{i,j'\}$ and $\{k,l'\}$, let $\alh$ 
be the Brauer element of rank $0$ obtained by replacing these transversals by the non-transversals $\{i,k\}$ and $\{j',l'\}$.  For $\al\in\B_n$ with $\rank(\al)=0$, let $\alh=\al$.  
\end{defn}

We will make no use of the mapping from Definition \ref{defn:hat} in this section, so there should be no confusion arising from our re-use of the $\alh$ notation.  An example calculation is given in Figure \ref{fig:beh}.

\begin{figure}[ht]
\begin{center}
\begin{tikzpicture}[scale=.5]
\begin{scope}[shift={(0,0)}]	
\uvs{1,...,6}
\lvs{1,...,6}
\uarc12
\uarc45
\darcx34{.2}
\darcx25{.5}
\stlines{3/1,6/6}
\draw(0.6,1)node[left]{$\al=$};
\draw[->](7.5,1)--(9.5,1);
\end{scope}
\begin{scope}[shift={(10,0)}]	
\uvs{1,...,6}
\lvs{1,...,6}
\uarc12
\uarcx45{.2}
\uarcx36{.5}
\darcx34{.2}
\darcx25{.5}
\darcx16{.8}
\draw(6.4,1)node[right]{$=\alh$};
\end{scope}
\end{tikzpicture}
\end{center}
\vspace{-5mm}
\caption{A rank-$2$ Brauer element $\al\in\B_6$ (left), with $\alh\in\B_6$ (right).}
\label{fig:beh}
\end{figure}
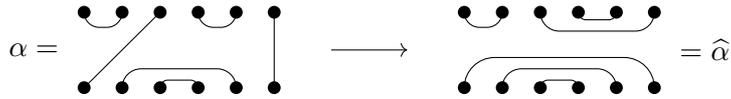

\begin{lemma}
\label{Bretract}
For $n$ even, the map $f: I_2\rightarrow I_0:\al\mt\alh$ is a retraction.
\end{lemma}

\pf
We know that $\al f=\al$ for $\al\in I_0$, so we just need to show that $\widehat{\al\be}=\alh\beh$ for all $\alpha,\beta\in I_2$.
If $\alpha,\beta\in I_0$ the assertion is obvious, so assume without loss of generality that $\alpha\in J_2$,
and write $\alpha = \partV{a}{b}{A_1}{A_{m-1}}{x}{y}{B_1}{B_{m-1}}$.  We may also write $\be=\partV{c}{d}{C_1}{C_{m-1}}{u}{v}{D_1}{D_{m-1}}$ or $\partIII{c,d}{C_1}\cdots{C_{m-1}}{u,v}{D_1}\cdots{D_{m-1}}$.  In either case, we clearly have $\alh\beh=\partIII{a,b}{A_1}\cdots{A_{m-1}}{u,v}{D_1}\cdots{D_{m-1}}$.  It is also clear that
$\al\be = \partV{a}{b}{A_1}{A_{m-1}}{u}{v}{D_1}{D_{m-1}}$
or $\partV{a}{b}{A_1}{A_{m-1}}{v}{u}{D_1}{D_{m-1}}$ or~$\partIII{a,b}{A_1}\cdots{A_{m-1}}{u,v}{D_1}\cdots{D_{m-1}}$, from which $\widehat{\al\be}=\alh\beh$ follows. \epf

As with $\P_n$, one may show that no ideals larger than $I_1$ ($n$ odd) or $I_2$ ($n$ even) are retractable.
When~$n$ is odd, we have the obvious retractable IN-pair $\C_1=(I_1,\{\overline\id_3\})$, but there are several retractable pairs when $n$ is even.

\begin{lemma}
\label{BCtriples}
For $n$ even, each of $\C_0=(I_0,\{\overline{\id}_2\})$,
$\C_{\S_2}=(I_0,\overline{\S}_2)$,
$\C_2=(I_2,\{\overline{\id}_4\})$,
$\C_K=(I_2,\overline{K})$
is a retractable IN-pair, where $K$ denotes the Klein 4-group.
\end{lemma}

\pf
Having already verified that $f$ is a retraction, we only need to check that
$|\alpha\overline{\S}_2|=|\overline{\S}_2\alpha|=|\alpha \overline{K}|=|\overline{K}\alpha|=1$ for any
$\alpha=\partIII{A_1}{A_2}\cdots{A_{m}}{B_1}{B_2}\cdots{B_{m}}\in I_0$.

For the $\S_2$ part of the assertion, it is clearly sufficient to show
that $\alpha\;\!\overline{\id}_2=\alpha\overline{(1, 2)}$.
This 
follows immediately by noting that both these Brauer elements have upper blocks $A_1,\dots,A_m$ and the lower blocks $\{1,2\},\{3,4\},\dots,\{n-1,n\}$.

Let us now prove the assertion for $K$.
Clearly, it is sufficient to prove that
$\alpha \;\!\overline{\id}_4=\alpha\overline{\gamma}$
for all $\gamma\in K\setminus\{\id_4\}$.
The upper blocks of $\alpha\;\!\overline{\id}_4$ are $A_1,\dots,A_m$.
The lower blocks are $\{5,6\},\dots,\{n-1,n\}$, and two blocks
partitioning $\{1,2,3,4\}$.
Without loss we may assume that these last two blocks are in fact $\{1,2\}$ and $\{3,4\}$.
This means that in the product graph $\Pi(\alpha,\overline{\id}_4)$
there is a path from $1'$ to $2'$ and a path from $3'$ to $4'$.
But this is equivalent to saying that there is a path $\pi_{12}$ from $1''$ to $2''$, and a path $\pi_{34}$ from $3''$ to $4''$.
Now consider the Brauer element $\alpha\overline{\gamma}$.
Its upper blocks are also $A_1,\dots,A_m$, and the lower blocks include $\{5,6\},\dots,\{n-1,n\}$.
Notice that the paths $\pi_{12}$ and $\pi_{34}$ are also present in the product graph $\Pi(\alpha,\overline{\gamma})$, as they involve only vertices from $\bn''$.

Consider the case in which $\gamma=(1,2)(3,4)$. Then concatenating the edge $2'\to1''$ with the path $\pi_{12}$ and the edge $2''\to1'$ yields a path from $2'$ to $1'$ in $\Pi(\alpha,\overline{\gamma})$.
It follows that $\{1,2\}$ and $\{3,4\}$ are lower blocks of~$\alpha\overline{\gamma}$, and hence $\alpha\overline{\gamma}=\alpha\;\!\overline{\id}_4$.
The cases in which $\gamma=(1,3)(2,4)$ or $(1,4)(2,3)$ are analogous. \epf

We denote the congruences arising from these pairs
$\lambda_0$, $\rho_0$, $\mu_0$, $\lambda_{\S_2}$, $\rho_{\S_2}$, $\mu_{\S_2}$,
$\lambda_2$, $\rho_2$, $\mu_2$, $\lambda_K$, $\rho_K$, $\mu_K$, respectively.

\begin{thm}\label{CongBn}
Let $n\geq3$, and let $\B_n$ be the Brauer monoid of degree $n$.    Also let $z\in\{0,1\}$ be such that $n\equiv z\pmod{2}$.
\bit
\itemit{i} The ideals of $\B_n$ are the sets $I_q$, for $q=z,z+2,\ldots,n$, yielding the Rees congruences $R_q=R_{I_q}$, as in Definition \ref{defn:Rees}.
\itemit{ii} The proper IN-pairs of $\B_n$ are of the form $(I_{q-2},\overline N)$, for $q=z+2,z+4,\ldots,n$ and for any non-trivial normal subgroup $N$ of $\S_q$, yielding the congruences $R_N=R_{I_{q-2},\overline N}$, as in Definition \ref{defn:RC}.
\itemit{iii} For $n$ odd, the only retractable IN-pair of $\B_n$ is $\C_1=(I_1,\{\overline\id_3\})$, yielding the congruences $\lambda_1$, $\rho_1$, $\mu_1$, as in Definition \ref{defn:lrmC}.
\itemit{iv} For $n$ even, the retractable IN-pairs of $\B_n$ are $\C_0=(I_0,\{\overline{\id}_2\})$, $\C_{\S_2}=(I_0,\overline{\S}_2)$, $\C_2=(I_2,\{\overline{\id}_4\})$ and $\C_K=(I_2,\overline{K})$, yielding the congruences $\lambda_0$, $\rho_0$, $\mu_0$, $\lambda_{\S_2}$, $\rho_{\S_2}$, $\mu_{\S_2}$,
$\lambda_2$, $\rho_2$, $\mu_2$, $\lambda_K$, $\rho_K$, $\mu_K$, respectively, as in Definition \ref{defn:lrmC}.  Here, $K$ denotes the Klein 4-group.
\itemit{v} The above congruences are distinct, and they exhaust all the congruences of $\B_n$.
\itemit{vi} The Hasse diagram of the congruence lattice $\Cong(\B_n) $ is shown in Figure \ref{fig:Hasse_Bn}.
\itemit{vii} The $\ast$-congruences of $\B_n$ are the same, but with $\lam_1$, $\rho_1$ ($n$ odd) or $\lambda_0$, $\lambda_{\S_2}$, $\lambda_2$, $\lambda_K$, $\rho_0$, $\rho_{\S_2}$, $\rho_2$, $\rho_K$ ($n$~even) excluded.
\eit
\end{thm}

\pf
The overall structure of the proof is the same as that of Theorem \ref{thm-CongPn}.
That all the listed relations are congruences follows from Proposition \ref{prop-AreCongs}, and
that they form sublattices of $\Cong(\B_n)$ as shown in Figure~\ref{fig:Hasse_Bn}
follows from Propositions \ref{prop-CR2}--\ref{prop-CR5}.
It remains to be proved that there are no other congruences, which can be done by showing that all the principal congruences are among the listed ones (cf.~Lemma~\ref{lem:principal}).
The principal congruences are shaded blue in Figure \ref{fig:Hasse_Bn}, 
and we prove that they are principal in the remainder of the section, characterising their minimal generating sets.
Again, the proof is completed by observing that this covers all possible pairs; see Table \ref{BnCongGens}.
\epf

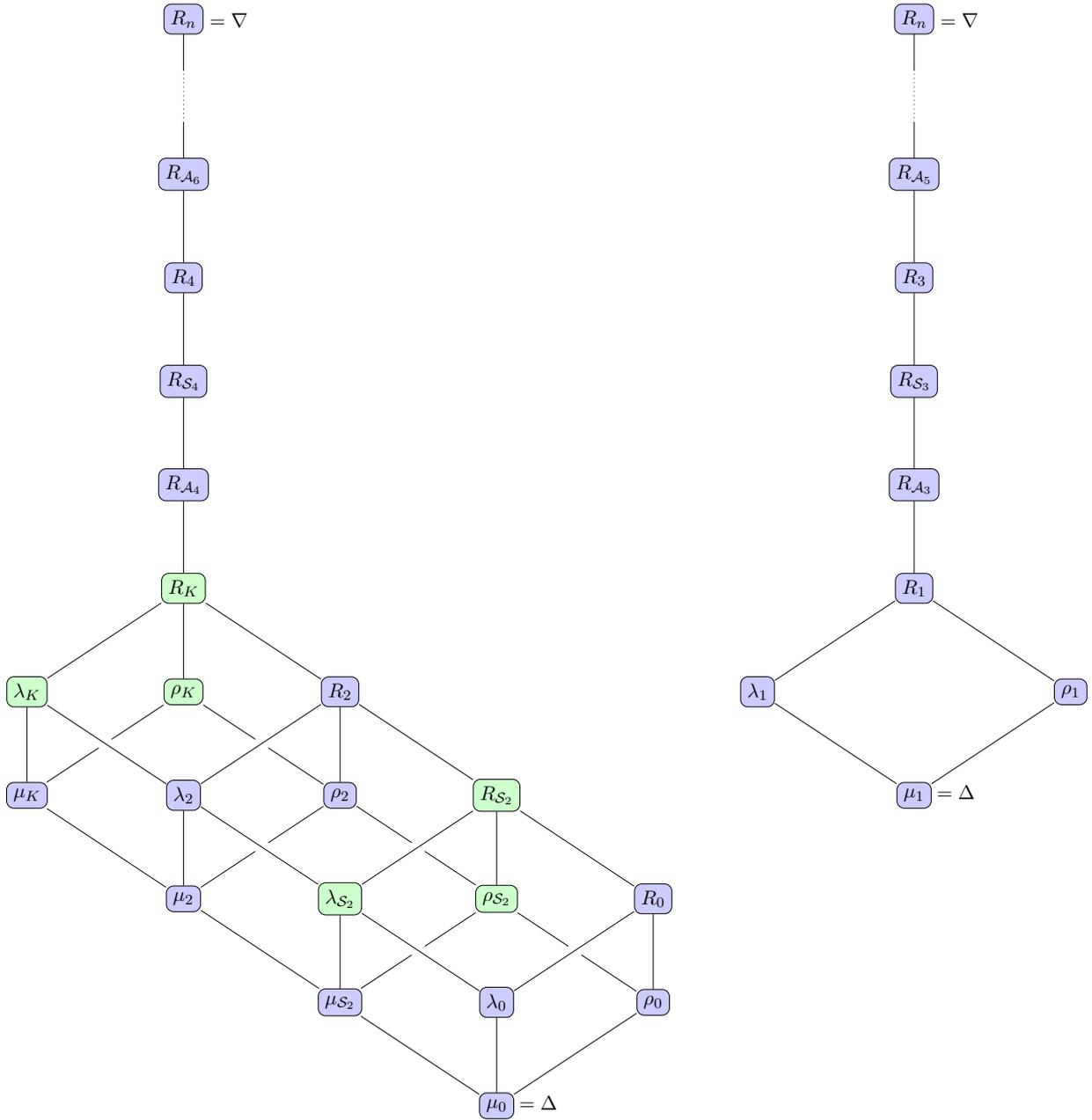
\begin{figure}[ht]
\begin{center}
\scalebox{0.78}{
\begin{tikzpicture}[scale=1]
\dottedinterval0{16}11
\node[rounded corners,rectangle,draw,fill=blue!20] (N) at (0,19) {$R_n$};
  \draw(0.87,18.98) node {$=\nabla$};
\node[rounded corners,rectangle,draw,fill=blue!20] (A6) at (0,16) {$R_{\A_6}$};
\node[rounded corners,rectangle,draw,fill=blue!20] (R4) at (0,14) {$R_4$};
\node[rounded corners,rectangle,draw,fill=blue!20] (S4) at (0,12) {$R_{\S_4}$};
\node[rounded corners,rectangle,draw,fill=blue!20] (A4) at (0,10) {$R_{\A_4}$};
\node[rounded corners,rectangle,draw,fill=green!20] (krl) at (0,8) {$R_K$};
\node[rounded corners,rectangle,draw,fill=green!20] (kr) at (-3,6) {$\lambda_K$};
\node[rounded corners,rectangle,draw,fill=green!20] (kl) at (0,6) {$\rho_K$};
\node[rounded corners,rectangle,draw,fill=blue!20] (k) at (-3,4) {$\mu_K$};
\node[rounded corners,rectangle,draw,fill=blue!20] (erl) at (3,6) {$R_2$};
\node[rounded corners,rectangle,draw,fill=blue!20] (er) at (0,4) {$\lambda_2$};
\node[rounded corners,rectangle,draw,fill=blue!20] (el) at (3,4) {$\rho_2$};
\node[rounded corners,rectangle,draw,fill=blue!20] (e) at (0,2) {$\mu_2$};
\node[rounded corners,rectangle,draw,fill=green!20] (nrl) at (6,4) {$R_{\S_2}$};
\node[rounded corners,rectangle,draw,fill=green!20] (nr) at (3,2) {$\lambda_{\S_2}$};
\node[rounded corners,rectangle,draw,fill=green!20] (nl) at (6,2) {$\rho_{\S_2}$};
\node[rounded corners,rectangle,draw,fill=blue!20] (n) at (3,0) {$\mu_{\S_2}$};
\node[rounded corners,rectangle,draw,fill=blue!20] (rl) at (9,2) {$R_0$};
\node[rounded corners,rectangle,draw,fill=blue!20] (r) at (6,0) {$\lambda_0$};
\node[rounded corners,rectangle,draw,fill=blue!20] (l) at (9,0) {$\rho_0$};
\node[rounded corners,rectangle,draw,fill=blue!20] (D) at (6,-2) {$\mu_0$};
\draw(6.8,-1.95) node {$=\Delta$};
\draw
(D)--(n)--(e)--(k)
(l)--(nl)--(el)--(kl)
(k)--(kl) (e)--(el) (n)--(nl) (D)--(l) 
(k)--(kr) (e)--(er) (n)--(nr) (D)--(r) 
(l)--(rl) (nl)--(nrl) (el)--(erl) (kl)--(krl)
(rl)--(nrl)--(erl)--(krl)--(A4)--(S4)--(R4)--(A6)
;
\fill[white] (-1.5,5)circle(.15);
\fill[white] (1.5,3)circle(.15);
\fill[white] (4.5,1)circle(.15);
\fill[white] (1.5,5)circle(.15);
\fill[white] (4.5,3)circle(.15);
\fill[white] (7.5,1)circle(.15);
\draw
(r)--(nr)--(er)--(kr)
(r)--(rl) (nr)--(nrl) (er)--(erl) (kr)--(krl) 
;
\begin{scope}[shift={(14,0)}]
\dottedinterval0{16}11
\node[rounded corners,rectangle,draw,fill=blue!20] (N) at (0,19) {$R_n$};
  \draw(0.87,18.98) node {$=\nabla$};
\node[rounded corners,rectangle,draw,fill=blue!20] (A6) at (0,16) {$R_{\A_5}$};
\node[rounded corners,rectangle,draw,fill=blue!20] (R4) at (0,14) {$R_3$};
\node[rounded corners,rectangle,draw,fill=blue!20] (S4) at (0,12) {$R_{\S_3}$};
\node[rounded corners,rectangle,draw,fill=blue!20] (A4) at (0,10) {$R_{\A_3}$};
\node[rounded corners,rectangle,draw,fill=blue!20] (krl) at (0,8) {$R_1$};
\node[rounded corners,rectangle,draw,fill=blue!20] (kr) at (-3,6) {$\lambda_1$};
\node[rounded corners,rectangle,draw,fill=blue!20] (erl) at (3,6) {$\rho_1$};
\node[rounded corners,rectangle,draw,fill=blue!20] (er) at (0,4) {$\mu_1$};
\draw(0.8,4.05) node {$=\Delta$};
\draw
(krl)--(kr)--(er)--(erl)--(krl)--(A4)--(S4)--(R4)--(A6)
;
\end{scope}
\end{tikzpicture}
}
\end{center}
\vspace{-7mm}
\caption{Hasse diagram for the congruence lattice $\Cong(\B_n)$ in the case of $n$ even (left) and $n$ odd (right).  Congruences shaded blue are principal; those shaded green are minimally generated by two pairs of Brauer elements.}
\label{fig:Hasse_Bn}
\end{figure}

\begin{table}
\begin{center}
\begin{tabular}{|r|l|c|l|} \hline
\multicolumn{1}{|c|}{\boldmath{$q$}} & \textbf{Additional conditions} & \boldmath{$(\alpha,\beta)^\sharp$} & \textbf{Reference} \\ \hline\hline
&$\alpha=\beta$ & $\Delta$ & \\ \hline
$z$& $(\alpha,\beta)\in{\R}\sm\Delta$ &
$\rho_z$ & Proposition \ref{prop:rl_Bn}(i)\\ \hline
$z$& $(\alpha,\beta)\in{\L}\sm\Delta$ &
$\lambda_z$ & Proposition \ref{prop:rl_Bn}(ii)\\ \hline
$z$& $(\alpha,\beta)\not\in{\L}\cup{\R}$ &
$R_z$ & Proposition \ref{prop:R01_Bn}\\ \hline
2& $(\alpha,\beta)\in{\H}\sm\Delta$ &
$\mu_{\S_2}$ & Proposition \ref{prop:nek}(i)\\ \hline
2& $\alh=\beh$, $(\al,\be)\not\in{\H}$ &
$\mu_2$ & Proposition \ref{prop:nek}(ii)\\ \hline
2& $(\widehat{\alpha},\widehat{\beta})\in{\R}\sm\De$ &
$\rho_2$ & Proposition \ref{prop:joins2_Bn}(i)\\ \hline
2& $(\widehat{\alpha},\widehat{\beta})\in{\L}\sm\De$ &
$\lambda_2$ & Proposition \ref{prop:joins2_Bn}(ii)\\ \hline
2& $(\widehat{\alpha},\widehat{\beta})\not\in{\L}\cup{\R}$ &
$R_2$ & Proposition \ref{prop:joins2_Bn}(iii)\\ \hline
4& $(\alpha,\beta)\in{\H}\setminus \Delta$, \ $\phi(\al,\be)\in K$ &
$\mu_K$ & Proposition \ref{prop:nek}(iii)\\ \hline
$\geq3$&  $(\alpha,\beta)\not\in{\H}$ & $R_q$ & Proposition \ref{prop:chain_Bn}\\ \hline
$\geq3$&  $(\alpha,\beta)\in{\H}\setminus \Delta$, \ $\phi(\al,\be)\not\in K$ & $R_N$ & Proposition \ref{prop:chain_Bn}\\ \hline
\end{tabular}
\caption{The principal congruences $(\alpha,\beta)^\sharp$ on $\B_n$, 
with $q=\rank(\alpha)\geq\rank(\beta)$.
Here, $K$ is the Klein 4-group.  In the last row, $N$ stands for the normal closure in $\S_q$ of $\phi(\alpha,\beta)$.}
\label{BnCongGens}
\end{center}
\end{table}

As was the case with $\P_n$, it will be convenient to give explicit criteria for membership in each of the lower congruences on $\B_n$.  For arbitrary $n$:
\begin{itemize}\begin{multicols}{2}
\item $\rho_z=\De \cup \set{(\al,\be)\in I_z\times I_z}{\ker(\al)=\ker(\be)}$,

\item $\lam_z=\De \cup \set{(\al,\be)\in I_z\times I_z}{\coker(\al)=\coker(\be)}$,
\end{multicols}\eit
and for even $n$:
\begin{itemize}
\begin{multicols}{2}
\item $\rho_{\S_2}=\rho_0\cup\big( {\H}\cap(J_2\times J_2)\big)$,
\item $\mu_{\S_2}=\De\cup\big( {\H}\cap(J_2\times J_2)\big)$,
\item $\lam_{\S_2}=\lam_0\cup\big( {\H}\cap(J_2\times J_2)\big)$,
\item[] ~
\end{multicols}
\begin{multicols}{2}
\item $\rho_2=\De\cup\set{(\al,\be)\in I_2\times I_2}{\ker(\alh)=\ker(\beh)}$,
\item $\mu_2=\De\cup\set{(\al,\be)\in I_2\times I_2}{\alh=\beh}$,
\item $\lam_2=\De\cup\set{(\al,\be)\in I_2\times I_2}{\coker(\alh)=\coker(\beh)}$,
\item[] ~
\end{multicols}
\begin{multicols}{3}
\item $\rho_K=\rho_2\cup\nu_K$,
\item $\lam_K=\lam_2\cup\nu_K$,
\item $\mu_K=\mu_2\cup\nu_K$,
\end{multicols}
\eit
where $\nu_K=\set{(\al,\be)\in J_4\times J_4}{\al\H\be,\ \phi(\al,\be)\in K}$.

We now embark on the task of describing the generating pairs for the principal congruences indicated in Figure \ref{fig:Hasse_Bn}.
We begin with the lower congruences, where the arguments, of necessity, will depend on the parity of $n$. First we offer a unified treatment of the congruences $\lambda_z$, $\rho_z$, $R_z$.  A \emph{2-partition} of a set is a partition of the set in which each block has size 2.

\newpage

\begin{prop}\label{prop:rl_Bn}
The relations $\rho_z$ and $\lam_z$ are principal congruences on $\B_n$.  Moreover, if $\al,\be\in\B_n$, then
\begin{itemize}\begin{multicols}{2}
\itemit{i} $\rho_z=\cg\al\be\iff (\al,\be)\in\rho_z\sm\De$, 
\itemit{ii} $\lam_z=\cg\al\be\iff (\al,\be)\in\lam_z\sm\De$.
\end{multicols}\end{itemize}
\end{prop}

\pf 
By duality, it suffices to prove (i).  
Note that $\cg\al\be\not=\rho_z$ if $(\al,\be)\not\in\rho_z$ or if $(\al,\be)\in\De$.  Conversely, fix some $(\al,\be)\in\rho_z\sm\De$ and, for simplicity, write ${\mathrel\xi}=\cg\al\be$.  For an integer $i$, write $H_i=\{2i-1,2i\}$.  For $\si\in J_z$  we write~$\si'$ for the unique element of~$\B_n$ with $\ker(\si')=\ker(\si)$ and $H_1,H_2,\ldots,H_{\lfloor n/2\rfloor}$ all being $\coker(\si')$-classes, noting that this forces $\codom(\si')=\{n\}$ in the case that $n$ is odd.  We claim that $\si\mathrel\xi\si'$ for all $\si\in J_z$.  
Once we prove the claim, part~(i) of the lemma will follow:
indeed, for any two $(\ga,\de)\in\rho_z$ we will have $\gamma\mathrel\xi\gamma'=\delta'\mathrel\xi\delta$.
To prove the claim we consider the even and odd cases separately.

\bigskip\noindent {\bf Case 1.}  Suppose first that $n=2m$ is even, and write $\si=\partII{A_1}\cdots{A_m}{B_1}\cdots{B_m}$,
where $\min(B_1)<\cdots<\min(B_m)$.  Define $h(\si)=m$ if $\si=\si'$ or else
$h(\si)=\min\set{i\in\bm}{B_i\not=H_i}$.  We proceed by descending induction
on~$h(\si)$.  If $h(\si)=m$, then there is nothing to show, since $\si=\si'$,
so suppose $h(\si)<m$.  Clearly it suffices to show that $\si\mathrel\xi\pi$ for some $\pi\in\B_n$ with $\ker(\pi)=\ker(\si)$ and $h(\pi)>h(\si)$.  Write $k=h(\si)$.
Then $B_i=H_i$ for each $1\leq i\leq k-1$, 
but $B_k\not=H_k$.  
By the assumption on the values of $\min(B_i)$, it follows that $B_k=\{2k-1,a\}$ and $B_{k+1}=\{2k,b\}$ for some $a,b\geq 2k+1$.  Since $\coker(\al)\not=\coker(\be)$, we may fix some $\coker(\al)$-class $\{x,y\}$ that is not a $\coker(\be)$-class.  Since $\rank(\be)=0$, it follows that $\be$ has cokernel-classes $\{x,u\}$ and $\{y,v\}$ for some $u,v$.  Fix an arbitrary 2-partition $\{C_1,\ldots,C_{m-2}\}$ of $\bn\sm\{u,v,x,y\}$, and define
\[
\tau = 
\Big( 
{ \scriptsize \renewcommand*{\arraystretch}{1}
\begin{array} {\c|\c|\c|\c|\c|\c|\c|\c|\c|\cend}
x \:&\: y \:&\: u \:&\: v \:&\: C_1 \:& \multicolumn{4}{\c|}{\cdots} &\: C_{m-2}  \\ \cline{5-10}
2k-1 \:&\: 2k \:&\: a \:&\: b \:&\:  B_1 \:&\: \cdots \:&\: B_{k-1} \:&\: B_{k+2} \:&\: \cdots & B_m
\rule[0mm]{0mm}{2.7mm}
\end{array} 
}
\hspace{-1.5 truemm} \Big).
\]
Then $\si=\si\be\tau \mathrel\xi \si\al\tau$, and $h(\si\al\tau)>k$, so the proof of the claim is complete in this case.  

\bigskip\noindent {\bf Case 2.}  Now suppose $n=2m+1$.  
Write $\si=\partIV a{A_1}\cdots{A_m}b{B_1}\cdots{B_m}$.
We first prove the sub-claim that $\si\mathrel\xi\pi$ for some $\pi\in\B_n$ with
$\ker(\pi)=\ker(\si)$ and $\codom(\pi)=n$.  Indeed, if $b=n$, then there is
nothing to prove, so suppose $b\not=n$.  Without loss of generality, we may assume that $n\in B_1$, and we may write $B_1=\{w,n\}$.  
We will prove that $\si\mathrel\xi\pi$, where $\pi=\partVI a{A_1}{A_2}{A_m}n{b,w}{B_2}{B_m}$.
To do this, we will define a Brauer element $\tau\in\B_n$ such that $\si\al\tau=\si$ and $\si\be\tau=\pi$.  Let the transversals of $\al$ and $\be$ be $\{c,d'\}$ and $\{c,e'\}$, respectively; note that $\dom(\al)=\dom(\be)$ because $\ker(\al)=\ker(\be)$.  If $d\not=e$, then we fix any $w\in\bn\sm\{d,e\}$, and any 2-partition $\{D_1,\ldots,D_{m-1}\}$ of $\bn\sm\{d,e,w\}$, and define
$
\tau = 
\Big( 
{ \scriptsize \renewcommand*{\arraystretch}{1}
\begin{array} {\c|\c|\c|\c|\c|\cend}
d \:&\: w \:&\: e \:&\: D_1 \:&\: \cdots \:&\: D_{m-1}  \\ \cline{4-6}
b \:&\: w \:&\: n \:&\: B_2 \:&\: \cdots \:&\: B_m
\rule[0mm]{0mm}{2.7mm}
\end{array} 
}
\hspace{-1.5 truemm} \Big)
$.
The reader may check that $\tau$ has the desired properties.  Now suppose $d=e$.  Because $(\al,\be)\in\rho_1\sm\De$, we may choose some $\coker(\al)$-class $\{x,y\}$ that is not a $\coker(\be)$-class.  Then $\be$ has cokernel-classes $\{x,u\}$ and $\{y,v\}$ for some $u,v$.  Fix an arbitrary 2-partition $\{E_1,\ldots,E_{m-1}\}$ of $\bn\sm\{u,x,y\}$, and define
$
\tau = 
\Big( 
{ \scriptsize \renewcommand*{\arraystretch}{1}
\begin{array} {\c|\c|\c|\c|\c|\cend}
u \:&\: x \:&\: y \:&\: E_1 \:&\: \cdots \:&\: E_{m-1}  \\ \cline{4-6}
b \:&\: w \:&\: n \:&\: B_2 \:&\: \cdots \:&\: B_m
\rule[0mm]{0mm}{2.7mm}
\end{array} 
}
\hspace{-1.5 truemm} \Big)
$.
Again, the reader may check that $\tau$ has the desired properties.  This completes the proof of the sub-claim.  If $n=3$, then $\pi=\si'$, so the claim is established in this case.  For the rest of the proof, we assume that~$n\geq5$.

In view of the sub-claim, we may suppose without loss of generality that $b=n$, so that $\si=\partIV a{A_1}\cdots{A_m}n{B_1}\cdots{B_m}$.
Write $\tau_1=\partIV n{H_1}\cdots{H_m}n{B_1}\cdots{B_m}$ and $\tau_2=\partIV n{H_1}\cdots{H_m}n{H_1}\cdots{H_m}$.
Since $\si=\si\tau_1$ and $\si'=\si\tau_2$, the proof will be complete if we can show that $\tau_1\mathrel\xi\tau_2$.  To do this, we first define
\[
\si_1 = 
\Big( 
{ \scriptsize \renewcommand*{\arraystretch}{1}
\begin{array} {\c|\c|\c|\c|\cend}
n \:&\: H_1 \:&\: \cdots \:&\: H_{m-1} \:&\: H_m  \\ \cline{2-5}
n-2 \:&\: H_1 \:&\: \cdots \:&\: H_{m-1} \:&\: n-1,n
\rule[0mm]{0mm}{2.7mm}
\end{array} 
}
\hspace{-1.5 truemm} \Big)
\AND
\si_2 =
\Big( 
{ \scriptsize \renewcommand*{\arraystretch}{1}
\begin{array} {\c|\c|\c|\c|\c|\cend}
n \:&\: H_1 \:&\: \cdots \:&\: H_{m-2} \:&\: H_{m-1} \:&\: H_m  \\ \cline{2-6}
n \:&\: H_1 \:&\: \cdots \:&\: H_{m-2} \:&\: n-4,n-1 \:&\: n-3,n-2
\rule[0mm]{0mm}{2.7mm}
\end{array} 
}
\hspace{-1.5 truemm} \Big).
\]
Note that the proof of the sub-claim 
above gives $\si_1\mathrel\xi\tau_2$.  Now define
\[
\ga_1 = 
\Big( 
{ \scriptsize \renewcommand*{\arraystretch}{1}
\begin{array} {\c|\c|\c|\c|\c|\c|\cend}
n-2 \:&\: n-1 \:&\: n \:&\: H_1 \:&\: \cdots \:&\: H_{m-2} \:&\: H_{m-1}  \\ \cline{4-7}
n-4 \:&\: n-3 \:&\: n-2 \:&\: H_1 \:&\: \cdots \:&\: H_{m-2} \:&\: n-1,n
\rule[0mm]{0mm}{2.7mm}
\end{array} 
}
\hspace{-1.5 truemm} \Big)
\AND
\ga_2 =
\Big( 
{ \scriptsize \renewcommand*{\arraystretch}{1}
\begin{array} {\c|\c|\c|\c|\c|\c|\cend}
n-2 \:&\: n-1 \:&\: n \:&\: H_1 \:&\: \cdots \:&\: H_{m-2} \:&\: H_{m-1}  \\ \cline{4-7}
n-4 \:&\: n-1 \:&\: n \:&\: H_1 \:&\: \cdots \:&\: H_{m-2} \:&\: n-3,n-2
\rule[0mm]{0mm}{2.7mm}
\end{array} 
}
\hspace{-1.5 truemm} \Big).
\]
One may then check that $\tau_2 \mathrel\xi \si_1=\tau_2\ga_1 \mathrel\xi \si_1\ga_1 = \si_1\ga_2 \mathrel\xi \tau_2\ga_2 =\si_2$.  In particular, $\cg{\si_2}{\tau_2}\sub{\mathrel\xi}=\cg\al\be$.  Now let $\si_3,\tau_3$ be the elements of $\B_{n-1}$ obtained by removing the block $\{n,n'\}$ from $\si_2,\tau_2\in\B_n$, respectively.  Write ${\mathrel\zeta}=\cg{\si_3}{\tau_3}_{\B_{n-1}}$ for the congruence on $\B_{n-1}$ generated by $(\si_3,\tau_3)$.  Since $\si_3\not=\tau_3$, yet $\si_3$ and $\tau_3$ are related under the $\rho_0$-congruence on $\B_{n-1}$, and since $n-1\geq4$ is even, it follows from Case 1 that ${\mathrel\zeta} = \rho_0$ (on $\B_{n-1}$).  Consequently, we have
$
\partII{H_1}\cdots{H_m}{B_1}\cdots{B_m} \mathrel\zeta \partII{H_1}\cdots{H_m}{H_1}\cdots{H_m}
$,
and it follows that $\tau_1\mathrel\xi\tau_2$, as required.~\epf

Now that we have described the generating pairs for $\rho_z$ and $\lam_z$, we may easily do the same for their join~$R_z$.

\begin{prop}\label{prop:R01_Bn}
The congruence $R_z$ is principal.  Moreover, if $\al,\be\in\B_n$, then
\[
R_z=\cg\al\be\iff(\al,\be)\in R_z\sm(\rho_z\cup\lam_z).
\]
\end{prop}

\pf
For simplicity, we omit the subscript $z$ throughout.
The ``only if'' part is immediate from $R=\rho\vee\lambda$.
For the ``if'' part, let  $(\al,\be)\in R\sm(\rho\cup\lam)$, noting that clearly $\cg\al\be\sub R$.  
Then $\ker(\al)=\ker(\al\be)$ and $\coker(\al)\not=\coker(\be)=\coker(\al\be)$, since $\rank(\al)=\rank(\be)=z$ and $(\al,\be)\not\in \lam$.  So $(\al,\al\be)\in\rho\sm\De$.  It follows from Proposition~\ref{prop:rl_Bn}(i) that $\rho=\cg{\al}{\al\be}=\cg{\al\al}{\al\be}\sub\cg\al\be$.  A symmetrical argument shows that $\lam\sub\cg\al\be$.  It follows that $R=\rho\vee\lam\sub\cg\al\be$. \epf

The remaining considerations of the lower congruences apply only to the case when $n$ is even.  Next we consider the three non-trivial $\mu$ congruences coming from the pairs $\C_{\S_2}$, $\C_2$ and $\C_K$.

\begin{prop}\label{prop:nek}
Let $n=2m\geq4$ be even.  The relations $\mu_{\S_2},\mu_2,\mu_K$ are all principal congruences on $\B_n$.  Moreover, if $\al,\be\in\B_n$, then
\begin{itemize}\begin{multicols}{2}
\itemit{i} $\mu_{\S_2}=\cg\al\be\iff (\al,\be)\in\mu_{\S_2}\sm\De$,
\itemit{ii} $\mu_2=\cg\al\be\iff (\al,\be)\in\mu_2\sm\mu_{\S_2}$, 
\itemit{iii} $\mu_K=\cg\al\be\iff (\al,\be)\in\mu_K\sm\mu_2$.
\end{multicols}\eit
\end{prop}

\pf 
The only non-obvious part is the direct inclusion ($\subseteq$) in the ``if'' statement in each part.
We begin with (i).
Let $(\al,\be)\in\mu_{\S_2}\sm\De$ be arbitrary.
Since $\rank(\al)=\rank(\be)=2$ and $\al\H\be$, yet $\al\not=\be$, we may write
$\al = \partV {a_1}{a_2}{A_1}{A_{m-1}}{b_1}{b_2}{B_1}{B_{m-1}}$ and $\be = \partV {a_1}{a_2}{A_1}{A_{m-1}}{b_2}{b_1}{B_1}{B_{m-1}}$.
Now let $(\ga,\de)\in\mu_{\S_2}$.  We must show that $(\ga,\de)\in\cg\al\be$.
If $\ga=\de$ there is nothing to prove, so suppose otherwise, and write
$\ga = \partV {c_1}{c_2}{C_1}{C_{m-1}}{d_1}{d_2}{D_1}{D_{m-1}}$ and $\de = \partV {c_1}{c_2}{C_1}{C_{m-1}}{d_2}{d_1}{D_1}{D_{m-1}}$.
Put
$\si = \partV{c_1}{c_2}{C_1}{C_{m-1}}{a_1}{a_2}{A_1}{A_{m-1}}$ and $\tau = \partV{b_1}{b_2}{B_1}{B_{m-1}}{d_1}{d_2}{D_1}{D_{m-1}}$.
Then $(\ga,\de)=(\si\al\tau,\si\be\tau)\in\cg\al\be$.

For (ii), let $(\al,\be)\in\mu_{2}\sm\mu_{\S_2}$  be arbitrary.
To simplify the notation, write $\xi=\cg\al\be$.  There are two possibilities: consulting the definitions of $\mu_2$ and $\mu_{\S_2}$, and renaming $\al,\be$ if necessary, either
\begin{itemize}\begin{multicols}{2}
\item[(a)] $\rank(\al)=2$, $\rank(\be)=0$ and $\alh=\be$, or
\item[(b)] $\rank(\al)=\rank(\be)=2$, $\alh=\beh$ and $(\al,\be)\not\in{\H}$.
\end{multicols}
\eit
Suppose first that we are in case (a), and write $\al=\partV ab{A_1}{A_{m-1}}xy{B_1}{B_{m-1}}$ and $\be=\partIII{a,b}{A_1}\cdots{A_{m-1}}{x,y}{B_1}\cdots{B_{m-1}}$.  
We first claim that if $\tau\in\B_n$ is such that $\rank(\tau)\leq2$, then
$(\tau,\tauh)\in\xi$.  Indeed, this is trivial if $\rank(\tau)=0$, since then
$\tauh=\tau$.  So suppose $\rank(\tau)=2$, and write $\tau=\partV cd{C_1}{C_{m-1}}uv{D_1}{D_{m-1}}$, so that $\tauh=\partIII{c,d}{C_1}\cdots{C_{m-1}}{u,v}{D_1}\cdots{D_{m-1}}$.  We then define $\si_1=\partV cd{C_1}{C_{m-1}}ab{A_1}{A_{m-1}}$ and $\si_2=\partV xy{B_1}{B_{m-1}}uv{D_1}{D_{m-1}}$, and note that $\tau = \si_1\al\si_2 \mathrel\xi \si_1\be\si_2 =\tauh$, completing the proof of the claim.  
Now, for arbitrary $(\gamma,\delta)\in\mu_2$, if $\gamma=\delta$ then clearly $(\gamma,\delta)\in\xi$, while if $\gamma\neq\delta$, then $\rank(\ga),\rank(\de)\leq2$ and $\gah=\deh$, so that $\ga \mathrel\xi \gah = \deh \mathrel\xi \de$, as required.

Now suppose we are in case (b).  Write $\alh=\beh=\partII{A_1}\cdots{A_m}{B_1}\cdots{B_m}$, and let $A_1=\{a,b\}$, $A_m=\{c,d\}$, $B_1=\{x,y\}$ and $B_m=\{u,v\}$.  Without loss of generality, we may assume that $\al=\partV ab{A_2}{A_m}xy{B_2}{B_m}$.  Since $\rank(\be)=2$, $\alh=\beh$, yet $(\al,\be)\not\in{\H}$, there are three possibilities: renaming $c,d,u,v$ if necessary, either
\begin{itemize}\begin{multicols}{3}
\item[(c)] $\be=\partV ab{A_2}{A_m}uv{B_1}{B_{m-1}}$, 
\item[(d)] $\be=\partV cd{A_1}{A_{m-1}}xy{B_2}{B_m}$, \ or
\item[(e)] $\be=\partV cd{A_1}{A_{m-1}}uv{B_1}{B_{m-1}}$.
\end{multicols}\eit
In all cases, put $\si_3=\partV ab{A_2}{A_m}ab{A_2}{A_m}$ and $\si_4=\partV xy{B_2}{B_m}xy{B_2}{B_m}$.  Then 
\[
\al=\si_3\al=\al\si_4
\AND
\alh=\begin{cases}
\si_3\be &\text{in cases (d) and (e)}\\
\be\si_4 &\text{in cases (c) and (e).}
\end{cases}
\]
In all cases, it follows that $(\al,\alh)\in \xi$.
Case~(a) then gives $\mu_2\sub\cg{\al}{\alh}\sub\xi$.

Finally, for (iii), let $(\alpha,\beta)\in \mu_K\setminus\mu_2$ and write $\xi=\cg\al\be$.  
Since $(\al,\be)\not\in\mu_2$, 
it follows that $\rank(\al)=\rank(\be)=4$, $\al\H\be$ and $\phi(\al,\be)\in K\sm\{\id_4\}$.  
We may therefore write
$
\al=
\Big( 
{ \scriptsize \renewcommand*{\arraystretch}{1}
\begin{array} {\c|\c|\c|\c|\c|\c|\cend}
a_1 \:&\: a_2 \:&\: a_3 \:&\: a_4 \:&\: A_1 \:&\: \cdots \:&\: A_{m-2}  \\ \cline{5-7}
b_1 \:&\: b_2 \:&\: b_3 \:&\: b_4 \:&\: B_1 \:&\: \cdots \:&\: B_{m-2}
\rule[0mm]{0mm}{2.7mm}
\end{array} 
}
\hspace{-1.5 truemm} \Big)
$ and $
\be=
\Big( 
{ \scriptsize \renewcommand*{\arraystretch}{1}
\begin{array} {\c|\c|\c|\c|\c|\c|\cend}
a_1 \:&\: a_2 \:&\: a_3 \:&\: a_4 \:&\: A_1 \:&\: \cdots \:&\: A_{m-2}  \\ \cline{5-7}
b_2 \:&\: b_1 \:&\: b_4 \:&\: b_3 \:&\: B_1 \:&\: \cdots \:&\: B_{m-2}
\rule[0mm]{0mm}{2.7mm}
\end{array} 
}
\hspace{-1.5 truemm} \Big)
$.
Define $\si_1=\Big( 
{ \scriptsize \renewcommand*{\arraystretch}{1}
\begin{array} {\c|\c|\c|\c|\c|\cend}
a_1 \:&\: a_4 \:&\: a_2,a_3 \:&\: A_1 \:&\: \cdots \:&\: A_{m-2}  \\ \cline{3-6}
a_1 \:&\: a_4 \:&\: a_2,a_3 \:&\: A_1 \:&\: \cdots \:&\: A_{m-2}
\rule[0mm]{0mm}{2.7mm}
\end{array} 
}
\hspace{-1.5 truemm} \Big)$.  Then 
$(\si_1\al,\si_1\be)\in\mu_2\sm\mu_{\S_2}$, so part (ii) gives $\mu_2=\cg{\si_1\al}{\si_1\be}\sub\xi$.  
It remains to show that $\mu_K\sm\mu_2\sub\xi$, so let $(\ga,\de)\in\mu_K\sm\mu_2$ be arbitrary.  As above, we may write
$
\ga = \Big( 
{ \scriptsize \renewcommand*{\arraystretch}{1}
\begin{array} {\c|\c|\c|\c|\c|\c|\cend}
c_1 \:&\: c_2 \:&\: c_3 \:&\: c_4 \:&\: C_1 \:&\: \cdots \:&\: C_{m-2}  \\ \cline{5-7}
d_1 \:&\: d_2 \:&\: d_3 \:&\: d_4 \:&\: D_1 \:&\: \cdots \:&\: D_{m-2}
\rule[0mm]{0mm}{2.7mm}
\end{array} 
}
\hspace{-1.5 truemm} \Big)
$ and $
\de = \Big( 
{ \scriptsize \renewcommand*{\arraystretch}{1}
\begin{array} {\c|\c|\c|\c|\c|\c|\cend}
c_1 \:&\: c_2 \:&\: c_3 \:&\: c_4 \:&\: C_1 \:&\: \cdots \:&\: C_{m-2}  \\ \cline{5-7}
d_2 \:&\: d_1 \:&\: d_4 \:&\: d_3 \:&\: D_1 \:&\: \cdots \:&\: D_{m-2}
\rule[0mm]{0mm}{2.7mm}
\end{array} 
}
\hspace{-1.5 truemm} \Big)
$.
Define $\si_2=\Big( 
{ \scriptsize \renewcommand*{\arraystretch}{1}
\begin{array} {\c|\c|\c|\c|\c|\c|\cend}
c_1 \:&\: c_2 \:&\: c_3 \:&\: c_4 \:&\: C_1 \:&\: \cdots \:&\: C_{m-2}  \\ \cline{5-7}
a_1 \:&\: a_2 \:&\: a_3 \:&\: a_4 \:&\: A_1 \:&\: \cdots \:&\: A_{m-2}
\rule[0mm]{0mm}{2.7mm}
\end{array} 
}
\hspace{-1.5 truemm} \Big)$ and $\si_3=\Big( 
{ \scriptsize \renewcommand*{\arraystretch}{1}
\begin{array} {\c|\c|\c|\c|\c|\c|\cend}
b_1 \:&\: b_2 \:&\: b_3 \:&\: b_4 \:&\: B_1 \:&\: \cdots \:&\: B_{m-2}  \\ \cline{5-7}
d_1 \:&\: d_2 \:&\: d_3 \:&\: d_4 \:&\: D_1 \:&\: \cdots \:&\: D_{m-2}
\rule[0mm]{0mm}{2.7mm}
\end{array} 
}
\hspace{-1.5 truemm} \Big)$.  Then $(\ga,\de)=(\si_2\al\si_3,\si_2\be\si_3)\in\xi$, as required.~\epf

The remaining principal lower congruences are the $\lam,\rho,R$ congruences arising from the pair $\C_2$.

\begin{prop}\label{prop:joins2_Bn}
Let $n=2m\geq4$ be even.  The congruences 
$\rho_2$, $\lambda_2$ and $R_2$ are all principal.  Moreover, if $\al,\be\in\B_n$, then
\begin{itemize}\begin{multicols}{2}
\itemit{i} $\rho_2=\cg\al\be\iff (\al,\be)\in\rho_2\sm(\mu_2\cup\rho_{\S_2})$,
\itemit{ii} $\lam_2=\cg\al\be\iff (\al,\be)\in\lambda_2\sm(\mu_2\cup\lam_{\S_2})$,
\itemit{iii} $R_2=\cg\al\be\iff (\al,\be)\in R_2 \sm
(\lambda_2\cup\rho_2\cup R_{\S_2})$.
\end{multicols}
\end{itemize}
\end{prop}

\pf Again, it suffices prove to the direct inclusion ($\subseteq$) in the ``if'' statement in each case.
For (i), fix some $(\al,\be)\in \rho_2\sm(\mu_2\cup\rho_{\S_2})$, and write $\xi=\cg\al\be$.  
Recalling the definitions of $\rho_2$, $\mu_2$ and $\rho_{\S_2}$, we may assume (renaming $\al,\be$ if necessary) that $\rank(\al)=2$, $\rank(\be)\leq2$, $\ker(\alh)=\ker(\beh)$ and $\coker(\alh)\not=\coker(\beh)$.
Since $\rho_2=\mu_2\vee\rho_0$ by Proposition \ref{prop-CR3}(v), it is sufficient to show that $\xi$ contains both $\mu_2$ and $\rho_0$.  In fact, it is enough to show that $\mu_2\sub\xi$;
indeed, this implies  $\alh \mathrel\xi \al \mathrel\xi \be \mathrel\xi \beh$,
where $(\alh,\beh)\in\rho_0\sm\De$, and hence Proposition~\ref{prop:rl_Bn}(i) gives $\rho_0=\cg{\alh}{\beh}\sub\xi$.  Now, one of the following must hold:
\begin{itemize}\begin{multicols}{2}
\item[(a)] $\rank(\be)=0$,
\item[(b)] $\rank(\be)=2$ and $\codom(\al)=\codom(\be)$, or
\item[(c)] $\rank(\be)=2$ and $\codom(\al)\not=\codom(\be)$.
\end{multicols}
\eit
In each case, we define a Brauer element $\ga\in\B_n$, and let the reader check that $(\al\ga,\be\ga)\in\mu_2\sm\mu_{\S_2}$.  It will then follow from Proposition \ref{prop:nek}(ii) that $\mu_2\sub\xi$, as desired.  
Write $\al=\partV ab{A_1}{A_{m-1}}xy{B_1}{B_{m-1}}$.
In case~(a), we put $\ga=\partV xy{B_1}{B_{m-1}}xy{B_1}{B_{m-1}}$.  
Next, suppose (b) holds.  Choose some $(c,d)\in\coker(\be)\sm\coker(\al)$.  Without loss of generality, we may assume that $B_{m-2}=\{c,e\}$ and $B_{m-1}=\{d,f\}$ for some $e,f$, and we define $\ga = \Big( 
{ \scriptsize \renewcommand*{\arraystretch}{1}
\begin{array} {\c|\c|\c|\c|\c|\c|\cend}
e \:&\: f \:&\: c,x \:&\: d,y \:&\: B_1 \:&\: \cdots \:&\: B_{m-3}  \\ \cline{3-7}
e \:&\: f \:&\: c,x \:&\: d,y \:&\: B_1 \:&\: \cdots \:&\: B_{m-3}
\rule[0mm]{0mm}{2.7mm}
\end{array} 
}
\hspace{-1.5 truemm} \Big)$.  
Finally, suppose~(c) holds.  Without loss of generality, we may assume that $y\not\in\codom(\be)$, so that $\{y',w'\}$ is a block of $\be$ for some $w$.  Write $\codom(\be)=\{u,v\}$, fix a 2-partition $\{C_1,\ldots,C_{m-2}\}$ of $\bn\sm\{u,v,y,w\}$, and define $\ga=\Big( 
{ \scriptsize \renewcommand*{\arraystretch}{1}
\begin{array} {\c|\c|\c|\c|\c|\cend}
y \:&\: w \:&\: u,v \:&\: C_1 \:&\: \cdots \:&\: C_{m-2}  \\ \cline{3-6}
y \:&\: w \:&\: u,v \:&\: C_1 \:&\: \cdots \:&\: C_{m-2}
\rule[0mm]{0mm}{2.7mm}
\end{array} 
}
\hspace{-1.5 truemm} \Big)$.

Next note that (ii) follows from (i) by duality.  Finally, for (iii), fix some $(\al,\be)\in R_2\sm (\lambda_2\cup\rho_2\cup R_{\S_2})$, and write $\xi=\cg\al\be$.  Renaming $\al,\be$ if necessary, we may assume that $\rank(\al)=2$, $\rank(\be)\leq2$, $\ker(\alh)\not=\ker(\beh)$ and $\coker(\alh)\not=\coker(\beh)$.  
As above, write $\al=\partV ab{A_1}{A_{m-1}}xy{B_1}{B_{m-1}}$.  Similarly to part (i), since $R_2=\mu_2\vee R_0$ by Proposition \ref{prop-CR3}(vi), it suffices to show that $\mu_2\sub\xi$.  
One of the following must be the case:
\begin{itemize}\begin{multicols}{2}
\item[(d)] $\rank(\be)=0$,
\item[(e)] $\rank(\be)=2$ and $\dom(\al)\cap\dom(\be)\not=\emptyset$, or
\item[(f)] $\rank(\be)=2$ and $\dom(\al)\cap\dom(\be)=\emptyset$.
\end{multicols}
\eit
In each case, we define a Brauer element $\ga\in\B_n$ such that $(\ga\al,\ga\be)\in\rho_2\sm(\mu_2\cup\rho_{\S_2})$.  It will then follow from part (i)
that $\mu_2\subseteq \rho_2\sub\xi$.  
In cases (d) and (e), we define $\ga=\partV ab{A_1}{A_{m-1}}ab{A_1}{A_{m-1}}$.  In case (f), we write $\dom(\be)=\{c,d\}$, fix a 2-partition $\{C_1,\ldots,C_{m-2}\}$ of $\bn\sm\{a,b,c,d\}$, and define $\ga=\Big( 
{ \scriptsize \renewcommand*{\arraystretch}{1}
\begin{array} {\c|\c|\c|\c|\c|\cend}
a \:&\: b \:&\: A_1 \:&\: A_2 \:&\: \cdots \:&\: A_{m-1}  \\ \cline{3-6}
a \:&\: b \:&\: c,d \:&\: C_1 \:&\: \cdots \:&\: C_{m-2}
\rule[0mm]{0mm}{2.7mm}
\end{array} 
}
\hspace{-1.5 truemm} \Big)$. \epf

\begin{rem}
The congruences $\rho_{\S_2}$, $\lam_{\S_2}$, $R_{\S_2}$, $\rho_K$, $\lambda_K$ and $R_K$  are not principal (cf.~Remark \ref{rem:nonprincipal}), but each is generated by two pairs of Brauer elements.  The pairs of generating pairs for each may be deduced in the manner of Remark \ref{NotPrinc}. 
\end{rem}

We now turn to the upper congruences, specifically:
\begin{equation}\label{eq:Bn_chain}
R_1 \subsetneq R_{\A_3} \subsetneq R_{\S_3} \subsetneq R_3 \subsetneq  \cdots \subsetneq R_n =\nabla  \text{ ($n$ odd),}\ \ \ \ 
R_K \subsetneq R_{\A_4} \subsetneq R_{\S_4} \subsetneq R_4 \subsetneq 
 \cdots \subsetneq R_n =\nabla \text{ ($n$ even).}
\end{equation}
We begin with a series of results giving sufficient conditions for a congruence on~$\B_n$ to contain a Rees congruence $R_q$.

\begin{lemma}\label{lem:1z_Bn}
Suppose $\xi\in\Cong(\B_n)$ and that $(\1,\al)\in\xi$ for some $\al\in J_z$.  Then $\xi=R_n$.
\end{lemma}

\pf 
We first show that $R_z\sub\xi$.  
Choose any $\be\in J_z$ with $\ker(\be)\not=\ker(\al)$ and $\coker(\be)\not=\coker(\al)$.  
Then $\be=\1\;\!\be\;\!\1\mathrel\xi\al\be\al=\al$,
and clearly $(\al,\be)\in R_z\sm(\rho_z\cup\lam_z)$, so Proposition \ref{prop:R01_Bn} gives $R_z=\cg\al\be\sub\xi$.  
Now 
the proof will be complete if we can show that every Brauer element is $\xi$-related to a Brauer element of rank $z$.  So let $\ga\in\B_n$.  Then $\ga=\ga\;\!\1\mathrel\xi\ga\al$, and since $\rank(\ga\al)\leq\rank(\al)=z$, we are done. \epf

Recall that group of units of $\B_n$ is the symmetric group $\S_n=\set{\al\in\B_n}{\rank(\al)=n}$, and recall  that, for any $\al\in\S_n$, $\al^*$ is the (group theoretic) inverse of $\al$ in $\S_n$.

\begin{lemma}\label{lem:1q_Bn}
Suppose $\xi\in\Cong(\B_n)$ and that $(\1,\al)\in\xi$ for some $\al\in\B_n$ with $\rank(\al)<n$.  Then $\xi=R_n$.
\end{lemma}

\pf We prove the lemma by induction on $\rank(\al)$.  If $\rank(\al)=z$, then
we are done, by Lemma~\ref{lem:1z_Bn}, so suppose $\rank(\al)\geq z+2$.  Inductively, it is enough to show that $(\1,\be)\in\xi$ for some $\be\in\B_n$ with ${\rank(\be)<\rank(\al)}$.  Since $2\leq\rank(\al)\leq n-2$, we may choose some $i,j\in\dom(\al)$ with $i\not=j$, and some $(k,l)\in\coker(\al)$ with $k\not=l$.  Choose any $\si\in\S_n$ with $k\si=i$ and $l\si=j$.  Then $\1 = \1\;\!\si\;\!\1\;\!\si^* \mathrel\xi \al\si\al\si^*$, and $\rank(\al\si\al\si^*)\leq\rank(\al)-2$, as required. \epf

For $q\in\{z,z+2,\dots,n\}$, let $\ve_q=\overline\id_q=\partI1q{q+1,q+2}{n-1,n}1q{q+1,q+2}{n-1,n}$
denote the identity of the maximal subgroup~$\overline{\S}_q$.  
Note that the subsemigroup $\ve_q\B_n\ve_q$ of $\B_n$ is isomorphic to $\B_q$; it consists of all Brauer elements that have the blocks $\{q+1,q+2\},\ldots,\{n-1,n\}$ and $\{(q+1)',(q+2)'\},\ldots,\{(n-1)',n'\}$.

\begin{lemma}\label{lem:qp_Bn}
Suppose $\xi\in\Cong(\B_n)$ and that there exists $(\al,\be)\in\xi$ with $q=\rank(\al)\geq3$ and $\rank(\be)<q$.  Then $R_q\sub\xi$.
\end{lemma}

\pf 
We first claim that $R_z\sub\xi$.  Write $\al=\partI{a_1}{a_q}{A_1}{A_r}{b_1}{b_q}{B_1}{B_r}$, and put
$$\ga=\partI1q{q+1,q+2}{n-1,n}{a_1}{a_q}{A_1}{A_r} \AND \de=\partI{b_1}{b_q}{B_1}{B_r}1q{q+1,q+2}{n-1,n}.$$
Then $\ve_q=\ga\al\de\mathrel\xi\ga\be\de$.  But $\ga\be\de\in\ve_q\B_n\ve_q$ and ${\rank(\ga\be\de)\leq\rank(\be)<q}$.  By Lemma~\ref{lem:1q_Bn}, and the fact that $\ve_q\B_n\ve_q$ is isomorphic to $\B_q$, it follows that all elements of $\ve_q\B_n\ve_q$ are $\cg{\ve_q}{\ga\be\de}$-related and, hence, $\xi$-related.  In particular, if we let $\tau_1,\tau_2\in\ve_q\B_n\ve_q$ be such that $\rank(\tau_1)=\rank(\tau_2)=z$, $\ker(\tau_1)\not=\ker(\tau_2)$ and $\coker(\tau_1)\not=\coker(\tau_2)$, then $(\tau_1,\tau_2)\in\xi$.  But $(\tau_1,\tau_2)\in R_z\sm(\rho_z\cup\lam_z)$, so it follows from Proposition \ref{prop:R01_Bn} that $R_z=\cg{\tau_1}{\tau_2}\sub\xi$.

Now that we know all Brauer elements of rank $z$ are $\xi$-related, the proof of the lemma will be complete if we can show any Brauer element of rank at most $q$ is $\xi$-related to a Brauer element of rank~$z$.  With this in mind, let $\si\in\B_n$ with $p=\rank(\si)\leq q$, and write $\si=\partI{c_1}{c_p}{C_1}{C_s}{d_1}{d_p}{D_1}{D_s}$.  Define ${\tau_3=\partI{c_1}{c_p}{C_1}{C_s}{1}{p}{p+1,p+2}{n-1,n}}$ and $\tau_4=\partI{1}{p}{p+1,p+2}{n-1,n}{d_1}{d_p}{D_1}{D_s}$.  Since $\ve_p,\ve_z\in\ve_q\B_n\ve_q$, we have $\ve_p \mathrel\xi \ve_z$, by the previous paragraph.  It follows that $\si = \tau_3\ve_p\tau_4 \mathrel\xi \tau_3\ve_z\tau_4$, with $\rank(\tau_3\ve_z\tau_4)=z$.  As noted above, this completes the proof. \epf

The previous result shows that a congruence identifying a Brauer element of rank $q\geq3$ with a Brauer element of lower rank must contain the Rees congruence $R_q$.
Our next goal is to give a condition under which a congruence that identifies two Brauer elements of equal rank $q\geq3$ must contain $R_q$; see Lemma~\ref{lem:notH_Bn}.
One of the key steps in the proof of Lemma \ref{lem:notH_Bn} is true also for $q=2$, and it will be convenient for later use to prove this intermediate result under this slightly more general hypothesis:

\begin{lemma}\label{lem:notH_prelim_Bn}
  Suppose $\xi\in\Cong(\B_n)$ and that $(\al,\be)\in (\xi\setminus{\H})\cap
  (J_q\times J_q)$ with $q\geq2$.  Then there exists $(\ga,\de)\in\xi$ with
  $\rank(\ga)=q$ and $\rank(\de)<q$.
\end{lemma}

\pf Since $(\al,\be)\not\in{\H}$, it follows that $(\al,\be)\not\in{\R}$ or $(\al,\be)\not\in{\L}$.  By symmetry, we may assume that $(\al,\be)\not\in{\R}$, so $\ker(\al)\not=\ker(\be)$.  Write $\al=\partI{a_1}{a_q}{A_1}{A_r}{b_1}{b_q}{B_1}{B_r}$ and $\be=\partI{c_1}{c_q}{C_1}{C_r}{d_1}{d_q}{D_1}{D_r}$.

\bigskip\noindent {\bf Case 1.}  Suppose first that $\dom(\al)=\dom(\be)$.  Renaming if necessary, we may assume that $A_1$ is not a block of $\be$.  Write $A_1=\{i,j\}$.  Since $i,j\in\bn\sm\dom(\al)=\bn\sm\dom(\be)$, we may assume that $C_1=\{i,k\}$ and $C_2=\{j,l\}$ for some $k,l$.  Define $\si=\Big( 
{ \scriptsize \renewcommand*{\arraystretch}{1}
\begin{array} {\c|\c|\c|\c|\c|\c|\c|\c|\c|\cend}
i \:&\: j \:&\: a_3 \:&\: \cdots \:&\: a_q \:&\: a_1,a_2 \:&\: A_2 \:&\: A_3 \:&\: \cdots \:&\: A_r  \\ \cline{6-10}
i \:&\: j \:&\: a_3 \:&\: \cdots \:&\: a_q \:&\: a_1,k \:&\: a_2,l \:&\: C_3 \:&\: \cdots \:&\: C_r
\rule[0mm]{0mm}{2.7mm}
\end{array} 
}
\hspace{-1.5 truemm} \Big)$.  Then $\rank(\si\al)=q-2$ and $\rank(\si\be)=q$, so we put $\ga=\si\be$ and $\de=\si\al$.  

\bigskip\noindent {\bf Case 2.}  Now suppose $\dom(\al)\not=\dom(\be)$.  Without loss of generality, we may assume that $a_q\not\in\dom(\be)$.  Denote the $\ker(\be)$-class of $a_q$ by $\{a_q,x\}$, and let $\si\in\B_n$ be an arbitrary Brauer element of rank $q$ such that $\codom(\si)$ contains $a_q$, $x$, and $q-2$ additional elements of $\{a_1,\ldots,a_{q-1}\}$.  Then $\rank(\si\al)=q$ and $\rank(\si\be)\leq q-2$, so we put $\ga=\si\al$ and $\de=\si\be$.    \epf

\begin{lemma}\label{lem:notH_Bn}
Suppose $\xi\in\Cong(\B_n)$ and that $(\al,\be)\in (\xi\setminus{\H})\cap (J_q\times J_q)$ with $q\geq3$.  Then $R_q\sub\xi$.
\end{lemma}

\pf By Lemma \ref{lem:notH_prelim_Bn}, there exists $(\ga,\de)\in\xi$ with $\rank(\ga)=q$ and $\rank(\de)<q$.  It then follows from Lemma \ref{lem:qp_Bn} that $R_q\sub\xi$. \epf

For $1\leq q\leq n$, and for a permutation $\pi\in\S_q$, we will write $\normal{\pi}$ for the normal closure of $\pi$ in $\S_q$.  The next result concerns the permutations $\phi(\al,\be)$ from Definition \ref{defn:phi}.

\begin{lemma}\label{lem:thN_Bn}
Let $3\leq q\leq n$ with $q\equiv n\pmod 2$, and let $N$ be one of $\A_q$ or
  $\S_q$.  Suppose $\xi\in\Cong(\B_n)$, and that $(\al,\be)\in \xi\cap{\H}\cap (J_q\times J_q)$ is such that $N=\normal{\phi(\al,\be)}$.  Then $R_{q-2}\sub\xi$.
\end{lemma}

\pf Write $\al=\partI{a_1}{a_q}{A_1}{A_r}{b_1}{b_q}{B_1}{B_r}$ and $\be=\partI{a_1}{a_q}{A_1}{A_r}{b_{1\si}}{b_{q\si}}{B_1}{B_r}$, where $a_1<\cdots<a_q$, noting that $\phi(\al,\be)=\si^{-1}$, so $N=\normal\si$.  Put $\ga_1=\partI1q{q+1,q+2}{n-1,n}{a_1}{a_q}{A_1}{A_r}$ and $\ga_2=\partI{b_1}{b_q}{B_1}{B_r}1q{q+1,q+2}{n-1,n}$, noting that $\ve_q=\ga_1\al\ga_2\mathrel\xi\ga_1\be\ga_2=\sib$.  
Define
\[
\de_1=\Big( 
{ \scriptsize \renewcommand*{\arraystretch}{1}
\begin{array} {\c|\c|\c|\c|\c|\c|\c|\c|\cend}
3 \:&\: 4 \:&\: 5 \:&\: \cdots \:&\: q \:&\: 1,2 \:&\: q+1,q+2 \:&\: \cdots \:&\: n-1,n  \\ \cline{6-9}
3 \:&\: 4 \:&\: 5 \:&\: \cdots \:&\: q \:&\: 1,2 \:&\: q+1,q+2 \:&\: \cdots \:&\: n-1,n
\rule[0mm]{0mm}{2.7mm}
\end{array} 
}
\hspace{-1.5 truemm} \Big)
\AND
\de_2=\Big( 
{ \scriptsize \renewcommand*{\arraystretch}{1}
\begin{array} {\c|\c|\c|\c|\c|\c|\c|\c|\cend}
1 \:&\: 4 \:&\: 5 \:&\: \cdots \:&\: q \:&\: 2,3 \:&\: q+1,q+2 \:&\: \cdots \:&\: n-1,n  \\ \cline{6-9}
1 \:&\: 4 \:&\: 5 \:&\: \cdots \:&\: q \:&\: 2,3 \:&\: q+1,q+2 \:&\: \cdots \:&\: n-1,n
\rule[0mm]{0mm}{2.7mm}
\end{array} 
}
\hspace{-1.5 truemm} \Big).
\]
We claim that $(\de_1,\de_2)\in\xi$.  In order to do this, let $\pi\in\A_q$ be the $3$-cycle mapping $1\mt2\mt3\mt1$ and fixing each of $4,\ldots,q$.  Since $\A_q\sub N=\normal\si$, we have $\pi=(\tau_1^{-1}\si\tau_1)\cdots(\tau_k^{-1}\si\tau_k)$ for some $k\geq1$ and some $\tau_1,\ldots,\tau_k\in\S_q$.  It follows that
$\pib = (\taub_1^*\ \sib\ \taub_1)\cdots(\taub_k^*\ \sib\ \taub_k) \mathrel\xi (\taub_1^*\ \ve_q\ \taub_1)\cdots(\taub_k^*\ \ve_q\ \taub_k) = \ve_q$, and we also have $\pib^*=\pib\ \pib \mathrel\xi \ve_q\ve_q=\ve_q$.  But then $\de_1 = \ve_q\de_1\ve_q \mathrel\xi \pib^*\de_1\pib = \de_2$, completing the proof of the claim.  Note that $\rank(\de_1)=\rank(\de_2)=q-2\geq1$, and that $(\de_1,\de_2)\not\in{\R}$ and $(\de_1,\de_2)\not\in{\L}$.  It follows from Proposition \ref{prop:R01_Bn} (if~$q=3$) or Proposition \ref{prop:joins2_Bn}(iii) (if $q=4$) or Lemma \ref{lem:notH_Bn} (if $q\geq5$) that $R_{q-2}\sub\cg{\de_1}{\de_2}\sub\xi$.  \epf

The next result completes the proof of Theorem \ref{CongBn}. 

\begin{prop}\label{prop:chain_Bn}
Consider the chains \eqref{eq:Bn_chain} of congruences on $\B_n$.
Apart from the first congruence on the second list, all of the listed congruences are principal.  Moreover, if $\xi$ is any of the listed congruences other than the first of either list, and if $\al,\be\in\B_n$, then $\xi=\cg\al\be\iff(\al,\be)\in\xi\sm\xi^-$, where $\xi^-$ denotes the previous congruence in the same list.
\end{prop}

\pf Let $\xi$ be any of the listed congruences other than the first of either
list.  If $(\al,\be)\not\in\xi$ or if $(\al,\be)\in\xi^-$, then clearly
$\cg\al\be\not=\xi$.  Conversely, suppose $(\al,\be)\in\xi\sm\xi^-$, and write $\zeta=\cg\al\be$.  The proof will be complete if we can show that $\xi\sub\zeta$.  

\bigskip\noindent {\bf Case 1.}  Suppose first that $\xi=R_q$ for some $q\geq3$.  So $(\al,\be)\in R_q\sm R_{\S_q}$.  Thus, renaming $\al,\be$ if necessary, we have $\rank(\al)=q$, and either $\rank(\be)<q$ or else $\rank(\be)=q$ and $(\al,\be)\not\in{\H}$.  It then follows from Lemma \ref{lem:qp_Bn} or \ref{lem:notH_Bn}, respectively, that $\xi=R_q\sub\zeta$.

\bigskip\noindent {\bf Case 2.}  Now suppose $\xi=R_{N}$, where $N$ is one of $\A_q$ or $\S_q$ for some $q\geq3$.  Let $H$ be the largest normal subgroup of $\S_q$ strictly contained in $N$,
noting that $\xi^-=R_H$.  (Recall that $R_{q-2}=R_{\{\id_q\}}$.)  Because
$(\al,\be)\in\xi\sm\xi^-=\nu_N\sm\nu_H$, it follows that $\rank(\al)=\rank(\be)=q$,
$\al\H\be$ and $\phi(\al,\be)\in N\sm H$.  Consequently, we have
$N=\normal{\phi(\al,\be)}$, so Lemma~\ref{lem:thN_Bn} gives $R_{q-2}\sub\zeta$.
Now suppose $(\ga,\de)\in R_N\sm R_{q-2}$.  The proof will be complete if we can show that $(\ga,\de)\in\zeta$.  Write $\ga=
\partI{a_1}{a_q}{A_1}{A_r}{b_1}{b_q}{B_1}{B_r}$ and $\de=\partI{a_1}{a_q}{A_1}{A_r}{b_{1\pi}}{b_{q\pi}}{B_1}{B_r}$, where $a_1<\cdots<a_q$, noting that $\pi=\phi(\ga,\de)^{-1}\in N$.  As in the proof of Lemma \ref{lem:thN_Bn}, we may use the fact that $\pi\in N=\normal{\phi(\al,\be)}$ to show that $(\ve_q,\pib)\in\zeta$.  But then $(\ga,\de)=(\tau_1\ve_q\tau_2,\tau_1\pib\tau_2)\in\zeta$, where $\tau_1=\partI{a_1}{a_q}{A_1}{A_r}1q{q+1,q+2}{n-1,n}$ and $\tau_2=\partI1q{q+1,q+2}{n-1,n}{b_1}{b_q}{B_1}{B_r}$. \epf

We conclude this section with a number of observations.
Recall that the symmetric inverse monoid $\I_n$ is (isomorphic to) a submonoid of $\P_n$, and that we utilised a surjective map $\Phi_{\P_n,\I_n}:\Cong(\P_n)\to\Cong(\I_n)$ in the proof of Theorem \ref{thm-CongPn} (see especially Lemmas \ref{lem:mn}, \ref{lemma-aa2} and Proposition \ref{prop-aa3}).  
At the beginning of this section, we noted that the symmetric inverse monoid $\I_n$ does not embed into the Brauer monoid $\B_{n}$.  
However, for $n=2m$ even, 
$\I_m$ \emph{does} embed into the even degree Brauer monoid $\B_{2m}$.
Such an embedding $\I_m\to\B_{2m}:\al\mt\alt$ is illustrated by example in Figure~\ref{fig:ImB2m}.  Write $S\sub\B_{2m}$ for the image of this embedding.  As in Subsection \ref{sect:prelim_congruences}, we then obtain mappings
\begin{align*}
\Phi&=\Phi_{\B_{2m},S}:\Cong(\B_{2m})\to\Cong(S):\xi\mt\xi^S=\xi\cap(S\times S),\\
\Psi&=\Psi_{\B_{2m},S}:\Cong(S)\to\Cong(\B_{2m}):\zeta\mt\zeta_{\B_{2m}}^\sharp.
\end{align*}
But we note that neither $\Phi$ nor $\Psi$ is injective, and neither is surjective.  For example, one may check that
\begin{equation}\label{eq:InB2n}
R_{\S_{2q}}\Phi=R_{\A_{2q}}\Phi=R_{\S_q}^{\I_m} \AND
R^{\I_m}_{\S_{q}}\Psi=R^{\I_m}_{\A_{q}}\Psi=R_{\A_{2q}} \qquad\text{for any $3\leq q\leq n$.}
\end{equation}
The reason for \eqref{eq:InB2n} is that, by the nature of the embedding $\I_m\to\B_{2m}$, $\phi(\alt,\bet)$ is always an even permutation for any $\al,\be\in\I_n$ with $\al\H\be$, regardless of the parity of $\phi(\al,\be)$.
Curiously, though, we note that $R_{\S_2}^{\I_m}\Psi=\mu_K$.

\begin{figure}[ht]
\begin{center}
\begin{tikzpicture}[scale=1.2]
\uarcxx{1.75}{2.25}{.2}{lightgray!80}
\uarcxx{5.75}{6.25}{.2}{lightgray!80}
\darcxx{0.75}{1.25}{.2}{lightgray!80}
\darcxx{3.75}{4.25}{.2}{lightgray!80}
\stlinex{.75}{2.75}{lightgray!80}
\stlinex{1.25}{3.25}{lightgray!80}
\stlinex{2.75}{5.75}{lightgray!80}
\stlinex{3.25}{6.25}{lightgray!80}
\stlinex{3.75}{1.75}{lightgray!80}
\stlinex{4.25}{2.25}{lightgray!80}
\stlinex{4.75}{4.75}{lightgray!80}
\stlinex{5.25}{5.25}{lightgray!80}
\stlines{1/3,3/6,4/2,5/5}
\foreach \x in {1,...,6} {
\fill (\x,2)circle(.1); \fill (\x,0)circle(.1);
\fill[lightgray!80] (\x+.25,2)circle(.1);
\fill[lightgray!80] (\x-.25,2)circle(.1);
\fill[lightgray!80] (\x+.25,0)circle(.1);
\fill[lightgray!80] (\x-.25,0)circle(.1);
}
\end{tikzpicture}
\end{center}
\vspace{-5mm}
\caption{A partial permutation $\al\in\I_6$ (black), with its corresponding Brauer element $\alt\in\B_{12}$ (gray).}
\label{fig:ImB2m}
\end{figure}
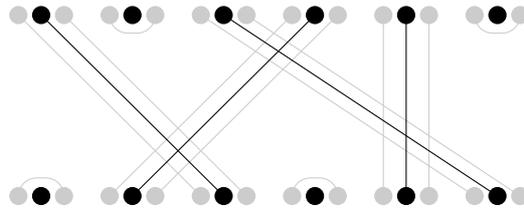

\section{The Jones monoid \boldmath{$\J_n$}}\label{sect:Jn}

Recall that the Jones monoid $\J_n=\B_n\cap\PP_n$ consists of all \emph{planar} Brauer elements.
In this final section, we describe the congruence lattice of $\J_n$.  Recall that in Section~\ref{sect:PPnMn} we were able to deduce descriptions of the planar congruence lattices $\Cong(\PP_n)$ and $\Cong(\M_n)$ by easily modifying the arguments used to treat their nonplanar counterparts $\Cong(\P_n)$ and $\Cong(\PB_n)$.  Unfortunately, and perhaps intriguingly, there does not appear to be a simple method for deriving a description of $\Cong(\J_n)$ from $\Cong(\B_n)$.  Roughly speaking, the reason for this is largely to do with the freedom (or lack thereof) in defining Motzkin or Jones elements.  For example, if we wish to define a Motzkin element $\al\in\M_n$ to have $k$ specified non-singleton blocks, we must only ensure that those blocks of $\al$ may be drawn in planar fashion, as the remaining blocks may be left as singletons.  However, if we wish to define a \emph{Jones} element $\al\in\J_n$ to have $k$ specified non-singleton blocks, we must ensure that it is possible to define the remaining $n-k$ (also non-singleton) blocks, without interfering with the planarity of $\al$.  We will discuss this in more detail later.

As in the previous section, we fix $n\geq3$ and $z\in\{0,1\}$ with $z\equiv n\pmod{2}$.
The $\gJ$-classes of $\J_n$ are $J_q=\set{\alpha\in\J_n}{\rank(\alpha)=q}$, for
$q=z,z+2,\dots,n$, the ideals are $I_q=J_z\cup\dots\cup J_q$, and the corresponding Rees congruences are denoted by $R_q=R_{I_q}$ (Definition \ref{defn:Rees}).
When $n$ is odd, we take the identity retraction on $I_1$ to get a retractable IN-pair $\C_1=(I_1,\{\overline{\id}_3\})$ and three congruences $\lambda_1$, $\rho_1$ and $\mu_1=\Delta$ (Definition~\ref{defn:lrmC}).
For $n$ even, it is easy to see that the retraction $f:\alpha\mapsto\widehat{\alpha}$ on $\B_n$ (Definition \ref{defn:retr_Bn} and Lemma \ref{Bretract}) preserves planarity, so it follows that $I_2$ is a retractable ideal of $\J_n$.
We then have retractable IN-pairs $\C_0=(I_0,\{\overline{\id}_2\})$
and $\C_2(I_2,\{\overline{\id}_4\})$, yielding the congruences $\lambda_0,\rho_0,\mu_0(=\Delta),\lambda_2,\rho_2,\mu_2$, respectively.  Again, no ideal $I_q$ of $\J_n$ with $q\geq3$ is retractable.  Since $\J_n$ is $\H$-trivial, there are no proper IN-pairs, and so no $R_N$ congruences (Definition \ref{defn:RC}).

\begin{thm}\label{thm:cong-Jn}
Let $n\geq3$, and let $\J_n$ be the Jones monoid of degree $n$.  Also let $z\in\{0,1\}$ be such that $n\equiv z\pmod{2}$.
\bit
\itemit{i} The ideals of $\J_n$ are the sets $I_q$, for $q=z,z+2,\ldots,n$, yielding the Rees congruences $R_q=R_{I_q}$, as in Definition \ref{defn:Rees}.
\itemit{ii} For $n$ odd, the only retractable IN-pair of $\J_n$ is $\C_1=(I_1,\{\overline\id_3\})$, yielding the congruences $\lambda_1$, $\rho_1$, $\mu_1$, as in Definition \ref{defn:lrmC}.
\itemit{iii} For $n$ even, the retractable IN-pairs of $\J_n$ are $\C_0=(I_0,\{\overline{\id}_2\})$ and $\C_2=(I_2,\{\overline{\id}_4\})$, yielding the congruences $\lambda_0$, $\rho_0$, $\mu_0$, $\lambda_2$, $\rho_2$, $\mu_2$, respectively, as in Definition \ref{defn:lrmC}.  
\itemit{iv} The above congruences are distinct, and they exhaust all the congruences of $\J_n$.
\itemit{v} The Hasse diagram of the congruence lattice $\Cong(\J_n) $ is shown in Figure \ref{fig:Hasse_Jn}.
\itemit{vi} The $\ast$-congruences of $\J_n$ are the same, but with $\lam_1$, $\rho_1$ ($n$ odd) or $\lam_0$, $\lam_2$, $\rho_0$, $\rho_2$ ($n$ even) excluded.
\eit
\end{thm}

\begin{figure}[ht]
\begin{center}
\scalebox{0.8}{
\begin{tikzpicture}[scale=1]
\dottedinterval0{12}11
\node[rounded corners,rectangle,draw,fill=blue!20] (N) at (0,15) {$R_n$};
  \draw(0.87,14.98) node {$=\nabla$};
\node[rounded corners,rectangle,draw,fill=blue!20] (S4) at (0,12) {$R_6$};
\node[rounded corners,rectangle,draw,fill=blue!20] (A4) at (0,10) {$R_4$};
\node[rounded corners,rectangle,draw,fill=blue!20] (krl) at (0,8) {$R_2$};
\node[rounded corners,rectangle,draw,fill=blue!20] (kr) at (-3,6) {$\lambda_2$};
\node[rounded corners,rectangle,draw,fill=blue!20] (kl) at (0,6) {$\rho_2$};
\node[rounded corners,rectangle,draw,fill=blue!20] (k) at (-3,4) {$\mu_2$};
\node[rounded corners,rectangle,draw,fill=blue!20] (erl) at (3,6) {$R_0$};
\node[rounded corners,rectangle,draw,fill=blue!20] (er) at (0,4) {$\lambda_0$};
\node[rounded corners,rectangle,draw,fill=blue!20] (el) at (3,4) {$\rho_0$};
\node[rounded corners,rectangle,draw,fill=blue!20] (e) at (0,2) {$\mu_0$};
  \draw(0.87,1.98) node {$=\Delta$};
\draw
(e)--(k)
(el)--(kl)
(k)--(kl) (e)--(el) 
(k)--(kr) (e)--(er) 
(el)--(erl) (kl)--(krl)
(erl)--(krl)--(A4)--(S4)
;
\fill[white] (-1.5,5)circle(.15);
\fill[white] (1.5,5)circle(.15);
\draw
(er)--(kr)
(er)--(erl) (kr)--(krl) 
;
\begin{scope}[shift={(12,0)}]
\dottedinterval0{12}11
\node[rounded corners,rectangle,draw,fill=blue!20] (N) at (0,15) {$R_n$};
  \draw(0.87,14.98) node {$=\nabla$};
\node[rounded corners,rectangle,draw,fill=blue!20] (S4) at (0,12) {$R_5$};
\node[rounded corners,rectangle,draw,fill=blue!20] (A4) at (0,10) {$R_3$};
\node[rounded corners,rectangle,draw,fill=blue!20] (krl) at (0,8) {$R_1$};
\node[rounded corners,rectangle,draw,fill=blue!20] (kr) at (-3,6) {$\lambda_1$};
\node[rounded corners,rectangle,draw,fill=blue!20] (erl) at (3,6) {$\rho_1$};
\node[rounded corners,rectangle,draw,fill=blue!20] (er) at (0,4) {$\mu_1$};
  \draw(0.87,3.98) node {$=\Delta$};
\draw
(krl)--(kr)--(er)--(erl)--(krl)--(A4)--(S4)
;
\end{scope}
\end{tikzpicture}
}
\end{center}
\vspace{-7mm}
\caption{Hasse diagrams for the congruence lattice $\Cong(\J_n)$ in the case of $n$ even (left) and $n$ odd (right).  All congruences are principal.  }
\label{fig:Hasse_Jn}
\end{figure}
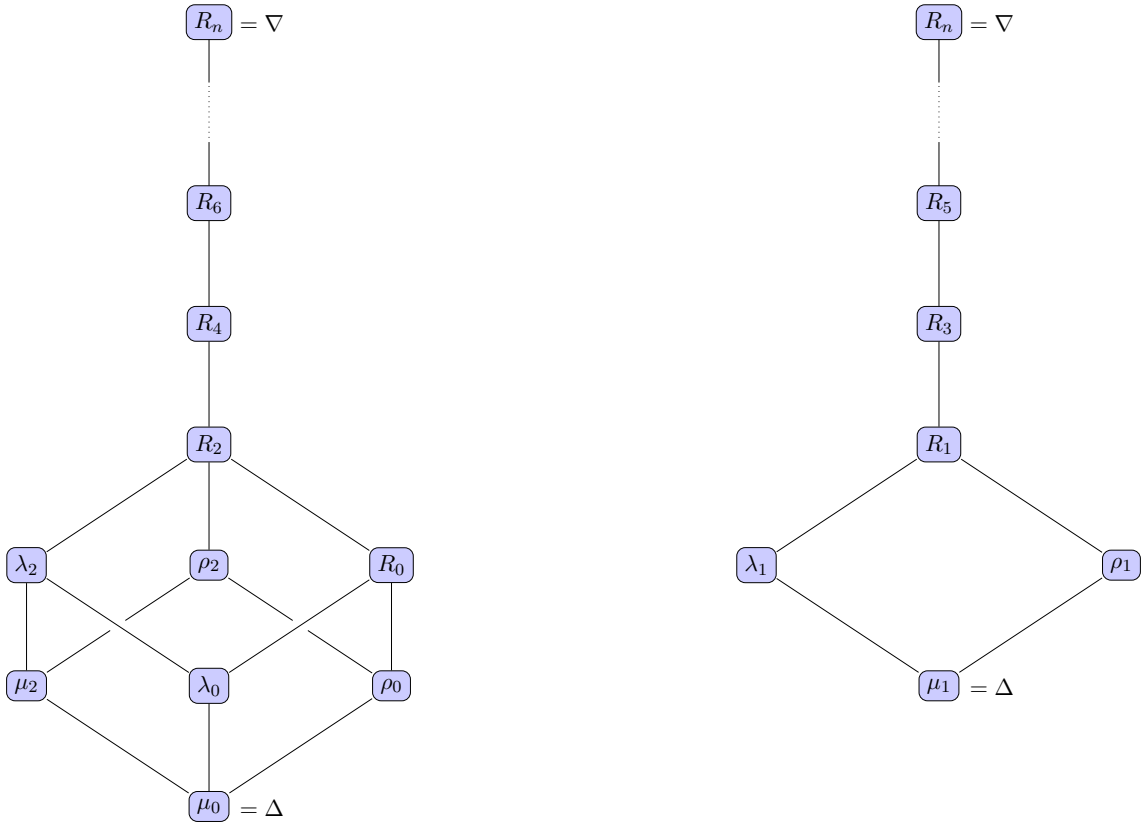

We will prove the theorem separately for $n$ even and for $n$ odd, beginning now with the even case.

\begin{proof}[{\bf Proof of Theorem \ref{thm:cong-Jn} for \boldmath{$n$} even.}]
Suppose $n=2m$ is even.  We have already mentioned that the Jones monoid $\J_n=\J_{2m}$ is isomorphic to the planar partition monoid $\PP_m$ \cite{HR2005,Jones1994_2}.  
Since we have already described the congruence lattice of $\PP_m$ (see Theorem \ref{thm-CongPPn}), we automatically obtain a description of $\Cong(\J_{2m})$.  To see that this agrees with the description given in Theorem \ref{thm:cong-Jn}, we need to recall the exact definition of the isomorphism between $\PP_m$ and $\J_{2m}$.  
We follow \cite[p873]{HR2005} and describe such an isomorphism $\PP_m\to\J_{2m}:\al\mt\alt$ by example in Figure \ref{fig:PmJ2m}; the fact that this map is well defined follows from the planarity of the \emph{canonical graphs} of planar partitions, as defined in Section \ref{sect:PPnMn} (see Lemma~\ref{lem:nested_or_separated}).  Of importance is the fact that $\rank(\alt)=2\rank(\al)$ for any $\al\in\PP_m$.
 It follows that $\xi\mt\xit=\bigset{(\alt,\bet)}{(\al,\be)\in\xi}$ defines an isomorphism $\Cong(\PP_m)\to\Cong(\J_{2m})$.  
It is now a routine matter to verify that under this isomorphism the congruences $R_q,\lambda_0,\lambda_1,\rho_0,\rho_1,\mu_0,\mu_1$ of $\PP_n$, as established by Theorem \ref{thm-CongPPn}, respectively correspond to the congruences
$R_{2q},\lambda_0,\lambda_2,\rho_0,\rho_2,\mu_0,\mu_2$ of $\J_n$.
\end{proof}

\begin{figure}[ht]
\begin{center}
\begin{tikzpicture}[scale=1.2]
\uarcxx{1.25}{3.75}{.65}{lightgray!80}
\uarcxx{1.75}{3.25}{.5}{lightgray!80}
\uarcxx{2.25}{2.75}{.2}{lightgray!80}
\uarcxx{4.75}{5.25}{.2}{lightgray!80}
\uarcxx{5.75}{6.25}{.2}{lightgray!80}
\uarcxx{7.75}{8.25}{.2}{lightgray!80}
\uarcx14{.8}
\uarcx23{.35}
\darcxx{2.75}{3.25}{.2}{lightgray!80}
\darcxx{1.25}{1.75}{.2}{lightgray!80}
\darcxx{5.25}{5.75}{.2}{lightgray!80}
\darcxx{6.25}{6.75}{.2}{lightgray!80}
\darcxx{2.25}{3.75}{.4}{lightgray!80}
\darcxx{4.75}{7.25}{.55}{lightgray!80}
\stlinex{.75}{.75}{lightgray!80}
\stlinex{4.25}{4.25}{lightgray!80}
\stlinex{6.75}{7.75}{lightgray!80}
\stlinex{7.25}{8.25}{lightgray!80}
\darcx12{.35}
\darcx24{.6}
\darcx56{.35}
\darcx67{.35}
\stlines{1/1,4/4,7/8}
\foreach \x in {1,...,8} {
\fill (\x,2)circle(.1); \fill (\x,0)circle(.1);
\fill[lightgray!80] (\x+.25,2)circle(.1);
\fill[lightgray!80] (\x-.25,2)circle(.1);
\fill[lightgray!80] (\x+.25,0)circle(.1);
\fill[lightgray!80] (\x-.25,0)circle(.1);
}
\end{tikzpicture}
\end{center}
\vspace{-5mm}
\caption{A planar partition $\al\in\PP_8$ (black), with its corresponding Jones element $\alt\in\J_{16}$ (gray).}
\label{fig:PmJ2m}
\end{figure}
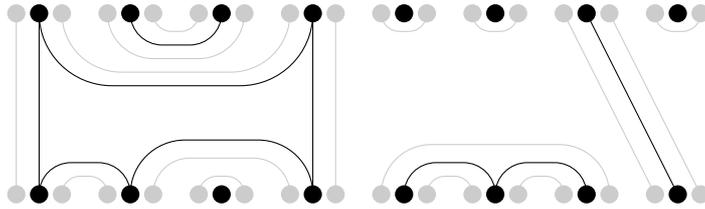

The proof of Theorem \ref{thm:cong-Jn} in the odd case follows the same pattern as the proofs of our previous main theorems.  At several points, we will appeal back to the fact that the even case is true.  In particular, we require the next result, which follows immediately from parts (i) and (ii) of Proposition \ref{prop:small_congruences:PPn}, together with the isomorphism $\PP_m\to\J_{2m}:\al\mt\alt$ described in the above proof.

\newpage

\begin{prop}[cf.~Proposition \ref{prop:small_congruences:PPn}]
\label{prop:rlJeven}
For $n\geq4$ even, the relations $\rho_0$ and $\lambda_0$ are principal congruences on $\J_n$.  Moreover, if $\al,\be\in\J_n$, then
\begin{itemize}\begin{multicols}{2}
\itemit{i} $\rho_0=\cg\al\be\iff (\al,\be)\in\rho_0\sm\Delta$, 
\itemit{ii} $\lambda_0=\cg\al\be\iff (\al,\be)\in\lambda_0\sm\Delta$.
\epfres
\end{multicols}\end{itemize}
\end{prop}

Next, we introduce some notation, and discuss some concepts relating to planarity.  The following definition makes no assumption on the parity of $n$.

\begin{defn}
For $A\sub\bn$, write $A=\{a_1<\cdots<a_k\}$ to indicate that $A=\{a_1,\ldots,a_k\}$ and $a_1<\cdots<a_k$.  For $i,j\in\bn$, denote the interval $\set{k\in\bn}{i\leq k\leq j}$ by~$[i,j]$.  
Let $A=\{a<b\}$ and $B=\{c<d\}$ be two disjoint subsets of $\bn$ of size $2$.  As in Section \ref{sect:PPnMn}, we say $A$ and $B$ are \emph{nested} if $a<c<d<b$ or $c<a<b<d$; we say $A$ and $B$ are \emph{separated} if $a<b<c<d$ or $c<d<a<b$.  As in Section \ref{sect:Bn}, we say a partition of a set is a 2-partition if each block has size $2$.  We say a 2-partition of an interval $I\sub\bn$ is \emph{planar} if any pair of blocks is either nested or separated.  
\end{defn}

It is clear that a planar 2-partition exists on an interval $I\sub\bn$ if and only if $|I|$ is even, in which case the number of distinct planar 2-partitions is given by a suitable Catalan number \cite[Sequence A000108]{OEIS}.  
It is also clear that if $A=\{a<b\}$ is a block of a planar 2-partition of an interval, then one of $a,b$ is even and the other odd;
we call $A$ \emph{even} or \emph{odd} according to whether $a=\min(A)$ is even or odd, respectively.
It is easy to check that a planar 2-partition on an interval is uniquely determined by its odd (respectively, even) blocks. 
From this we immediately deduce the following.

\begin{lemma}\label{lem:odd}
Suppose $\bP,\bQ$ are two planar 2-partitions of an interval $I\sub\bn$ with $\bP\not=\bQ$.  Then, renaming $\bP,\bQ$ if necessary, there exists an even block and an odd block of $\bP$, neither of which is a block of $\bQ$. \epfres
\end{lemma}

To see the relevance of planar 2-partitions, consider a Jones element ${\al=\partI{a_1}{a_q}{A_1}{A_r}{b_1}{b_q}{B_1}{B_r}\in\J_n}$, where $n$ is arbitrary, and where~$a_1<\cdots<a_q$
(and, of course,  $q\equiv n\pmod2$).  By planarity (cf.~Lemma~\ref{lem:nested_or_separated}), we also have~${b_1<\cdots<b_q}$, and the non-transversals of $\al$ induce planar 2-partitions of the intervals
\begin{gather*}
[1,a_1-1]\ , \ [a_1+1,a_2-1] \ ,\ \ldots\ ,\ [a_{q-1}+1,a_q-1]\ ,\  [a_q+1,n] , \\
[1,b_1-1]\  , \ [b_1+1,b_2-1] \ ,\ \ldots\ ,\ [b_{q-1}+1,b_q-1]\ ,\  [b_q+1,n].
\end{gather*}
Note that some of these intervals may be empty.  Since all the above intervals must have even size, it follows that $a_i,b_i$ are both odd when $i$ is odd, or both even when $i$ is even: that is, $a_i\equiv b_i\equiv i\pmod2$.  Note also that for any upper non-transversal $A_k=\{x<y\}$, and for any $1\leq i\leq q$, we have either $y<a_i$ or $a_i<x$.  

\begin{rem}\label{rem:Jdefine}
At several stages in the subsequent argument, we will need to define a Jones element $\al\in\J_n$ with specified $\dom(\al)=\{a_1<\cdots<a_q\}$, $\codom(\al)=\{b_1<\cdots<b_q\}$, and a single extra non-transversal $\{x,y\}$ or $\{x',y'\}$ for some $x,y\in\bn$.  
Keeping the previous paragraph in mind, we must be careful to ensure that ${q\equiv n\pmod2}$, that $a_i\equiv b_i\equiv i\pmod2$ for all $i$, that one of $x,y$ is even and the other odd, and that the specified non-transversal does not intersect any of the transversals $\{a_i,b_i'\}$.  Conversely, it is easy to see that if these conditions are satisfied, then such a required Jones element $\al\in\J_n$ exists.  
\end{rem}

With the above concepts in place, we are now ready to investigate the congruences on odd-degree $\J_n$.
For the rest of this section, unless otherwise specified, we assume $n=2m+1\geq3$ is odd.  
We begin with the lower congruences, $\rho_1$, $\lam_1$ and $R_1$.  The first main step is to prove that $\rho_1$ and $\lam_1$ are principal (Proposition~\ref{prop:rlJodd}); the pattern of proof is similar to that of Proposition \ref{prop:rl_Bn}, but the planarity restriction on Jones elements means that we have to work quite a bit harder.  For convenience, we will divide the argument into a number of technical lemmas.

\begin{lemma}\label{lem:bzn}
Let $n\geq3$ be odd, and let $(\al,\be)\in\rho_1\sm\Delta$.  For any $\si\in\J_n$ with $\rank(\si)=1$, there exists some $\pi\in\J_n$ with $(\si,\pi)\in\cg\al\be$ and $\codom(\pi)=\{n\}$.
\end{lemma}

\pf Write ${\mathrel\xi}=\cg\al\be$ and $\si=\partIV
a{A_1}\cdots{A_m}b{B_1}\cdots{B_m}$, noting that $a,b$ are both odd.  If $b=n$, then there is nothing to prove,
so suppose $b\not=n$.  Without loss of generality, we may assume that $n\in B_1$, and we may write $B_1=\{w,n\}$.  By planarity, we have $b<w<n$, and $w$ is even.  
We will prove that $\si\mathrel\xi\pi$, where $\pi=\partVI a{A_1}{A_2}{A_m}n{b,w}{B_2}{B_m}$.
To do this, we will define a Jones element $\tau\in\J_n$ such that 
\begin{equation}\label{eq:bzn}
\si = \si\al\tau \ \ \text{and} \ \ \pi = \si\be\tau
\OR
\si = \si\be\tau \ \ \text{and} \ \  \pi = \si\al\tau.
\end{equation}
The exact definition of $\tau$ varies according to several possible cases we will enumerate below.  But in every case, we will have $\codom(\tau)=\{b,w,n\}$, and the non-trivial $\coker(\tau)$-classes will be $B_2,\ldots,B_m$.  
Let the transversals of $\al$ and $\be$ be $\{c,d'\}$ and $\{c,e'\}$, respectively, noting that $c,d,e$ are all odd.  

\bigskip\noindent {\bf Case 1.}  Suppose first that $d\not=e$.  Without loss of generality, we may assume that $d<e$.  Since $d,e$ are odd, we may define $\dom(\tau)=\{d,d+1,e\}$, and we choose the non-trivial $\ker(\tau)$-classes arbitrarily.  
Then one may easily check that $\si=\si\al\tau$, while $\pi=\si\be\tau$, as required.  

\bigskip\noindent {\bf Case 2.}  Now suppose $d=e$, noting that this forces $n\geq5$.  
Recall that $\codom(\al)=\codom(\be)=\{d\}$.  So $\coker(\al)$ induces two planar 2-partitions $\bP_1$ and $\bP_2$ of the intervals $[1,d-1]$ and $[d+1,n]$, respectively.  Similarly, $\coker(\be)$ induces two planar 2-partitions $\bQ_1$ and $\bQ_2$ of the intervals $[1,d-1]$ and $[d+1,n]$, respectively.  Since $\coker(\al)\not=\coker(\be)$, we must have either $\bP_1\not=\bQ_1$ or $\bP_2\not=\bQ_2$.  We consider these possibilities separately.

\bigskip\noindent {\bf Subcase 2.1.}  Suppose first that $\bP_1\not=\bQ_1$.  By Lemma \ref{lem:odd}, renaming $\al,\be$ if necessary, we may choose some block $\{x<y\}$ of $\bP_1$ that is not a block of $\bQ_1$, with $x$ odd.  Let the blocks of $\bQ_1$ containing $x$ and $y$ be $\{x,u\}$ and $\{y,v\}$.  Note that $y,u$ are even, and~$v$ is odd.  
Also, considering that $\{x',u'\}$ and $\{y',v'\}$ are both blocks of~$\be$, planarity ensures that we must be in one of the following eight cases:
\begin{itemize}
\begin{multicols}{4}
\item[(a)] $v<u<x<y<d$,
\item[(b)] $u<x<v<y<d$,
\item[(c)] $u<x<y<v<d$,
\item[(d)] $v<x<u<y<d$,
\item[(e)] $x<u<v<y<d$,
\item[(f)] $x<u<y<v<d$, 
\item[(g)] $x<v<y<u<d$, or
\item[(h)] $x<y<v<u<d$.
\end{multicols}
\eit
Depending on the case we are in, we then define $\dom(\tau)$ to be
\begin{itemize}
\begin{multicols}{4}
\item[(a)] $\{v,u,x\}$,
\item[(b)] $\{1,u,x\}$,
\item[(c)] $\{1,u,x\}$,
\item[(d)] $\{v,v+1,x\}$, 
\item[(e)] $\{x,u,v\}$,
\item[(f)] $\{x,u,u+1\}$, 
\item[(g)] $\{x,y,y+1\}$, or
\item[(h)] $\{x,y,v\}$,
\end{multicols}
\eit
respectively.
In each of cases (a--f), we specify that $\{y,d\}$ is a $\ker(\tau)$-class.  In cases (g) and (h), we specify that $\{u,d\}$ is a $\ker(\tau)$-class.  In all cases, all other non-trivial $\ker(\tau)$-classes may be chosen arbitrarily.  
Figure~\ref{fig:bzn} gives a diagrammatic verification that equation \eqref{eq:bzn} holds in each of cases (a)--(h).  In Figure~\ref{fig:bzn}, we have only pictured: the single transversal $\{c,d'\}$ of $\al$ and $\be$; the non-transversal $\{x',y'\}$ of $\al$; the non-transversals $\{x',u'\}$ and $\{y',v'\}$ of $\be$; the three transversals of $\tau$; and the single specified upper non-transversal of $\tau$.  In light of Remark \ref{rem:Jdefine}, odd-labelled vertices are drawn black, and even-labelled vertices white, so that the reader may easily check that $\tau$ is well defined in all cases.  For reasons of space, we have omitted (double) dashes on vertices, and we have written $i^+=i+1$ for any~$i\in\bn$.  No suggestion of the actual values of $b,c,d,u,v,w,x,y$ are intended, but the ordering of the points $b,w,n$ and of $u,v,x,y,d$ are accurate in each diagram; note, however, that no suggestion of the relative ordering of a point from the first list and a point from the second list is intended in the diagram (so, for example, in any diagram, we could have $b<v$, $b=v$ or $b>v$).

\bigskip\noindent {\bf Subcase 2.2.}  Now suppose $\bP_2\not=\bQ_2$.  By Lemma \ref{lem:odd}, renaming $\al,\be$ if necessary, we may choose some block $\{x<y\}$ of $\bP_2$ that is not a block of $\bQ_2$, with $x$ even.  Let the blocks of $\bQ_2$ containing $x$ and $y$ be $\{x,u\}$ and $\{y,v\}$.  This time, note that $y,u$ are odd, and $v$ is even.  Again, planarity ensures that we are in one of the following eight cases:
\begin{itemize}
\begin{multicols}{4}
\item[(a)] $d<v<u<x<y$,
\item[(b)] $d<u<x<v<y$,
\item[(c)] $d<u<x<y<v$,
\item[(d)] $d<v<x<u<y$,
\item[(e)] $d<x<u<v<y$,
\item[(f)] $d<x<u<y<v$, 
\item[(g)] $d<x<v<y<u$, or
\item[(h)] $d<x<y<v<u$.
\end{multicols}
\eit
Depending on the case we are in, we then define $\dom(\tau)$ to be
\begin{itemize}
\begin{multicols}{4}
\item[(a)] $\{u,x,y\}$,
\item[(b)] $\{x+1,v,y\}$,
\item[(c)] $\{y,v,n\}$,
\item[(d)] $\{v+1,x,y\}$, 
\item[(e)] $\{u,v,y\}$,
\item[(f)] $\{y,v,n\}$, 
\item[(g)] $\{y,y+1,u\}$, or
\item[(h)] $\{y,v,u\}$,
\end{multicols}
\eit
respectively.
In cases (a) and (d), we specify that $\{d,v\}$ is a $\ker(\tau)$-class.  In all other cases, we specify that $\{d,x\}$ is a $\ker(\tau)$-class.  In all cases, all other non-trivial $\ker(\tau)$-classes may be chosen arbitrarily.  The reader may draw diagrams such as those in Figure~\ref{fig:bzn} to verify that $\tau$ has the desired properties. \epf

\begin{figure}[ht]
\begin{center}
\scalebox{.82}{ % have to fiddle with scale to get picture to fit
\begin{tikzpicture}[scale=.45]
\alphabetataudiagram{7*\shiftup}{\vv\uu\xx\yy}{.3}{.8}{\vv}{\uu}{\xx}{\yy}{\caselabel{a} \uarc\yy\dd \uarc{\yy+\shiftvalue}{\dd+\shiftvalue}
\udarcx{\uu+\shiftvalue}{\xx+\shiftvalue}{\sillyone} 
\udarcx{\vv+\shiftvalue}{\yy+\shiftvalue}{\sillytwo}
}
\alphabetataudiagram{6*\shiftup}{\uu\xx\vv\yy}{.5}{.5}{\uu-2}{\uu}{\xx}{\yy}{\caselabel{b}  \uarc\yy\dd \uarc{\yy+\shiftvalue}{\dd+\shiftvalue}
\uvlab{\uu-2}{1}\uvlab{\uu-2+\shiftvalue}{1}
\udarcx{\uu+\shiftvalue}{\xx+\shiftvalue}{\sillyone} 
\udarcx{\vv+\shiftvalue}{\yy+\shiftvalue}{\sillytwo}}
\alphabetataudiagram{5*\shiftup}{\uu\xx\yy\vv}{.5}{.5}{\uu-2}{\uu}{\xx}{\vv+2}{\caselabel{c} \uarc\yy\dd \uarc{\yy+\shiftvalue}{\dd+\shiftvalue}\uvlab{\uu-2}{1}\uvlab{\uu-2+\shiftvalue}{1}
\udarcx{\uu+\shiftvalue}{\xx+\shiftvalue}{\sillyone} 
\udarcx{\yy+\shiftvalue}{\vv+\shiftvalue}{\sillytwo}}
\alphabetataudiagram{4*\shiftup}{\vv\xx\uu\yy}{.3}{.8}{\vv}{\vv+1}{\xx}{\yy}{\caselabel{d} \uarc\yy\dd \uarc{\yy+\shiftvalue}{\dd+\shiftvalue}
\uvlabw{\vv+1}{v^+}\uvlabw{\vv+1+\shiftvalue}{v^+}
\udarcx{\xx+\shiftvalue}{\uu+\shiftvalue}{\sillyone} 
\udarcx{\vv+\shiftvalue}{\yy+\shiftvalue}{\sillytwo}}
\alphabetataudiagram{3*\shiftup}{\xx\uu\vv\yy}{.5}{.5}{\xx}{\uu}{\vv}{\yy}{\caselabel{e} \uarc\yy\dd \uarc{\yy+\shiftvalue}{\dd+\shiftvalue}
\udarcx{\xx+\shiftvalue}{\uu+\shiftvalue}{\sillyone} 
\udarcx{\vv+\shiftvalue}{\yy+\shiftvalue}{\sillytwo}}
\alphabetataudiagram{2*\shiftup}{\xx\uu\yy\vv}{.5}{.5}{\xx}{\uu}{\uu+1}{\yy}{\caselabel{f} \uarc\yy\dd \uarc{\yy+\shiftvalue}{\dd+\shiftvalue}
\uvlab{\uu+1}{u^+}\uvlab{\uu+1+\shiftvalue}{u^+}
\udarcx{\xx+\shiftvalue}{\uu+\shiftvalue}{\sillyone} 
\udarcx{\yy+\shiftvalue}{\vv+\shiftvalue}{\sillytwo}}
\alphabetataudiagram{1*\shiftup}{\xx\vv\yy\uu}{.8}{.3}{\xx}{\yy}{\yy+1}{\yy}{\caselabel{g} \uarc\uu\dd \uarc{\uu+\shiftvalue}{\dd+\shiftvalue}\uvlab{\yy+1}{y^+}\uvlab{\yy+1+\shiftvalue}{y^+}
\udarcx{\xx+\shiftvalue}{\uu+\shiftvalue}{\sillyone} 
\udarcx{\vv+\shiftvalue}{\yy+\shiftvalue}{\sillytwo}}
\alphabetataudiagram{0*\shiftup}{\xx\yy\vv\uu}{.8}{.3}{\xx}{\yy}{\vv}{\uu}{\caselabel{h} \uarc\uu\dd \uarc{\uu+\shiftvalue}{\dd+\shiftvalue}
\udarcx{\xx+\shiftvalue}{\uu+\shiftvalue}{\sillyone} 
\udarcx{\yy+\shiftvalue}{\vv+\shiftvalue}{\sillytwo}}
\foreach \u in {0,1,...,8} {\draw(-6,-1+\shiftup*\u)--(\aaa+\shiftvalue+2,-1+\shiftup*\u);}
\foreach \u in {-6,-4,14,32} {\draw(\u,-1)--(\u,-1+8*\shiftup);}
\end{tikzpicture}
}
\end{center}
\vspace{-5mm}
\caption{Verification of equation \eqref{eq:bzn}; see the proof of Lemma \ref{lem:bzn} for more details.}
\label{fig:bzn}
\end{figure}
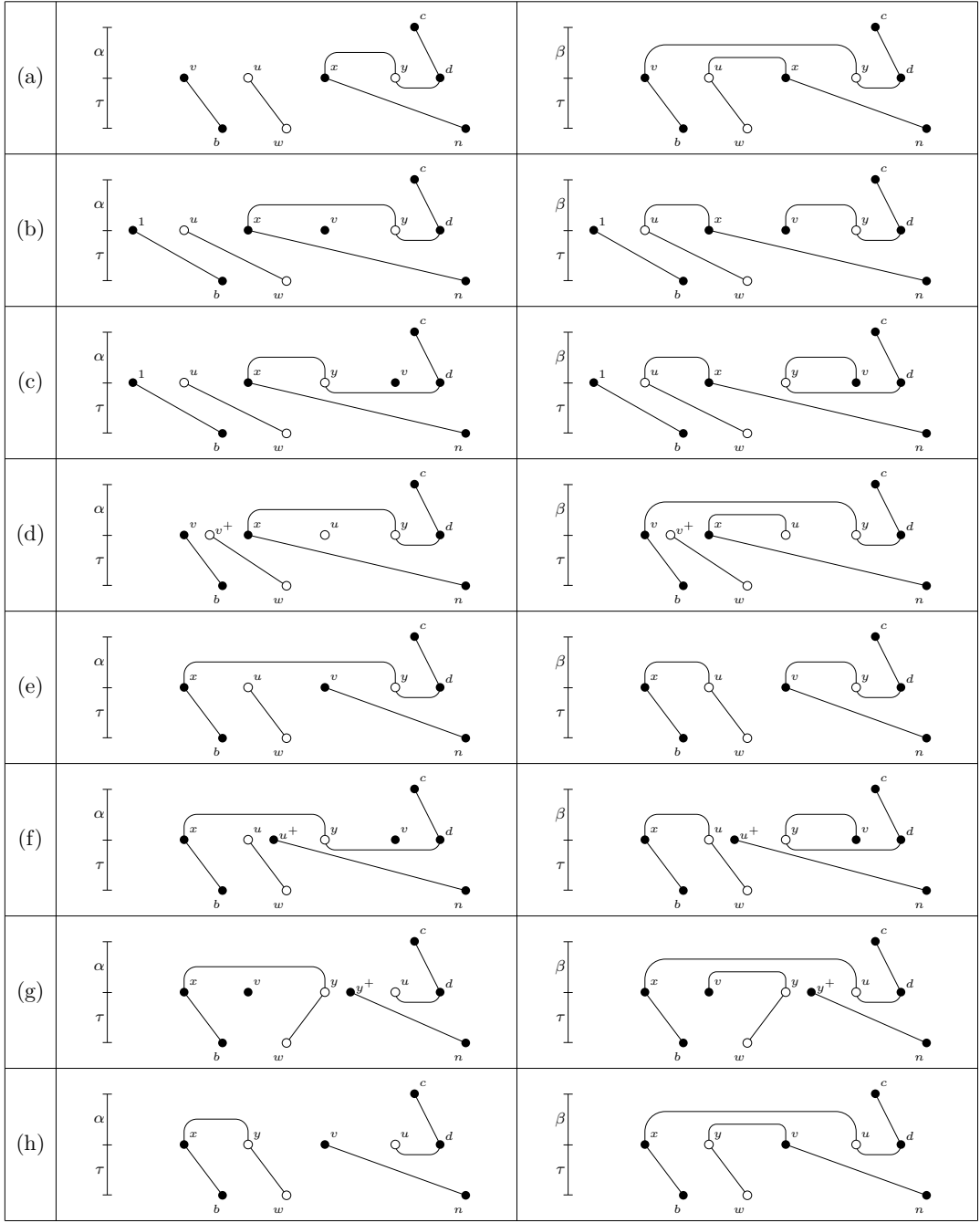

As in the proof of Proposition \ref{prop:rl_Bn}, for an integer $i$, we write $H_i=\{2i-1,2i\}$.  The proof of the next result could be extracted from a section of the proof of Proposition \ref{prop:rl_Bn}, but we provide the details here for convenience.

\begin{lemma}\label{lem:Jtech1}
Let $n\geq5$ be odd, and let $(\al,\be)\in\rho_1\sm\Delta$.  Then $(\ga,\de)\in\cg\al\be$, where
\[
\ga=\partIV n{H_1}\cdots{H_m}n{H_1}\cdots{H_m} \AND \de=\Big( 
{ \scriptsize \renewcommand*{\arraystretch}{1}
\begin{array} {\c|\c|\c|\c|\c|\cend}
n \:&\: H_1 \:&\: \cdots \:&\: H_{m-2} \:&\: H_{m-1} \:&\: H_m  \\ \cline{2-6}
n \:&\: H_1 \:&\: \cdots \:&\: H_{m-2} \:&\: n-4,n-1 \:&\: n-3,n-2
\rule[0mm]{0mm}{2.7mm}
\end{array} 
}
\hspace{-1.5 truemm} \Big).
\]
\end{lemma}

\pf Write ${\mathrel\xi}=\cg\al\be$.  First, define $\si=\Big( 
{ \scriptsize \renewcommand*{\arraystretch}{1}
\begin{array} {\c|\c|\c|\c|\cend}
n \:&\: H_1 \:&\: \cdots \:&\: H_{m-1} \:&\: H_m  \\ \cline{2-5}
n-2 \:&\: H_1 \:&\: \cdots \:&\: H_{m-1} \:&\: n-1,n
\rule[0mm]{0mm}{2.7mm}
\end{array} 
}
\hspace{-1.5 truemm} \Big)$.
The proof of Lemma \ref{lem:bzn} gives $\si\mathrel\xi\ga$; in the notation of that proof, $a=n$, $b=n-2$ and $w=n-1$.  Define
$
\tau_1 = 
\Big( 
{ \scriptsize \renewcommand*{\arraystretch}{1}
\begin{array} {\c|\c|\c|\c|\c|\c|\cend}
n-2 \:&\: n-1 \:&\: n \:&\: H_1 \:&\: \cdots \:&\: H_{m-2} \:&\: H_{m-1}  \\ \cline{4-7}
n-4 \:&\: n-3 \:&\: n-2 \:&\: H_1 \:&\: \cdots \:&\: H_{m-2} \:&\: n-1,n
\rule[0mm]{0mm}{2.7mm}
\end{array} 
}
\hspace{-1.5 truemm} \Big)
$ and $
\tau_2 =
\Big( 
{ \scriptsize \renewcommand*{\arraystretch}{1}
\begin{array} {\c|\c|\c|\c|\c|\c|\cend}
n-2 \:&\: n-1 \:&\: n \:&\: H_1 \:&\: \cdots \:&\: H_{m-2} \:&\: H_{m-1}  \\ \cline{4-7}
n-4 \:&\: n-1 \:&\: n \:&\: H_1 \:&\: \cdots \:&\: H_{m-2} \:&\: n-3,n-2
\rule[0mm]{0mm}{2.7mm}
\end{array} 
}
\hspace{-1.5 truemm} \Big)$.  
Then $\ga \mathrel\xi \si = \ga\tau_1 \mathrel\xi \si\tau_1 = \si\tau_2 \mathrel\xi \ga\tau_2 = \de$. \epf

\begin{lemma}\label{lem:Jtech2}
Let $n\geq3$ be odd, and let $(\al,\be)\in\rho_1\sm\Delta$.  Let $\tau_1,\tau_2\in\J_n$ be such that $\dom(\tau_1)=\codom(\tau_1)=\dom(\tau_2)=\codom(\tau_2)=\{n\}$ and $\ker(\tau_1)=\ker(\tau_2)$.
Then $(\tau_1,\tau_2)\in\cg\al\be$.
\end{lemma}

\pf If $n=3$, then the assumptions force $\tau_1=\tau_2$, so we assume $n\geq5$ for the rest of the proof.  Let~$\ga,\de$ be as in Lemma \ref{lem:Jtech1}.  Write $S$ for the submonoid of $\J_n$ consisting of all Jones elements containing the block $\{n,n'\}$.  Note that $S$ is isomorphic to $\J_{n-1}$, and that $\ga,\de,\tau_1,\tau_2\in S$.  For $\si\in S$, write $\si^\vee$ for the element of $\J_{n-1}$ obtained from $\si$ by removing the block $\{n,n'\}$.  Write $\mathrel\xi$ for the $\rho_0$-congruence on $\J_{n-1}$.  By Proposition~\ref{prop:rlJeven}, ${\mathrel\xi}=\cg{\ga^\vee}{\de^\vee}$.  So $(\tau_1^\vee,\tau_2^\vee)\in\cg{\ga^\vee}{\de^\vee}$, and it follows that $(\tau_1,\tau_2)\in\cg\ga\de\sub\cg\al\be$, as required. \epf

\begin{prop}\label{prop:rlJodd}
Let $n\geq3$ be odd.  The relations $\rho_1$ and $\lambda_1$ are principal congruences on $\J_n$.  Moreover, if $\al,\be\in\J_n$, then
\begin{itemize}\begin{multicols}{2}
\itemit{i} $\rho_1=\cg\al\be\iff (\al,\be)\in\rho_1\sm\Delta$,
\itemit{ii} $\lambda_1=\cg\al\be\iff (\al,\be)\in\lambda_1\sm\Delta$.
\end{multicols}\end{itemize}
\end{prop}

\pf As usual, it suffices to prove the converse implication in (i), so fix some $(\al,\be)\in\rho_1\sm\Delta$, and
write ${\mathrel\xi}=\cg\al\be$.  
Let $(\si_1,\si_2)\in\rho_1\sm\Delta$ be arbitrary.  
By Lemma \ref{lem:bzn}, $\si_1\mathrel\xi\pi_1$ and $\si_2\mathrel\xi\pi_2$, for some $\pi_1,\pi_2\in\J_n$ with $\codom(\pi_1)=\codom(\pi_2)=\{n\}$, and we note that $\ker(\pi_1)=\ker(\si_1)=\ker(\si_2)=\ker(\pi_2)$, since ${\mathrel\xi}\sub\rho_1$.  
Let $\ga$ be as in Lemma \ref{lem:Jtech1}, and put $\tau_1=\ga\pi_1$ and $\tau_2=\ga\pi_2$.  Then $\tau_1,\tau_2$ satisfy the conditions of Lemma \ref{lem:Jtech2}, so it follows by that lemma that $\tau_1\mathrel\xi\tau_2$.  But then $\si_1\mathrel\xi\pi_1=\pi_1\ga\pi_1=\pi_1\tau_1\mathrel\xi\pi_1\tau_2=\pi_1\ga\pi_2=\pi_2\mathrel\xi\si_2$; note that $\pi_1\ga\pi_2=\pi_2$ follows from the fact that $\ker(\pi_1)=\ker(\pi_2)$. \epf

The proof of the next result is virtually identical to that of Proposition \ref{prop:R01_Bn}.

\begin{prop}\label{prop:R01_Jn}
Let $n\geq3$ be odd.  Then $R_1$ is a principal congruence on $\J_n$.  Moreover, if $\al,\be\in\J_n$, then $R_1=\cg\al\be\iff(\al,\be)\in R_1\sm(\rho_1\cup\lambda_1)$.  \epfres
\end{prop}

Having now described the generating pairs for the lower congruences $\rho_1,\lam_1,R_1$, we turn our attention to the chain of Rees congruences:
\[
R_1\suq R_3\suq R_5\suq\cdots\suq R_n=\nabla.
\]
Specifically, we wish to show that each congruence in this list is principal, and that if $q\geq3$ is odd, then~$R_q$ is generated by any pair from $R_q\sm R_{q-2}$; see Proposition \ref{prop:Rq_Jn}.
As in Section \ref{sect:Bn}, we first obtain sufficient conditions for a congruence $\xi$ on $\J_n$ to contain a Rees congruence~$R_q$; see Lemmas \ref{lem:pq2_Jn} and \ref{lem:qq_Jn}.
The proof of the analogous result concerning the Brauer monoid~$\B_n$, Lemma~\ref{lem:qp_Bn}, relied on the fact that $\B_n$ contains the symmetric group~$\S_n$; since $\J_n$ does not contain $\S_n$, we will have to work harder to prove Lemmas \ref{lem:pq2_Jn} and \ref{lem:qq_Jn}.
We begin with a simple result that will be of use on several occasions; its statement and proof do not assume $n$ is odd.

\begin{lemma}\label{lem:pq_Jn}
Let $n\geq3$ be arbitrary.  Let $\al,\be\in\J_n$ with $p=\rank(\al)<q=\rank(\be)$.  Suppose ${\xi\in\Cong(\J_n)}$ is such that $R_p\sub\xi$ and $(\al,\be)\in\xi$.  Then $R_q\sub\xi$.
\end{lemma}

\pf Since $R_p\sub\xi$, it suffices to show that any Jones element $\ga\in\J_n$ of rank at most $q$ is $\xi$-related to an element of rank at most $p$.  Since $\rank(\ga)\leq\rank(\be)$, it follows from Proposition \ref{prop:green_all_inclusive}(iv) that $\ga=\de_1\be\de_2$ for some $\de_1,\de_2\in\J_n$.  But then $\ga=\de_1\be\de_2 \mathrel\xi \de_1\al\de_2$, and we are done, since $\rank(\de_1\al\de_2)\leq\rank(\al)=p$. \epf

\begin{lemma}[cf.~Lemma \ref{lem:1z_Bn}]
\label{lem:1z_Jn}
Let $n\geq3$ be odd.  Suppose $\xi\in\Cong(\J_n)$ and that $(\1,\al)\in\xi$ for some $\al\in J_1$.  Then $\xi=R_n$.
\end{lemma}

\pf By Lemma \ref{lem:pq_Jn}, it suffices to show that $R_1\sub\xi$.  Let $A$ be a nontrivial $\coker(\al)$-class, and let $\be\in J_1$ be such that $A$ is not a $\coker(\be)$-class.  Then $(\be,\be\al)=(\be\;\!\1,\be\al)\in\xi$.  But $\rank(\be\al)=\rank(\be)=1$, $\ker(\be\al)=\ker(\be)$, yet $\coker(\be\al)\not=\coker(\be)$, since $A$ is a $\coker(\be\al)$-class but not a $\coker(\be)$-class.  In particular, $(\be,\be\al)\in\rho_1\sm\Delta$, so Proposition \ref{prop:rlJodd} gives $\rho_1=\cg{\be}{\be\al}\sub\xi$.  Similarly, $\lambda_1\sub\xi$, and it follows from Proposition \ref{prop-CR4} that $R_1=\rho_1\vee\lam_1\sub\xi$, as required. \epf

In what follows,
we will make extensive use of the idempotents $\ve_q=\partI1q{q+1,q+2}{n-1,n}1q{q+1,q+2}{n-1,n}$, for $q=1,3,\ldots,n$, already used in Section \ref{sect:Bn}.  We note that $\ve_q\in\J_n$, and that $\ve_q\J_n\ve_q$ is a monoid with identity~$\ve_q$, isomorphic to $\J_q$.

\begin{lemma}\label{lem:qz_Jn}
Let $n\geq3$ be odd.  Suppose $\xi\in\Cong(\J_n)$ and $(\al,\be)\in\xi\cap (J_q\times J_1)$ with $q\geq3$.  Then $R_q\sub\xi$.
\end{lemma}

\pf Again, by Lemma \ref{lem:pq_Jn}, it suffices to show that $R_1\sub\xi$.  Write $\al=\partI{a_1}{a_q}{A_1}{A_r}{b_1}{b_q}{B_1}{B_r}$, with $a_1<\cdots<a_q$, and define
\[
\ga_1=\partI1q{q+1,q+2}{n-1,n}{a_1}{a_q}{A_1}{A_r}
\AND
\ga_2=\partI{b_1}{b_q}{B_1}{B_r}1q{q+1,q+2}{n-1,n}.
\]
For simplicity, write $\ga=\ga_1\be\ga_2$.  
Then $\ve_q=\ga_1\al\ga_2\mathrel\xi\ga_1\be\ga_2=\ga$.  But $\ga\in\ve_q\J_n\ve_q$ and ${\rank(\ga)=1}$.  
For $\si\in\ve_q\J_n\ve_q$, let $\si^\vee\in\J_q$ be the Jones element of degree $q$ obtained from $\si$ by removing the blocks ${\{q+1,q+2\},\ldots,\{n-1,n\}}$ and $\{(q+1)',(q+2)'\},\ldots,\{(n-1)',n'\}$.  
In $\J_q$, note that $(\ve_q^\vee,\ga^\vee)=(\id_q,\ga^\vee)$, with ${\rank(\ga^\vee)=1}$, so by Lemma \ref{lem:1z_Jn}, this pair generates the universal congruence on $\J_q$.  Since $(\ve_q,\ga)\in\xi$, it follows that all elements of $\ve_q\J_n\ve_q$ are $\xi$-related.  In particular, choosing any pair of elements $\de_3,\de_4\in\ve_q\J_n\ve_q$ such that $\rank(\de_3)=\rank(\de_4)=1$, $\ker(\de_3)\not=\ker(\de_4)$ and $\coker(\de_3)\not=\coker(\de_4)$, Proposition \ref{prop:R01_Jn} gives $R_1=\cg{\de_3}{\de_4}\sub\xi$.  \epf

Following \cite{NS1978}, we say that an element $a$ of a regular $*$-semigroup $S$ is a \emph{projection} if $a^2=a=a^*$.  We denote the set of all projections of $S$ by $\Proj(S)$.  It is well known that ${\Proj(S)=\set{aa^*}{a\in S}=\set{a^*a}{a\in S}}$, and that for any $a\in \Proj(S)$ and $x\in S$, $x^*ax\in \Proj(S)$.
By \cite[Lemma 4]{EF2012}, $\al\in\J_n$ is a projection if and only if it has the form $\al=\partI{a_1}{a_q}{A_1}{A_r}{a_1}{a_q}{A_1}{A_r}$.  
We do not assume $n$ is odd for the statement or proof of the next result.

\begin{lemma}\label{lem:proj1_Jn}
Let $n\geq3$ be arbitrary.  Suppose $\al\in \Proj(\J_n)$ and $3\leq \rank(\al)\leq n-2$.  Then there exists $\be\in \Proj(\J_n)$ with  $\rank(\al\be)<\rank(\al)=\rank(\be)$.
\end{lemma}

\pf Write $\al=\partI{a_1}{a_q}{A_1}{A_r}{a_1}{a_q}{A_1}{A_r}$, where $a_1<\cdots<a_q$.  There are three possibilities:
\begin{itemize}
\begin{multicols}{3}
\item[(a)] $a_1=1$ and $a_2=2$,
\item[(b)] $a_1=1$ and $a_2>2$, or
\item[(c)] $a_1>1$.
\end{multicols}
\eit
In case (a), we put $\be=\Big( 
{ \scriptsize \renewcommand*{\arraystretch}{1}
\begin{array} {\c|\c|\c|\c|\c|\c|\cend}
3 \:&\: \cdots \:&\: q+2 \:&\: 1,2 \:&\: q+3,q+4 \:&\: \cdots \:&\: n-1,n  \\ \cline{4-7}
3 \:&\: \cdots \:&\: q+2 \:&\: 1,2 \:&\: q+3,q+4 \:&\: \cdots \:&\: n-1,n
\rule[0mm]{0mm}{2.7mm}
\end{array} 
}
\hspace{-1.5 truemm} \Big)$.  In case (b), we assume without loss of generality that $A_1=\{2,x\}$, and define $\be=\Big( 
{ \scriptsize \renewcommand*{\arraystretch}{1}
\begin{array} {\c|\c|\c|\c|\c|\c|\c|\c|\c|\cend}
1 \:&\: 2 \:&\: x \:&\: a_4 \:&\: \cdots \:&\: a_q \:&\: a_2,a_3 \:&\: A_2 \:&\: \cdots \:&\: A_r  \\ \cline{7-10}
1 \:&\: 2 \:&\: x \:&\: a_4 \:&\: \cdots \:&\: a_q \:&\: a_2,a_3 \:&\: A_2 \:&\: \cdots \:&\: A_r
\rule[0mm]{0mm}{2.7mm}
\end{array} 
}
\hspace{-1.5 truemm} \Big)$.  In case (c), we assume that $A_1=\{1,x\}$, and define $\be=\Big( 
{ \scriptsize \renewcommand*{\arraystretch}{1}
\begin{array} {\c|\c|\c|\c|\c|\c|\c|\c|\cend}
1 \:&\: x \:&\: a_3 \:&\: \cdots \:&\: a_q \:&\: a_1,a_2 \:&\: A_2 \:&\: \cdots \:&\: A_r  \\ \cline{6-9}
1 \:&\: x \:&\: a_3 \:&\: \cdots \:&\: a_q \:&\: a_1,a_2 \:&\: A_2 \:&\: \cdots \:&\: A_r
\rule[0mm]{0mm}{2.7mm}
\end{array} 
}
\hspace{-1.5 truemm} \Big)$. \epf

\begin{rem}\label{rem:proj1_Jn}
Of course the previous result does not hold for $\rank(\al)\leq1$, since the set of Jones elements of rank at most $1$ is closed under multiplication.  It does not hold for $\rank(\al)=2$ either, although it \emph{almost} does.  We do not need to know this, but it can be proved (see Remark \ref{rem:proj2_Jn}) that the only exception is $\al=\partV1n{2,3}{n-2,n-1}1n{2,3}{n-2,n-1}$: for now, it is easy to check that $\rank(\al\be)=2$ for any projection $\be\in\Proj(\J_n)$ with $\rank(\be)=2$.
\end{rem}

Consider a projection $\al=\partI{a_1}{a_q}{A_1}{A_r}{a_1}{a_q}{A_1}{A_r}\in \Proj(\J_n)$, where $a_1<\cdots<a_q$.  The set $\al\J_n\al$, consisting of all Jones elements containing the blocks $A_i$ and $A_i'$ for each $1\leq i\leq r$, is a subsemigroup of $\J_n$ isomorphic to~$\J_q$; the identity element of $\al\J_n\al$ is $\al$.  We have already made use of this fact in the special case that~$\al=\ve_q$.

\begin{lemma}[cf.~Lemma \ref{lem:1q_Bn}]
\label{lem:1q_Jn}
Let $n\geq3$ be odd.  Suppose $\xi\in\Cong(\J_n)$ and that $(\1,\al)\in\xi$ for some $\al\in\J_n$ with $\rank(\al)<n$.  Then $\xi=R_n$.
\end{lemma}

\pf We proceed by induction on $n$.  The result is true for $n=3$ by Lemma
    \ref{lem:1z_Jn}, so suppose $n\geq5$.  Write $q=\rank(\al)$.  If
    $q=1$, then we are also done by Lemma \ref{lem:1z_Jn}, so suppose
    $q\geq3$.  First note that $\1 \mathrel\xi \al = \al\al^*\al \mathrel\xi \al\al^*\;\!\1 = \al\al^*$,
so that $\1$ is $\xi$-related to a projection of rank $q$.  As such, from this point on, we may assume without loss of generality that $\al$ is itself a projection (of rank $q\geq3$).  

By Lemma \ref{lem:proj1_Jn}, we may fix a projection $\be\in \Proj(\J_n)$ with $\rank(\be)=q$ and $\rank(\al\be)<q$.  Write $\be=\partI{a_1}{a_q}{A_1}{A_r}{a_1}{a_q}{A_1}{A_r}$, where $a_1<\cdots<a_r$.  As noted above, $\be\J_n\be$ is a subsemigroup of $\J_n$ isomorphic to~$\J_q$, and has $\be$ as its identity element.  For $\si\in\be\J_n\be$, we write $\si^\vee$ for the element of~$\J_q$ obtained from~$\si$ by removing the blocks $A_i,A_i'$ ($1\leq i\leq r$), and renaming the elements of the remaining blocks via the bijection $\dom(\be)\to\bq:a_i\mt i$.  So $\be\J_n\be\to\J_q:\si\mt\si^\vee$ is an isomorphism.

Now, $\be=\be\be=\be\;\!\1\;\!\be \mathrel\xi \be\al\be$, and we note that $\be,\be\al\be\in\be\J_n\be$ and $\rank(\be\al\be)\leq\rank(\al\be)<q=\rank(\be)$.  It follows that $\be^\vee=\id_q$ and $(\be\al\be)^\vee\in\J_q$ with ${\rank((\be\al\be)^\vee)<q}$.  
By induction, it follows that the pair $(\be^\vee,(\be\al\be)^\vee)$ generates the universal congruence on $\J_q$.  Since $(\be,\be\al\be)\in\xi$, as noted above, it follows that all elements of $\be\J_n\be$ are $\xi$-related.  In particular, $(\be,\be\ve_1\be)\in\xi$.  Since $\rank(\be)=q\geq3$ and $\rank(\be\ve_1\be)=1$, Lemma \ref{lem:qz_Jn} gives $R_q\sub\xi$.  Together with the fact that $(\1,\al)\in\xi$ and $\rank(\al)=q$, Lemma \ref{lem:pq_Jn} then gives $R_n\sub\xi$, whence $\xi=R_n$. \epf

The proof of the next result is almost identical to that of Lemma \ref{lem:qz_Jn}, but relies on Lemma \ref{lem:1q_Jn} instead of Lemma \ref{lem:1z_Jn}.

\begin{lemma}[cf.~Lemma \ref{lem:qp_Bn}]
\label{lem:pq2_Jn}
Let $n\geq3$ be odd.  Suppose $\xi\in\Cong(\J_n)$ and that there exists $(\al,\be)\in\xi$ with $q=\rank(\al)\geq3$ and $\rank(\be)<q$.  Then $R_q\sub\xi$. \epfres
\end{lemma}

For the proof of the next lemma, if $a$ and $b$ are non-negative integers, and if $\al\in\J_a$ and $\be\in\J_b$, we write $\al\oplus\be$ for the element of $\J_{a+b}$ obtained by placing $\be$ to the right of $\al$.  Formally, we rename elements of each block of $\be\in\J_b$ via the bijection $[1,b]\to[a+1,a+b]:i\mt a+i$, and then $\al\oplus\be$ consists of all these modified blocks plus the (unmodified) blocks of $\al$.  Note that, by convention, we consider $\J_0$ to consist of a single element: namely, the \emph{empty partition} $\emptyset$.  If $a$ is any non-negative integer and $\al\in\J_a$, then $\al\oplus\emptyset=\emptyset\oplus\al=\al$.  An extension of this operation has been used to give \emph{diagram categories} such as the Brauer and Temperley-Lieb categories \emph{(strict) monoidal} structures; see for example \cite{LZ2015,Martin2008}.  The $\oplus$ operation was also used to classify and enumerate the idempotents in the partition and (partial) Brauer monoids in \cite{DEEFHHL1}.  

For the statement of the next result, we do not assume $n$ is odd, but the proof makes use of the fact that it holds for even $n$, and indicates why this is the case.

\begin{lemma}\label{lem:proj2_Jn}
Let $n\geq3$ be arbitrary.  Suppose $\al,\be\in\Proj(\J_n)$ are such that $\al\not=\be$ and $2\leq\rank(\al)=\rank(\be)\leq n-2$.  Then, renaming $\al,\be$ if necessary, there exists $\ga\in\Proj(\J_n)$ such that $\rank(\al\ga)<\rank(\be\ga)=\rank(\al)=\rank(\ga)$.
\end{lemma}

\pf The result is true for even $n$, because of Lemma~\ref{lemma-aa1b:PPn} (cf.~Lemma \ref{lemma-aa1b}), where the constructed element $\gamma$ is indeed a projection, and the isomorphism ${\PP_m\to\J_{2m}:\al\mt\alt}$, keeping in mind the fact that $\rank(\alt)=2\rank(\al)$ for any $\al\in\PP_m$.  So we assume $n$ is odd for the rest of the proof.  This also forces $\rank(\al)=\rank(\be)\geq3$ to be odd. Write $\al=\partI{a_1}{a_q}{A_1}{A_r}{a_1}{a_q}{A_1}{A_r}$ and $\be=\partI{b_1}{b_q}{B_1}{B_r}{b_1}{b_q}{B_1}{B_r}$, where $a_1<\cdots<a_q$ and $b_1<\cdots<b_q$.

\bigskip\noindent {\bf Case 1.}  Suppose first that $a_1=b_1$ and $a_2=b_2$.  We may write $\al=\al_1\oplus\id_1\oplus\al_2\oplus\id_1\oplus\al_3$ and $\be=\be_1\oplus\id_1\oplus\be_2\oplus\id_1\oplus\be_3$, where $\al_1,\be_1\in\Proj(\J_{a_1-1})$, $\al_2,\be_2\in\Proj(\J_{a_2-a_1-1})$ and $\al_3,\be_3\in\Proj(\J_{n-a_2})$.
Note that $\rank(\al_1)=\rank(\be_1)=\rank(\al_2)=\rank(\be_2)=0$ and $\rank(\al_3)=\rank(\be_3)=q-2\geq1$.

\bigskip\noindent {\bf Subcase 1.1.}  Suppose $\al_1\not=\be_1$, and write $\al_4=\al_1\oplus\id_1\oplus\al_2\oplus\id_1$ and $\be_4=\be_1\oplus\id_1\oplus\be_2\oplus\id_1$.  Then $\al_4,\be_4\in\Proj(\J_{a_2})$, $\rank(\al_4)=\rank(\be_4)=2$ and $\al_4\not=\be_4$.  Since $a_2$ is even, and since the lemma holds for even $n$, there exists $\de\in\Proj(\J_{a_2})$ such that, renaming if necessary, $\rank(\al_4\de)<\rank(\be_4\de)=\rank(\de)=2$.  It is then easy to check that $\ga=\de\oplus\be_3$ has the desired properties.

\bigskip\noindent {\bf Subcase 1.2.}  Suppose $\al_1=\be_1$, and write $\al_5=\al_2\oplus\id_1\oplus\al_3$ and $\be_5=\be_2\oplus\id_1\oplus\be_3$.  This time we have $\al_5,\be_5\in\Proj(\J_{n-a_1})$, with $n-a_1$ even, and with all the hypotheses of the lemma satisfied by $\al_5,\be_5$.  The proof concludes in similar fashion to Subcase 1.1: we find a suitable $\de\in\Proj(\J_{n-a_1})$ and put $\ga=\be_1\oplus\id_1\oplus\de$.

\bigskip\noindent {\bf Case 2.}  Next suppose $a_1=b_1$ but $a_2\not=b_2$.  We may write $\al=\al_1\oplus\id_1\oplus\al_2$ and $\be=\be_1\oplus\id_1\oplus\be_2$, where $\al_1,\be_1\in\Proj(\J_{a_1-1})$ and $\al_2,\be_2\in\Proj(\J_{n-a_1})$.  Since $a_2\not=b_2$, we know that $\al_2\not=\be_2$.  So, by the even version of the lemma, and renaming if necessary, there exists $\de\in\Proj(\J_{n-a_1})$ such that $\rank(\al_2\de)<\rank(\be_2\de)=\rank(\de)=q-1$, and we put $\ga=\be_1\oplus\id_1\oplus\de$.

\bigskip\noindent {\bf Case 3.} Next suppose $a_1\not=b_1$ but $a_2=b_2$.  We may write $\al=\al_1\oplus\id_1\oplus\al_2$ and $\be=\be_1\oplus\id_1\oplus\be_2$, where $\al_1,\be_1\in\Proj(\J_{a_2-1})$, $\al_2,\be_2\in\Proj(\J_{n-a_2})$.  Since $a_1\not=b_1$, $\al_1\oplus\id_1$ and $\be_1\oplus\id_1$ are distinct projections of rank $2$ from $\J_{a_2}$, and the proof concludes in similar fashion to Case 2.

\bigskip\noindent {\bf Case 4.}  Finally, suppose $a_1\not=b_1$ and $a_2\not=b_2$.  Without loss of generality, we may assume that $a_1<b_1$.  
If $b_2<a_2$, then we let $\ga$ be any projection with $\dom(\ga)=\dom(\al)$ and containing the block $\{b_1,b_2\}$.  
If $a_2<b_1$, then we let $\ga$ be any projection with $\dom(\ga)=\dom(\be)$ and containing the block $\{a_1,a_2\}$.  
Since $a_1<b_1$, $a_2\not=b_2$ and $a_2\not=b_1$ (the latter because $b_1,a_2$ have opposite parities), the only remaining possibility is $a_1<b_1<a_2<b_2$; we assume this is the case for the remainder of the proof.
For each $i\in\{1,2\}$, let the $\coker(\be)$-class of $a_i$ be $\{a_i,x_i\}$ and the $\coker(\al)$-class of $b_i$ be $\{b_i,y_i\}$.  If $x_1>a_1$, then we let $\ga$ be any projection with $\dom(\ga)=\{a_1,x_1,a_3,\ldots,a_q\}$.  If $x_2<a_2$, then we let $\ga$ be any projection with $\dom(\ga)=\{x_2,a_2,a_3,\ldots,a_q\}$.  The cases in which $y_1>b_1$ or $y_2<b_2$ are treated in similar fashion.  So suppose instead that $x_1<a_1$, $x_2>a_2$,  $y_1<b_1$ and $y_2>b_2$.  Since also $y_1>a_1$ (by planarity of $\al$) and $x_2<b_2$ (by planarity of $\be$), it follows that
\[
x_1<a_1<y_1<b_1<a_2<x_2<b_2<y_2.
\]
Note that if $\rank(\al\be)<q$, then we may simply take $\ga=\be$.  So we assume that $\rank(\al\be)=q$.  In particular, $\dom(\al\be)=\dom(\al)$ and $\codom(\al\be)=\codom(\be)$.  Consider the product graph $\Pi(\al,\be)$.  Since $a_1=\min(\dom(\al\be))$ and $b_1=\min(\codom(\al\be))$, there is a path $P$ in $\Pi(\al,\be)$ from $a_1$ to $b_1'$.  By the above discussion, we know the first two and last two edges the path $P$ uses:
\[
a_1 \lar\al a_1'' \lar\be x_1'' \lar\al \cdots \lar\be y_1'' \lar\al b_1'' \lar\be b_1',
\]
with all edges other the first and last connecting vertices from $\bn''$.  (Here, for example, we have written~${u\lar\al v}$ to indicate that the edge $\{u,v\}$ from $\Pi(\al,\be)$ comes from $\al$.)  Note that for any edge $i''\lar\al j''$ or $k''\lar\be l''$ in $P$, $i,l$ are even, and $j,k$ odd.
Let $w$ be the minimal element of $\bn$ such that the path $P$ visits vertex $w''$.  So $w\leq x_1<a_1$.  We claim that $w$ is odd.  Indeed, since $w<a_1<b_1$, the path $P$ contains edges
\begin{itemize}\begin{multicols}{2}
\item[(i)] $i''\lar\al w''\lar\be j''$ for some $i,j$, \ or 
\item[(ii)] $i''\lar\be w''\lar\al j''$ for some $i,j$.  
\end{multicols}\eit
Keeping in mind the above note concerning the parity of endpoints of edges in $P$, to prove the claim that $w$ is odd, it suffices to show that (i) is the case.  
In order to do this, consider the product graph $\Pi(\al,\be)$ embedded in the plane $\Rtwo$ as follows.  For each $i\in\bn$, we identify the vertices $i,i'',i'$ with the points $(i,n),(i,0),(i,-n)$ respectively.  Define the rectangle $\Rect=\bigset{(x,y)\in\Rtwo}{1\leq x\leq n,\ -n\leq y\leq n}$ to be the convex hull of these~$3n$ points, and add edges from $\al$ and $\be$ such that:
\bit
\item each transverse edge of $\al$ or $\be$ is drawn as a vertical line segment connecting its endpoints, and 
\item each non-transverse edge is drawn as a semicircle within $\Rect$, above or below the line $y=0$ depending on whether the edge belongs to $\al$ or $\be$, respectively.
\eit
So $\Pi(\al,\be)$, embedded in $\Rtwo$ as above, contains a smooth planar curve $\C$, induced by the path $P$, connecting~$a_1$ (on the upper edge of the rectangle $\Rect$) to $b_1'$ (on the lower edge of $\Rect$).  Let $\C_1$ be the portion of this curve joining $a_1$ to $w''$, and $\C_2$ the portion joining $w''$ to $b_1'$.  Note that $\C$ is contained in the smaller rectangle~${\Rect_1=\bigset{(x,y)\in\Rtwo}{w\leq x\leq n,\ -n\leq y\leq n}}$.  A schematic diagram of all this is given in Figure~\ref{fig:curve}.  

Since~$\C_1$ is contained in $\Rect_1$, and joins $a_1$ (on the upper side of $\Rect_1$) to $w''$ (on the left side of $\Rect_1$), $\C_1$ splits~$\Rect_1$ into two regions: $\Rect_1^+$, containing the vertex $w$ (the upper left corner of $\Rect_1$), and $\Rect_1^-$, containing the vertex $n'$ (the lower right corner of $\Rect_1$).  (Note that $P$ does not visit the point $n''$, since $\Pi(a,b)$ also contains a path from $a_2$ to $b_2'$.)  Now, $\C_2$ is contained in $\Rect_1$ and, apart from its initial vertex~$w''$, never intersects~$\C_1$, so $\C_2\sm\{w''\}$ is contained in one of $\Rect_1^+$ or $\Rect_1^-$.  But $\C_2$ connects $w''$ to $b_1'$, with the latter point belonging to $\Rect_1^-$.  So it follows that $\C_2\sm\{w''\}$ is contained in $\Rect_1^-$.  
Recall that $\C_1$ and $\C_2$ are unions of edges from the product graph $\Pi(\al,\be)$.  Let $E_1$ be the last such edge in $\C_1$, and $E_2$ the first such edge in~$\C_2$.  So~$E_1$ is of the form $i''\to w''$, and $E_2$ is of the form $w''\to j''$, for some $i,j\in\bn$.  If $E_1$ belonged to $\be$, then~$E_2$ would belong to $\al$; but then at least a segment of $E_2$ (for $w<x<w+1/2$, say) would lie above the corresponding segment of $E_1$, so that $\C_2$ contained points in $\Rect_1^+$, a contradiction.  Hence, $E_1$ belongs to~$\al$.  This completes the proof that (i) holds and, hence, that $w$ is odd.

Since the first edge of the path $P$ is $a_1\lar\al a_1''$, the edge immediately preceding those listed above in~(i) must be of the form $k''\lar\be i''$ for some $k$.  To summarise, we know that $\Pi(\al,\be)$ contains the edges
\[
k''\lar\be i''\lar\al w''\lar\be j'' \qquad\text{for some $i,j,k>w$.}
\]
In particular, $\{k,i\}$ and $\{w,j\}$ are $\coker(\be)$-classes, and $\{i,w\}$ is a $\coker(\al)$-class.  By planarity, and keeping in mind that $w<i,j,k$, we must be in one of the following four cases:
\begin{itemize}\begin{multicols}{4}
\item[(a)] $w<k<i<j$,
\item[(b)] $w<i<k<j$,
\item[(c)] $w<j<k<i$, or
\item[(d)] $w<j<i<k$.
\end{multicols}\eit
Since $w<a_1<b_1$, planarity also implies that $i<a_1$, $j<b_1$ and $k<b_1$.  In cases (a) and (b), we define $\ga$ to be any projection with domain $\{w,i,b_3,\ldots,b_q\}$ and containing the block $\{j,b_1\}$.  In cases (c) and (d), we define $\ga$ to be any projection with domain $\{w,j,a_3,\ldots,a_q\}$ and containing the block $\{i,a_1\}$.  This completes the proof. \epf

\begin{figure}[ht]
\begin{center}
\scalebox{.91}{
\begin{tikzpicture}[scale=.5]
\fill[gray!20] (0,0)--(30,0)--(30,20)--(0,20)--(0,0);
\fill[green!15]
(14,20)--(14,10) 
\arcdl3 
\arcul1 
\arcdr6
\arcur{12}
\arcdl4
\arcur1 
\arcdr1 
\arcul{9}
\arcdl{9} 
\arcul4 
--(4,20)--(14,20)
;
\fill[orange!15]
(14,20)--(14,10) 
\arcdl3 
\arcul1 
\arcdr6
\arcur{12}
\arcdl4
\arcur1 
\arcdr1 
\arcul{9}
\arcdl{9} 
\arcul4 
--(4,0)--(30,0)--(30,20)--(14,20)
;
\draw (0,0)--(30,0)--(30,20)--(0,20)--(0,0) (0,10)--(30,10) (4,0)--(4,20);
\draw[red, ultra thick] (14,20)--(14,10) 
\arcdl3 
\arcul1 
\arcdr6
\arcur{12}
\arcdl4
\arcur1 
\arcdr1 
\arcul{9}
\arcdl{9} 
\arcul4 
;
\draw[blue, ultra thick](4,10)
\arcdr2
\arcur1 
\arcdr{14} 
\arcul1
\arcdl1
\arcur3
--(22,0)
;
\foreach \x in {0,4,7,10,14,17,20,22,25,28,30} {\mv{\x}}
\foreach \x in {6,8,11,16,19,21,24,26} {\mvw{\x}}
\foreach \x in {0,4,14,30} {\hmv{\x}}
\foreach \x in {0,4,22,30} {\lmv{\x}}
\node[red] () at (14,18) [right] {${\C_1}$};
\node[blue] () at (22,2) [right] {${\C_2}$};
\node[black] () at (5,19) {${\Rect_1^+}$};
\node[black] () at (29,1) {${\Rect_1^-}$};
\node[black] () at (2,1) {${\Rect\sm\Rect_1}$};
\node () at (0,20.2) [above] {\small $1$};
\node () at (4,20.2) [above] {\small $w$};
\node () at (14,20.2) [above] {\small ${}_{\phantom{1}}a_1$};
\node () at (30,20.2) [above] {\small $n$};
\node () at (0,-.2) [below] {\small $\phantom{'}1'$};
\node () at (4,-.2) [below] {\small $\phantom{'}w'$};
\node () at (22,-.2) [below] {\small ${}_{\phantom{1}}b_1'$};
\node () at (30,-.2) [below] {\small $\phantom{'}n'$};
\node () at (.2,10) [above left] {\small $\phantom{''}1''$};
\node () at (4.2,10) [above left] {\small $\phantom{''}w''$};
\node () at (6.3,9.9) [above left] {\small $\phantom{''}j''$};
\node () at (8-.4,10-.1) [above right] {\small $\phantom{''}i''$};
\node () at (11-.4,10-.1) [above right] {\small $\phantom{''}x_1''$};
\node () at (14-.4,10-.1) [above right] {\small $\phantom{''}a_1''$};
\node () at (17-.4,10-.1) [above right] {\small $\phantom{''}k''$};
\node () at (19+.2,10-.1) [above left] {\small $\phantom{''}y_1''$};
\node () at (22-.4,10-.1) [above right] {\small $\phantom{''}b_1''$};
\node () at (30-.4,10-.1) [above right] {\small $\phantom{''}n''$};
\draw[|-|] (-2,20)--(-2,10);
\draw[-|] (-2,10)--(-2,0);
\draw(-2,5)node[left]{$\be$};
\draw(-2,15)node[left]{$\al$};
\draw(32,5)node[right]{$\phantom{\be}$};
\draw(32,15)node[right]{$\phantom{\al}$};
\end{tikzpicture}
}
\end{center}
\vspace{-5mm}
\caption{The curve $\C=\C_1\cup\C_2$ from Case 4 of the proof of Lemma \ref{lem:proj2_Jn}.  The curves $\C_1$ and $\C_2$ are drawn red and blue, respectively.  The regions $\Rect_1^+$ and $\Rect_1^-$ are shaded green and orange, respectively, and the region $\Rect\sm\Rect_1$ is shaded gray.  Odd-labelled vertices are drawn black, and even-labelled vertices white. }
\label{fig:curve}
\end{figure}
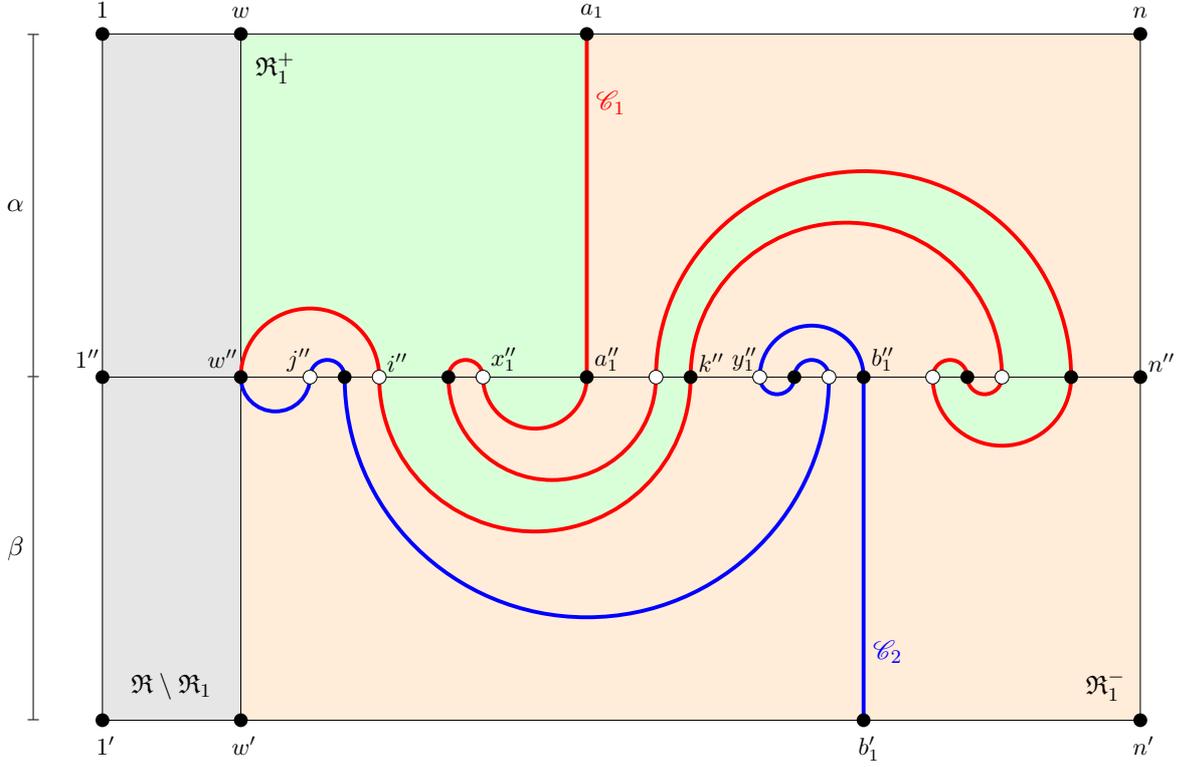

\begin{rem}
In Figure \ref{fig:curve}, we have $w<j<i<k$, so we are in case (d), as enumerated at the end of the previous proof.  The reader might like to construct the projections $\al,\be,\ga$ in this case.  
\end{rem}

\begin{rem}\label{rem:proj2_Jn}
We also note that Lemma \ref{lem:proj2_Jn} has another interpretation.
Fix some $2\leq q\leq n-2$, and fix an ordering on the set $\{\al_1,\ldots,\al_k\}$ of all projections from the $\gJ$-class $J_q$.  (By the proof of \cite[Theorem~9.5]{EG2017}, we have $k=\frac{q+1}{n+1}\binom{n+1}{(n-q)/2}$.)  Define a $k\times k$ matrix $M=(m_{ij})$ with entries in $\{0,1\}$, where $m_{ij}=1$ if and only if 
the unique element $\alpha\in J_q$ satisfying $\al_i\R\al\L\al_j$ is an idempotent.
(We note that $M$ is the \emph{sandwich matrix} for the representation of the principal factor of $\overline{J}_q$ as a Rees matrix semigroup; see \cite[Section 3.2]{Howie}.)
By \cite[Lemma 2.3(ii)]{emojoka}, we also have $m_{ij}=1$ if and only if $\rank(\al_i\al_j)=q$.  Thus, Lemma \ref{lem:proj2_Jn} says that no two rows of the matrix $M$ are equal.  Note that this is also vacuously true for $q=n$, but not true for~$q\leq1$, as all Jones (or even Brauer) elements of minimal rank are idempotents.  Similarly, Lemma \ref{lem:proj1_Jn} may be interpreted as saying that if $3\leq q\leq n-2$, then every row of the matrix $M$ has at least one entry equal to $0$.  Remark \ref{rem:proj1_Jn} asserts that when $q=2$, the matrix $M$ has exactly one row with no entries equal to $0$.  This can now be seen to be true.  Indeed, it is easy to check, when $q=2$, that one row of $M$ has no entries equal to $0$ (the exact row was specified in Remark \ref{rem:proj1_Jn}).  The fact that this row is unique follows from the reinterpreted version of Lemma \ref{lem:proj2_Jn}.
\end{rem}

\begin{lemma}\label{lem:qq_Jn}
Suppose $\xi\in\Cong(\J_n)$ and $(\al,\be)\in (\xi\sm\Delta)\cap (J_q\times J_q)$ with $q\geq3$.  Then $R_q\sub\xi$.
\end{lemma}

\pf Since $\al\not=\be$, we have either $(\al,\be)\not\in{\R}$ or $(\al,\be)\not\in{\L}$.  By symmetry, we may assume the latter is the case.  Note that $\al^*\al$ and $\be^*\be$ are distinct projections of rank $q$, since $\al^*\al\L\al$ and $\be^*\be\L\be$.  So, by Lemma \ref{lem:proj2_Jn}, there exists $\ga\in\Proj(\J_n)$ such that, renaming $\al,\be$ if necessary, $\rank(\al^*\al\ga)<q$ and $\rank(\be^*\be\ga)=q$.  Now, $\al\ga=\al(\al^*\al\ga)$, so it follows that $\al\ga\L\al^*\al\ga$.  Consequently, $\al\ga\gJ\al^*\al\ga$, whence ${\rank(\al\ga)=\rank(\al^*\al\ga)<q}$.  Similarly, $\rank(\be\ga)=\rank(\be^*\be\ga)=q$.  Since $(\al\ga,\be\ga)\in\xi$, Lemma \ref{lem:pq2_Jn} then gives $R_q\sub\xi$. \epf

Armed with the previous results, we may now characterise the generating pairs for the Rees congruences~$R_q$, for $q\geq3$.

\begin{prop}\label{prop:Rq_Jn}
Let $n$ be odd, and suppose $3\leq q\leq n$ is odd.  Then $R_q$ is a principal congruence on $\J_n$.  Moreover, if $\al,\be\in\J_n$, then $R_q=\cg\al\be\iff(\al,\be)\in R_q\sm R_{q-2}$.
\end{prop}

\pf Let $(\al,\be)\in R_q\sm R_{q-2}$.  The proof will be complete if we can show that $R_q\sub\cg\al\be$.  Renaming $\al,\be$ if necessary, we may assume that $\rank(\al)=q$.  But then $R_q\sub\cg\al\be$ follows from Lemma \ref{lem:pq2_Jn} if $\rank(\be)<q$, or from Lemma \ref{lem:qq_Jn} if $\rank(\be)=q$. \epf

We now have all we need to conclude the proof.

\begin{proof}[{\bf Proof of Theorem \ref{thm:cong-Jn} for \boldmath{$n$} odd.}]
As usual, the proof is completed by verifying that in describing the generating pairs of the congruences $R_q$
(Propositions \ref{prop:R01_Jn} and \ref{prop:Rq_Jn}) and $\lambda_1,\rho_1$
(Proposition \ref{prop:rlJodd}), we have covered all possible pairs of distinct elements of $\J_n$.
\end{proof}

\subsubsection*{Acknowledgements}

This work was initiated during a visit of the first author to the other three in 2016; he thanks the University of St Andrews for its hospitality during his stay.
We thank the referee for a number of helpful suggestions.

\footnotesize
\def\bibspacing{-1.1pt}
\bibliography{biblio}
\bibliographystyle{abbrv}
\end{document}